%% file: ArticleMainFile.tex
\documentclass[11pt,reqno]{amsart}
\usepackage[left=1.3in, right=1.2in, top=1in, bottom=1in, includefoot]{geometry}

\usepackage{graphicx}
\usepackage{mathrsfs} 
\usepackage{amssymb}   
\usepackage{amsthm}  
\usepackage{bbm}
\usepackage[shortlabels]{enumitem}
\usepackage{xcolor}
\usepackage{soul}

\newtheorem{thm}{Theorem}[section]
\newtheorem{prop}[thm]{Proposition}
\newtheorem{cor}[thm]{Corollary}
\newtheorem{Def}[thm]{Definition}
\newtheorem{lemma}[thm]{Lemma}
\newtheorem{remark}[thm]{Remark}

\newcommand{\ve}{{\varepsilon}}
\newcommand{\N}{{\mathbb{N}}}
\newcommand{\R}{{\mathbb{R}}}
\numberwithin{equation}{section}
\title{Long time well-posedness of  Whitham-Boussinesq  systems}

\author[M. Oen Paulsen]{Martin Oen Paulsen}
\address{Department of Mathematics\\ University of Bergen\\ Postbox 7800\\ 5020 Bergen\\ Norway}
\email{Martin.Paulsen@UiB.no}

\date{\today}

\begin{document}

	\include{MainFile/Paper}

	%
	%
	%
	\include{References/ref}

\end{document}

%% file: MainFile/Paper.tex
\maketitle

\begin{abstract}
	
	\noindent 
	Consideration is given to three different full dispersion Boussinesq systems arising as asymptotic models in the bi-directional propagation of weakly nonlinear surface waves in shallow water. We prove that, under a non-cavitation condition on the initial data, these three systems are well-posed on a time scale of order $\mathcal{O}(\frac1\ve)$, where $\ve$  is a small parameter measuring the weak non-linearity of the waves. This result seems new for one of these systems, even for short time. The two other systems involve surface tension, and for one of them, the non-cavitation condition has to be sharpened when the surface tension is small. The proof relies on suitable symmetrizers and the classical theory of hyperbolic systems. However, we have to track the small parameters carefully in the commutator estimates to get the long time well-posedness.  
	
	Finally, combining our results with the recent ones of Emerald provide a full justification of these systems as water wave models in a larger range of regimes than the classical $(a,b,c,d)$-Boussinesq systems. 
	
\end{abstract}

\section{Introduction}

\subsection{Full dispersion models} 
The Korteweg-de Vries (KdV) equation is an asymptotic model for the unidirectional propagation of small
amplitude, long waves on the surface of an ideal fluid of constant depth. It was introduced in \cite{boussinesq1877,KdV1895} to model the propagation of solitary waves in shallow water with a wide range of applications both mathematically and physically.  However, its dispersion is too strong in high frequencies when compared to the full water wave system. In particular, the KdV equation does not feature wave breaking or peaking waves. To overcome these shortcomings, Whitham introduced in \cite{Whitham1967} an equation with an improved dispersion relation.  He replaced the KdV dispersion with the exact dispersion of the linearized water wave system obtaining the equation
\begin{equation}\label{1-D disp eq}
	\partial_t \zeta + \sqrt{\mathcal{K}_{\mu}}(D) \partial_x \zeta + \ve \zeta \partial_x \zeta = 0,
\end{equation}
for $(x,t) \in \R \times \R^+$, where the function $ \zeta(x,t) \in \R$ denotes the surface elevation and  the operator $\sqrt{\mathcal{K}_{\mu}}(D)$ is the square root of the Fourier multiplier $\mathcal{K}_{\mu}(D)$ defined in frequency by
\begin{equation} \label{mult K}
	K_{\mu}(\xi) = \frac{tanh(\sqrt{\mu}|\xi|)}{\sqrt{\mu}|\xi|}\left(1+\beta\mu\xi^2\right).
\end{equation}
Moreover, $\mu$ and $\ve$ are small parameters related to the level of dispersion and nonlinearity, and $\beta$ is a nonnegative parameter related to the surface tension\footnote{Actually, Whitham introduced the equation formally without the parameters $\mu$ and $\ve$.}$^{,}$\footnote{ He also did not include surface tension, \textit{i.e.} $\beta=0$.}.

Whitham conjectured in \cite{Whitham1967} that equation  \eqref{1-D disp eq} would allow, in addition to the KdV traveling-wave regime, the occurrence of waves of greatest height with a sharp crest as well as the formation of shocks. However, it was not until recently that these phenomena were rigorously proved. We mention among others the existence of periodic waves \cite{ehrnstrom2009traveling}, the existence and stability of traveling waves \cite{ehrnstrom2012existence,arnesen2015existence,stefanov2020small, johnson2020generalized},  the formation of shocks  \cite{Hur17WBreaking,saut2020wave}, Benjamin-Feir instabilities \cite{HenrikCarter,HurJohnson}, the existence of periodic waves of greatest height \cite{ehrnstrom2019whitham}  and solitary waves of greatest height \cite{Truong T 2020}. Note that in the case of surface tension ($\beta>0$), the dynamics appear to be rather different (see \text{e.g.} \cite{KleinLinaresPilodSaut2018} and the references therein).

These results illustrate some mathematical properties uniquely related to an improved dispersion relation, though there are some phenomena that the Whitham equation does not feature due to its unidirectionality. For instance, the Euler equations admit non-modulational instabilities of small-amplitude periodic traveling waves \cite{MacKay1986}, but the unidirectional nature of the Whitham equation is believed to prohibit such instabilities \cite{Deconinck2017}.

Regarding the  two-way propagation of waves at the surface of a fluid and in the long wave regime, Bona, Chen, and Saut derived a three-parameter family of Boussinesq systems \cite{BonaChenSaut2002}
\begin{equation}\label{Boussinesq}
	\begin{cases}
		(1-b\mu\partial_x^2)\partial_t \zeta +  (1+a\mu \partial_x^2)\partial_x v + \ve\partial_x(\zeta v)   = 0 \\
		(1-d\mu\partial_x^2 )\partial_t v + (1+c \mu \partial_x^2)\partial_x \zeta + \ve v \partial_x v    = 0,
	\end{cases}
\end{equation}
%
%
%
%These models approximately describes the two-way propagation of waves in a fluid  of constant depth where the dispersion and nonlinear effects are balanced. 
where $a$, $b$, $c$ and $d$ are real parameters satisfying $a+b+c+d=\frac13$, $\zeta$ is the deviation of the free surface with respect to its rest state, and $v$ approximates the fluid velocity at some height in the fluid domain. Like the KdV equation, the Boussinesq systems are celebrated models for surface waves in coastal oceanography. Analogously to the unidirectional case, one could replace the dispersion by the linearized dispersion of the water wave equations in \eqref{Boussinesq}. These improved dispersion versions are expected to lead to a more \lq\lq accurate\rq\rq \:  description of the full water wave system.  Those systems are commonly referred to as the Whitham-Boussinesq systems or full dispersion Boussinesq systems.

Actually, there are different possibilities of full dispersion Boussinesq systems. This paper will focus on three important ones, linking them to some specific cases of the Boussinesq systems without BBM terms ($b=d=0$). 
To be precise, we introduce the operator $ \mathcal{T}_{\mu}(D)$  corresponding to $\mathcal{K}_{\mu}(D)$ for $\beta=0$, and whose Fourier symbol is defined by
\begin{equation}\label{Op T}
	T_{\mu}(\xi) = \frac{tanh(\sqrt{\mu}|\xi|)}{\sqrt{\mu}|\xi|}.
\end{equation}

First, we consider the system
\begin{equation}\label{full dispersion}%\label{First Wh}
	\begin{cases}
		\partial_t \zeta + \mathcal{K}_{\mu}(D) \partial_x v + \ve\partial_x(\zeta v) = 0 \\
		\partial_t v + \partial_x \zeta + \ve v \partial_x v = 0, 
	\end{cases}
\end{equation}
introduced in \cite{Lannes_Book2013,aceves2013numerical,KalischMoldabayev} without surface tension and in \cite{KleinLinaresPilodSaut2018} with surface tension. Here, as above, $\zeta$ denotes the elevation of the surface around its equilibrium position, while $v$ approximates the fluid velocity at the free surface. In the case zero surface tension, it is proved that \eqref{full dispersion} models solitary waves \cite{nilsson2019solitary} and admit high-frequency (non-modulational) instabilities of small-amplitude periodic traveling waves  \cite{ehrnstrom2019existence}. We also observe that \eqref{full dispersion} is related to \eqref{Boussinesq} by expanding \eqref{mult K} in low frequencies. Indeed, since $K_{\mu}(\xi) \simeq 1+\mu(\beta-\frac13) \xi^2$ by a Taylor expansion we see that \eqref{full dispersion} reduce to \eqref{Boussinesq} with $(a,b,c,d) = (\frac13-\beta,0,0,0)$.

 A second system is obtained by applying the operator \eqref{Op T} to $\partial_{x}\zeta$, which gives
\begin{equation}\label{Whitham Boussinesq}%\label{Second Wh}%\label{Whitham Boussinesq}
	\begin{cases}
		\partial_t \zeta + \partial_x v + \ve\partial_x(\zeta v) = 0 \\
		\partial_t v +  \mathcal{T}_{\mu}(D)\partial_x \zeta + \ve v \partial_x v = 0.
	\end{cases}
\end{equation}
This system was first introduced in \cite{HurPandey2019}, where it is proved that \eqref{Whitham Boussinesq} features Benjamin-Feir (modulational) instabilities. Note that while $\zeta$ plays the same role as for system \eqref{full dispersion}, it is $\mathcal{T}_{\mu}^{-1}(D)v$ which approximates the velocity potential at the free surface in this case. We also observe that the system reduces in the formal limit $\sqrt{\mu}|\xi| \to 0$  to the Boussinesq system \eqref{Boussinesq} with $(a,b,c,d) = (0,0,\frac13,0)$ in low frequencies. 

Finally, we will also consider a full dispersion version of \eqref{Boussinesq} when $\mathcal{T}_{\mu}(D)$ is applied to the nonlinear terms, while $\mathcal{K}_{\mu}(D)$ is applied on the $\partial_x\zeta$. This system reads
\begin{equation}\label{2nd Whitham Boussinesq}
	\begin{cases}
		\partial_t \zeta + \partial_x v + \ve \mathcal{T}_{\mu}(D) \partial_x(\zeta v) = 0 \\
		\partial_t v +  \mathcal{K}_{\mu}(D)\partial_x \zeta + \ve \mathcal{T}_{\mu}(D) (v \partial_x v )= 0.
	\end{cases}
\end{equation}
Here $\zeta$ and $v$ play the same roles as for system \eqref{Whitham Boussinesq}.
It was introduced in \cite{dinvay2017whitham} and has the advantage of being Hamiltonian. Moreover, the existence of solitary waves is proved in  \cite{dinvay2021solitary}. 

\subsection{Full justification}% A fundamental question in the theory of surface water waves is on the strong convergence of a solution to an asymptotic model, Lannes: Ch. 6

%A fundamental issue in the theory of surface water waves is the rigorous derivation and justification of asymptotic models in various regimes. 

%separate point of interest is the rigorous justification of asymptotic models as an approximation of the Euler equations with a free surface.

A fundamental question in the derivation of an asymptotic model is whether its solution converges to the solution of the original physical system.  In particular,  we say that  an asymptotic model is a valid approximation of the Euler equations with a free surface if we can answer the following points in the affirmative \cite{Lannes_Book2013}:
\begin{enumerate}[1.]
	\item  The solutions of the water wave equations exist on the relevant scale $\mathcal{O}(\frac{1}{\ve})$.
	\item The solutions of the asymptotic model exist (at least) on the scale $\mathcal{O}(\frac{1}{\ve})$.

	\item Lastly, we must establish the \textit{consistency} between the asymptotic model and the water wave equations, and then show that the error is of order $\mathcal{O}( \mu\ve t)$ when comparing the two solutions.
	
\end{enumerate}

The first point was proved by Alvarez-Samaniego and Lannes \cite{ alvarez2008large} for surface gravity waves and Ming, Zhang and Zhang \cite{ming2012long} for gravity-capillary waves in the weakly transverse regime, while points $2.$ and $3.$ are specific to the asymptotic model under consideration. For instance, in the case of the Whitham equation, Klein \textit{et al.} \cite{KleinLinaresPilodSaut2018} compared its solution rigorously with those of the KdV equation. In particular, they proved that the difference of two solutions evolving from the same initial datum is bounded by $\mathcal{O}(\ve^2 t)$ for all $0\leq t \lesssim \ve^{-1}$ with $\ve, \mu$ in the KdV-regime:
$$\mathcal{R}_{KdV} =  \{(\ve,\mu ) \: : \: 0\leq \mu \leq 1, \quad \mu = \ve\},$$
which justified the Whitham equation as a water wave model  in this regime by relying on the justification of the KdV equation \cite{Craig1985,Lannes_Book2013}.

On the other hand, due to the improved dispersion relation of \eqref{1-D disp eq}, Emerald \cite{Emerald2021} was able to decouple the parameters $(\ve,\mu)$ and  prove an  error estimate between the Whitham equation and the water wave system with a precision $\mathcal{O}(\mu \ve t)$ for $0\leq t \lesssim \ve^{-1}$ in the shallow water regime:
\begin{equation}\label{Shalow water regime}
	\mathcal{R}_{SW} = \{(\ve,\mu) \: : \: 0\leq \mu \leq 1, \quad 0 \leq \ve \leq 1 \}.
\end{equation}
Consequently, the Witham equation is valid for a larger set of small parameters when compared to the KdV equation.% However, one should note that for $\ve << \mu$, which is the regime we would expect the Whitham equation to feature shocs% Specifically, the result extends the validity to include the case $\mu<<\ve $.  Rart --- litt motsatt...

% the Whitham equation, both are proved to  -.. do classical Klein --- Emerald - explain --- then do boussinesq , and we want the extension. same rythm.\\ 

%\textcolor{red}{put in the right order}
In the case of the Boussinesq systems \eqref{Boussinesq}, consistency was first proved  in \cite{BonaColinLannes2005} for $(\ve, \mu) \in \mathcal{R}_{KdV}$ by relying on intermediate symmetric systems for which the long time well-posedness follows by classical arguments. However the long time well-posedness for the $(a,b,c,d)$ Boussinesq is far from trivial. This result was proved\footnote{In the most dispersive case $(a,b,c,d)=(\frac16,\frac16,0,0)$, the relevant time scale $\mathcal{O}(\varepsilon^{-1})$ is still missing; the best results being on a time scale $\mathcal{O}(\varepsilon^{-\frac23})$ \cite{SautXu2020,SautXuCPDE2020},
(see also \cite{LinaresPilodSaut} on a time scale $\mathcal{O}(\varepsilon^{-\frac12})$ by using dispersive techniques).}
later by Saut, Xu and Wang \cite{SautXu2012,saut2017cauchy}. The proof relies on suitable symmetrizers and hyperbolic theory. 
%Then having the long time existence of the asymptotic model and the full system, one must provide uniform bounds to compare the two solutions. This was established in \cite{BonaColinLannes2005} for $(\ve, \mu) \in \mathcal{R}_{KdV}$.\textcolor{red}{(see also \cite{Butrea JMPA21})}
%
%
%

The natural next step is to consider the Whitham-Boussinesq systems for $(\ve, \mu ) \in \mathcal{R}_{SW}$. \textit{In particular, the goal of this paper is to establish the well-posedness of \eqref{full dispersion}, \eqref{Whitham Boussinesq} and \eqref{2nd Whitham Boussinesq}, with uniform bounds, on time intervals of size $\mathcal{O}(\frac{1}{\ve})$}. Since point $1.$ of the justification is already established, the long-time existence and consistency remain. Using the method of Emerald, one can prove the consistency of any Whitham-Boussinesq system with the water wave system (see also \cite{emerald2021rigorous} for other full dispersion shallow water models). Therefore, having the long time well-posedness theory for \eqref{full dispersion}, \eqref{Whitham Boussinesq} and \eqref{2nd Whitham Boussinesq} will provide the final step for the full justification of these systems. 
%Though, the long time existence of \eqref{First Wh}, \eqref{Whitham Boussinesq} and \eqref{Third Wh} is still an open problem.

\subsection{Former well-posedness results}\label{Former well-posedness results}

Regarding system \eqref{full dispersion}, we know from previous studies that surface tension plays a fundamental role in the well-posedness theory. In fact, when $\beta=0$ the initial value problem associated to system \eqref{full dispersion} is probably ill-posed unless $\zeta>0$ (see the formal argument in Section 4 in \cite{KleinLinaresPilodSaut2018}). We refer to \cite{PeiWang2019} for a well-posedness under the non-physical condition $\zeta \geq c_0>0$. When surface tension is taken into account, system \eqref{full dispersion} was proved to be locally well-posed by Kalisch and Pilod \cite{KalischPilod2019} for $(\zeta,v) \in H^s(\mathbb R) \times H^{s+\frac12}(\mathbb R)$, $s>\frac52$, by using a modified energy method. We also refer to the work by Wang \cite{Wang2020} for an alternative proof using a nonlocal symmetrizer. However, it is worth noting that all these well-posedness resutls were proved on a short time without considering the small parameters $\ve$ and $\mu$. Finally, in the formal limit $\sqrt{\mu} |\xi| \rightarrow 0$, one recovers the Boussinesq system corresponding to  $(a,b,c,d) = (\frac13-\beta,0,0,0)$. This system has been proved in \cite{saut2017cauchy} to be well-posed on large time for $\beta >\frac{1}{3}$, while it is known to be ill-posed for $\beta < \frac{1}{3}$ \cite{ambrose2019global}. This is a formal indication that the threshold $\beta=\frac13$ will play an important role for the long well-posedness of \eqref{full dispersion}. We will come back to this issue in the next section (see Figure 1).

As far as we know, there are no well-posedness results for system \eqref{Whitham Boussinesq} even on short time. In the formal limit $\sqrt{\mu} |\xi| \rightarrow 0$, it reduces to the Boussinesq system corresponding to  $(a,b,c,d) = (0,0,\frac13,0)$, which is believed to be ill-posed \cite{KleinLinaresPilodSaut2018}.

Lastly, attention is turned to \eqref{2nd Whitham Boussinesq}. There are several results when $\beta=0$. In this case, Dinvay  \cite{dinvay2019well} proved short time local well-posedness for $(\zeta,v) \in H^{s+\frac12}(\mathbb R)\times H^{s+1}(\mathbb R)$, $s \geq 0$.  The proof is based on standard hyperbolic theory that involves a modified energy similar to \cite{KalischPilod2019}. This result was then extended in \cite{Achenef2020well}  by exploiting the smoothing effect of the linear flow using dispersive techniques improving the regularity to $H^{s}(\mathbb R)\times H^{s+\frac12}(\mathbb R)$, $s >-\frac1{10}$. Furthermore, when considering small data, the system is globally well-posed due to the control of the Hamiltonian. The estimates derived in the aforementioned papers are not uniform\footnote{After the completion of this work, we learned from Vincent Duch{\^e}ne that Louis Emerald proved in his PhD thesis long time well-posedness results for a class of Whitham-Boussinesq systems containing \eqref{2nd Whitham Boussinesq} in the case $\beta=0$ \cite{Emerald thesis} (see also \cite{Duschene}). His results deal with different systems than the one in this paper and thus complete each other well.} in $\mu$. However, a recent study by Tesfahun \cite{tesfahun2022long} proved that the system corresponding to \eqref{2nd Whitham Boussinesq} in the 2-dimensional case and without surface tension is well-posed on a time interval of order $\mathcal{O}(\frac{1}{\sqrt{\ve}})$ in the KdV$-$regime. Indeed, dispersive techniques are tailored-made for short waves and therefore seem not to be well suited to capture the long wave regime (see for instance \cite{LinaresPilodSaut} for similar results for the Boussinesq system in the KdV-KdV case).  Finally, in the case of surface tension $\beta>0$, Dinvay proved in \cite{dinvay2020SurfaceTension} the short time local well-posedness of the system by using modified energy techniques. This result also implies the small data global well-posedness in this case.

%%%%%%%----------------%%%%%%%%%--------------

\subsection{Main results} 
In the current paper, we take into account the small parameters $(\ve,\mu)$ and prove the well-posedness of \eqref{full dispersion}, \eqref{Whitham Boussinesq} and \eqref{2nd Whitham Boussinesq} on a time scale $\mathcal{O}(\frac1\ve)$. 

In the case of systems \eqref{Whitham Boussinesq} and \eqref{2nd Whitham Boussinesq}, we will work under the standard non-cavitation condition. 
\begin{Def}[Non-cavitation condition]\label{nonCavitation} Let $s >\frac{1}{2}$ and $\ve \in (0,1)$. We say the initial surface elevation $\zeta_0 \in H^s(\mathbb{R})$ satisfies the \lq\lq non-cavitation condition\rq\rq\: if there exist $h_0\in(0,1)$ such that 
\begin{equation}\label{NonCav}
	1+\ve\zeta_0(x) \geq h_0, \quad \text{for all} \: \: \: x\in \mathbb{R} .
\end{equation}
\end{Def}

In the case of system \eqref{full dispersion}, we will distinguish between the cases $\beta \geq \frac{1}{3}$ and $0<\beta < \frac{1}{3}$. More precisely, for $\beta \geq \frac{1}{3}$, we will also assume the non-cavitation condition in Definition \ref{nonCavitation}, while for $0<\beta < \frac{1}{3}$, we have to impose the following $\beta-$dependent surface condition.

\begin{Def}[$\beta-$dependent surface condition]\label{nonCavitationSurfaceTension} Let $s >\frac{1}{2}$, $\ve \in (0,1)$ and $\beta\in (0, \frac{1}{3})$. We say the initial surface elevation $\zeta_0 \in H^s(\mathbb{R})$ satisfy the \lq\lq$\beta-$dependent surface condition\rq\rq\: if 
\begin{equation}\label{Beta NonCav}
	1+ \ve \zeta_0(x) \geq h_\beta, \quad \text{for all}\:\: \:  x\in \mathbb{R},
\end{equation}
where $h_{\beta} = 1 - \frac{\beta}{2}$.
\end{Def}

\begin{remark}
For $0<\beta<\frac{1}{3}$, $K_{\mu}(\xi)$ is not a monotone function for positive frequencies as we can be seen in the figure below. 
This is why we choose to impose condition \eqref{Beta NonCav} in this case.
\begin{figure}[h]
	\hspace{-0.8cm}
	\includegraphics[scale=0.3]{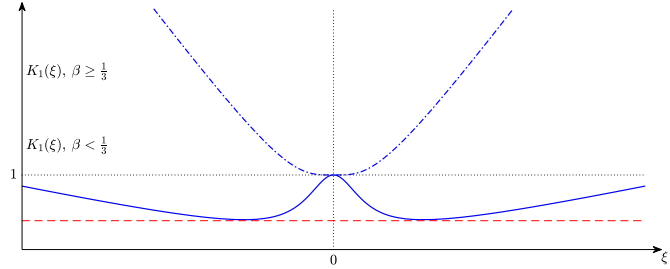}
	\caption{\small The multipiler $K_1(\xi)$ in the case when $\beta \geq \frac{1}{3}$ (dash-dot) and   $\beta < \frac{1}{3}$ (line). The horizontal line (dashed) specifies the minimum.}%, similar to Figure $1.1$.}
\end{figure}
\end{remark}
\begin{remark}
One can see the $\beta-$dependent surface condition as a constraint on the initial data that is related to the minimum of the function $K_{\mu}(\xi)$. For instance, if we consider the multiplier in Figure $1$, then an admissible initial datum must satisfy the constraint in the figure below. 
%
%
%

%MultiplierK
%
%
\begin{figure}[h!]
	\hspace{-1cm}
	\includegraphics[scale=0.3]{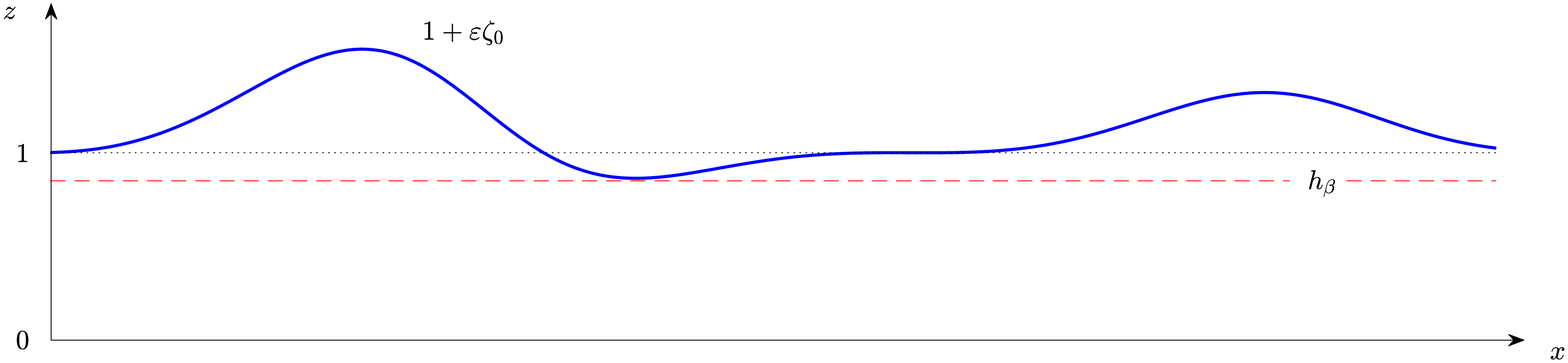}
	\caption{\small The blue line denotes the initial surface elevation $1 + \ve \zeta_0$, and is restricted by  $h_{\beta}$ when  $0<\beta<\frac{1}{3}$.}%, similar to Figure $1.1$.}
\end{figure}
\end{remark}

Before we state the main results, we define a natural solution space for systems \eqref{full dispersion} and \eqref{Whitham Boussinesq}. %associated to  \eqref{full dispersion}.%  by comparing the linear terms.

\begin{Def}\label{Function space}
	We define the norm on the function space $V^s_{\sqrt{\mu}}(\mathbb{R})$ to be 
	\begin{equation*}
		\|(\zeta, v)\|_{V^s_{\sqrt{\mu}}}^2: = \|\zeta\|_{H^s}^2 + \|v\|_{H^s}^2 + \sqrt{\mu} \|D^{\frac{1}{2}}v\|_{H^{s}}^2.
	\end{equation*}
\end{Def}	
%
%
%
%Then the  main result reads:
%
%
%
\begin{thm}\label{Well-posedness long time full dispersion} Let $s> 2$, $\beta>0$ and $\ve, \mu \in (0,1)$. Assume that $(\zeta_0,v_0 ) \in V^s_{\sqrt{\mu}}(\mathbb{R})$ satisfies either the non-cavitation condition \eqref{NonCav} in the case $\beta \geq 1/3$ or the $\beta-$dependent surface condition \eqref{Beta NonCav} in the case $0<\beta<\frac{1}{3}$.  Moreover,  we assume that
\begin{equation}\label{eps beta}
0< \ve \leq \frac{1}{k^2_{\beta} \|(\zeta_0,v_0 )\|_{V^s_{\sqrt{\mu}}}} \quad \text{for} \quad k^2_{\beta} 
=
\begin{cases}
	\frac{c}{\beta} \: \:  \quad \text{for} \:\: \: 0< \beta < \frac{1}{3} 
	\\
	c \beta \quad \text{for} \:  \: \: \beta\geq \frac{1}{3}
\end{cases}
\end{equation}
for some $c>0$. Then there exists a positive $T$ given by
\begin{equation}\label{Time T}
	T = \frac{1}{ k^{1}_{\beta}\|(\zeta_0,v_0 )\|_{V^s_{\sqrt{\mu}}}} \quad \text{with} \quad  k_{\beta}^1 = 
	\begin{cases}
		\frac{c}{\beta} \quad \quad 
		\text{for} \: \:  0<\beta<\frac{1}{3} 
		\\
		c\beta^2
		\: \:  \: \quad\text{for} \: \:   \beta\geq \frac{1}{3} 
	\end{cases}
\end{equation}
such that \eqref{full dispersion} admits a unique solution 
$$(\zeta, v) \in C([0,T/\ve] : V^s_{\sqrt{\mu}}(\mathbb{R})) \cap C^1([0,T/\ve] : V^{s-\frac{3}{2}}_{\sqrt{\mu}}(\mathbb{R})),$$
that satisfies 
\begin{equation}\label{bound in terms of initial data}
		\sup\limits_{t \in [0,T/\ve ]} \|(\zeta,v)\|_{V^s_{\sqrt{\mu}}} \lesssim \| (\zeta_0,v_0) \|_{V^s_{\sqrt{\mu}}}.
\end{equation}

Furthermore, there exist a neighborhood $\mathcal{U}$ of $(\zeta_0,v_0)$ such that the flow map
%
%
%\
\begin{equation*}
	F_{T,\ve, \mu }^s: V_{\sqrt{\mu}}^s(\R) \: \rightarrow \: C([0,\tfrac{T}{2\ve}]; V^s_{\sqrt{\mu}}(\R)), \quad 
	(\zeta_0,v_0) \: \: \mapsto \: \: (\zeta, v),
\end{equation*}
is continuous.

\end{thm}

\begin{remark} The proof of the continuous dependence on long time of order $\mathcal{O}(\frac1\ve)$ seems to be new for Boussinesq type systems. It relies on the Bona-Smith argument and could be easily adapted for the $(a,b,c,d)$-Boussinesq systems.

\end{remark}

\begin{remark}\label{physical time} 
	A heuristic argument can be made to argue  that the physical solutions appear when the initial data is of order one in terms of $\ve$ \cite{Personal comm}. 
	%The  main interest is for initial data of size $\| (\zeta_0,v_0) \|_{V^s_{\sqrt{\mu}}} = \mathcal{O}_{\ve}(1)$. 
	To illustrate this point, take the Burgers equation
	\begin{equation*}
		u_t - \ve u u_x = 0,
	\end{equation*}
	a simple model that can describe an inviscid fluid in shallow water theory. Then by the energy method, it is easy to deduce that the time of existence is of order $T \sim \frac{1}{\ve \|u_0\|_{H^s}}$ for $s> \frac{3}{2}$. As a consequence, we have that $T \sim  \frac{1}{\ve}$ if the initial data is of size $\mathcal{O}_{\ve}(1)$.
\end{remark}

\begin{remark}\label{Remark 1 on T - beta}
	%Even though $\ve$ and $T$ both depend on the surface tension, we have the long time existence when the initial data satisfy $\| (\zeta_0,v_0) \|_{V^s_{\sqrt{\mu}}} \sim 1$.  Indeed, 
	 If $\beta  \sim 1$ then $\ve \lesssim 1$ by \eqref{eps beta}, and so  \eqref{Time T} implies that $T/\ve \sim 1/\ve$. On the other hand, in the  case of having $\beta <<1$, \eqref{eps beta} would impose $\ve \lesssim \beta$ , and by  \eqref{Time T} we have the existence on the timescale $T/\ve \sim \beta/\ve$.

%	Formally, we could therefore expect \eqref{full dispersion} to be ill-posed when $\beta = 0$, but this is not included in the main result of the precent paper.  
\end{remark}

\begin{remark}
	Regarding the $\beta-$dependent surface condition, we demonstrate that the solution will persist for a long time and satisfy $\ve \zeta (x,t) \geq -c \beta $  for some constant  $c>0$. One should also note that this is coherent since $0<\ve \lesssim \beta$ as explained in the previous remark. For a related discussion on this physical condition see Subsection \ref{Former well-posedness results}.
\end{remark}

Next, we state a   well-posedness result for \eqref{Whitham Boussinesq}. The system does not feature any surface tension but is well-posed for a long time under the standard non-cavitation condition.

\begin{thm}\label{W-P Whitham Boussinesq} Let $s> \frac{3}{2}$ and $\mu \in (0,1)$. Assume that $(\zeta_0,v_0) \in V^s_{\sqrt{\mu}}(\mathbb{R})$ satisfies the non-cavitation condition \eqref{NonCav}. Also assume that for some $c>0$ that $0<\ve \leq c\big{(}\|(\zeta_0, v_0)\|_{V_{\sqrt{\mu}}^s}\big{)}^{-1}$.
Then there exists $T = c \big{(}\|(\zeta_0, v_0)\|_{V^s_{\sqrt{\mu}}}\big{)}^{-1}$ such that \eqref{Whitham Boussinesq} admits a unique solution
$$(\zeta, v) \in C([0,T/ \ve] : V^s_{\sqrt{\mu}}(\mathbb{R})) \cap C^1([0,T/\ve] : V^{s-1}_{\sqrt{\mu}}(\mathbb{R})),$$
that satisfies
\begin{equation*}
	\sup\limits_{t \in [0,T/\ve ]} \|(\zeta,v)\|_{V^s_{\sqrt{\mu}}} \lesssim \| (\zeta_0,v_0) \|_{V^s_{\sqrt{\mu}}}.
\end{equation*}

	In addition, the flow map is continuous with respect to the initial data. 	
\end{thm}

\begin{remark}
As far as we know, Theorem \ref{W-P Whitham Boussinesq} is the first well-posedness result for system \eqref{Whitham Boussinesq}. 
\end{remark}

Similarly, we can combine the techniques used to prove Theorem \ref{Well-posedness long time full dispersion} and Theorem \ref{W-P Whitham Boussinesq} to establish the long time well-posedness of \eqref{2nd Whitham Boussinesq} in the space:

\begin{Def}\label{2nd Function space}
	Define the norm on $X^s_{\beta, \mu}(\R)$ to be
	\begin{equation*}
		\|(\zeta,v)\|_{X^s_{\beta,\mu}}^2 : = \|\zeta\|_{H^s}^2 + \beta\mu\|D^1\zeta\|_{H^s}^2 + \| v \|^2_{H^s}   + \sqrt{\mu}\| D^{\frac{1}{2}} v \|^2_{H^s}.
	\end{equation*}
\end{Def}
\begin{thm}\label{W-P 2nd Whitham Boussinesq} Let $s> \frac{3}{2}$, $\beta>0$ and $\mu \in (0,1)$. Assume that $(\zeta_0,v_0) \in X^s_{\beta,\mu}(\mathbb{R})$ satisfies the non-cavitation condition \eqref{NonCav}. Also assume that for some $c>0$ that $0<\ve \leq c\big{(}\|(\zeta_0, v_0)\|_{X^s_{\beta,\mu}}\big{)}^{-1}$.
	Then there exists $T = c \big{(}\|(\zeta_0, v_0)\|_{X^s_{\beta,\mu}}\big{)}^{-1}$ such that \eqref{2nd Whitham Boussinesq} admits a unique solution
	$$(\zeta, v) \in C([0,T/ \ve] : X^s_{\beta,\mu}(\mathbb{R}) \cap C^1([0,T/\ve] : X^{s-1}_{\beta,\mu}(\mathbb{R})),$$
	that satisfies
	\begin{equation*}
		\sup\limits_{t \in [0,T/\ve ]} \|(\zeta,v)\|_{ X^s_{\beta,\mu}} \lesssim \| (\zeta_0,v_0) \|_{ X^s_{\beta,\mu}}.
	\end{equation*}
	
	In addition, the flow map is continuous with respect to the initial data. \\
\end{thm}
\begin{remark}\label{Remark beta thm 1.8}
	Including $\beta>0$ in the norm in the definition of $X^s_{\beta,\mu}(\R)$ will allow us to obtain a long time well-posedness result under the non-cavitation condition. Additionally,  when $0<\beta <\frac{1}{3}$ then $\ve$ is independent  from the surface tension parameter. 
\end{remark}

\begin{remark} For the sake of clarity, we have chosen to work in the one dimensional case. Theorems \ref{Well-posedness long time full dispersion}, \ref{W-P Whitham Boussinesq} and \ref{W-P 2nd Whitham Boussinesq} could possibly be extended to the $2$-dimensional case by following the same methods since the symbols $\mathcal{K}_{\mu}(D)$ and $\mathcal{T}_{\mu}(D)$ are radial.
\end{remark}

\subsection{Strategy and outline}

The proof of Theorem \ref{Well-posedness long time full dispersion}  relies mainly on  energy estimates similar to the ones provided in \cite{KalischPilod2019} on a fixed time. Though, we use the idea of Wang \cite{Wang2020}, who included the nonlocal operator $\mathcal{K}_{\mu}(D)$ in the definition of the energy\footnote{Wang actually used  this multiplier in the case $\mu=1$.}:
\begin{Def}\label{Def Energy}
Let $(\eta, u)  = \ve (\zeta, v)$ and $J^s$ be the bessel potential of order $-s$. Then we define the energy associated to \eqref{full dispersion}:
\begin{equation*}
	E_s(\eta, u) : = \int_{\R} (J^s \eta)^2 \: dx + \int_{\R} (\mathcal{K}_{\mu}(D) + \eta) (J^s u)^2\: dx.
\end{equation*}
\end{Def}
\noindent
This energy formulation will free us to cancel out specific nonlinear terms that appear naturally in the computations yielding the estimate
\begin{equation}\label{Energy rate}
	\frac{d}{dt} E_s(\eta, u ) \lesssim_{\beta} \big{(} E_s(\eta, u)\big{)}^{\frac{3}{2}}.
\end{equation} 
Combined with the coercivity of the energy, then by a standard  bootstrap argument, one deduces a solution with the lifespan of $T_0 = \mathcal{O}(\frac{1}{\ve})$. We refer the reader to Proposition \ref{Energy fully disp Boussinesq} and Lemma \ref{Lemma time of existence} for these results. The proof  of the energy estimate is similar to the one presented  in  \cite{Wang2020}, but we keep track of the small parameters. We should also note that estimate \eqref{Energy rate} is applied to a regularized version of \eqref{full dispersion}, where we recover the original system using a Bona-Smith argument. 

To run the Bona-Smith argument for $s > 2$, one  classically needs to estimate the difference between two solutions at the  $V^0_{\sqrt{\mu}}(\R)-$level. These estimates will be the most technical point of the paper and are specific to the dependence of the small parameters. In short, the technical difficulty is related to the  apparent need for 'generalized' Kato-Ponce type commutator estimates on $\mathcal{K}_{\mu}(D)$  (see Lemma  \ref{Kato ponce} and the generalization for $\mathcal{K}_{\mu}(D)$ in Lemma \ref{Commutator L2}). Whereas for the case $\mu = 1$, one can use Calder\'on type estimates to simplify $\mathcal{K}_{\mu}(D)$ directly (see \cite{KalischPilod2019} and the reformulated system $(2.1)$). The main idea will be to split $\mathcal{K}_{\mu}(D)$ in high and low frequencies, and then derive new commutator estimates that allow us to obtain the necessary order of $\mu$ in the estimates related to the energy.

For the proof of Theorem \ref{W-P Whitham Boussinesq}, we follow the same strategy, but in this case, the dispersion operator \eqref{Op T} is regularizing. The trick will be to introduce a scaled Bessel potential in the energy, allowing us to mimic the properties of \eqref{mult K}. The energy is given by:
\begin{Def}\label{Energy Whitham Boussinesq} Let $(\eta, u) = \ve (\zeta, v)$ and $J^{\frac{1}{2}}_{\mu} $ be the scaled Bessel potential defined by the symbol $\xi \mapsto (1+\mu \xi^2)^{\frac{1}{4}}$ in frequency. Then the energy associated to \eqref{Whitham Boussinesq} reads:
	\begin{equation*}
		\mathcal{E}_s(\eta, u) := \int_{\R} \mathcal{T}_{\mu}(D) (J^{\frac{1}{2}}_{\mu}J^s\eta)^2 \: dx +   \int_{\R}  (1+\eta)(J_{\mu}^{\frac{1}{2}}J^s u)^2\: dx.
	\end{equation*}
\end{Def}
\noindent
The energy formulated in Definition \ref{Energy Whitham Boussinesq} is new and will require commutator estimates specific to the equation. This will, in turn, allow us to decouple the parameters $\mu$ and $\ve$ in the estimates and, by extension, provide an estimate in the form of \eqref{Energy rate}. 

In the same spirit, we define a modified energy for system \eqref{2nd Whitham Boussinesq}:

\begin{Def}\label{2nd Energy Whitham Boussinesq}Let $(\eta, u) = \ve (\zeta, v)$ and $\beta >0$.  Then the energy associated to \eqref{2nd Whitham Boussinesq} reads:
	\begin{equation*}
		\mathscr{E}_s(\eta,u) := \int_{\R} (1 + \beta \mu D^2)(J^s \eta)^2 \: dx + \int_{\R} (\mathcal{T}^{-1}_{\mu}(D) + \eta) (J^s u)^2 \: dx.
	\end{equation*}
\end{Def}
\noindent
Note also that the energy includes the surface tension parameter $\beta$ and will allow us to deduce  an estimate on the form \eqref{Energy rate}, where the coercivity estimate will be uniform in $\beta$. In turn, this will provide the long time well-posedness for $\beta <<1$ and $T/\ve \sim 1/\ve$ as pointed out in Remark \ref{Remark beta thm 1.8}.

\bigskip

The paper is organized as follows. In Section \ref{PointW}, we introduce some important technical results whose proofs will be postponed to the appendix. In the same section, we also present new commutator estimates needed to treat the nonlinear terms when estimating the energy in Sections \ref{Energy est} and \ref{Energy est diff}. Lastly,  we conclude in Section \ref{Main proof of thm} by combining the results obtained in the former sections to prove Theorem \ref{Well-posedness long time full dispersion} in full detail, while the proof of Theorem \ref{W-P Whitham Boussinesq} and Theorem \ref{W-P 2nd Whitham Boussinesq} will follow by the same arguments.

\subsection{Notation} 
\begin{itemize}
	\item We  let $c$ denote a positive constant independent of $\mu, \ve$ that may change from line to line. Also, as a shorthand, we use the notation $a \lesssim b$ to mean $a \leq c\: b$. Similarly, if the constant depends on $\beta$, we write $a \lesssim_{\beta} b $.
	In particular, we define the constants depending on $\beta$,
	\begin{equation}\label{Constant beta}
		c_{\beta}^1 =
		\begin{cases}
			c\beta \quad \text{for} \: \: 0<\beta<\frac{1}{3}\\
			c \:\:\hspace{0.05cm}\quad \text{for} \: \: \beta\geq \frac{1}{3} 
		\end{cases}
		\quad \text{and} \quad 
		c_{\beta}^2 =
		\begin{cases}
			c \: \:   \hspace{0.05cm}\quad \text{for} \: \:  0<\beta<\frac{1}{3}\\
			c\beta \quad \text{for} \: \:  \beta\geq \frac{1}{3} 
		\end{cases}
	\end{equation}

	\item Let $(V,\|\cdot \|_{V})$ be a vector space. Then for $\alpha\geq 0$, $\lambda>0$ and $f_{\lambda}\in V$ be a function depending on $\lambda$, we define the \lq\lq big$-\mathcal{O}$\rq\rq\: notation  to be
	\begin{equation*}
		\|f_{\lambda}\|_{V} = \mathcal{O}(\lambda^{\alpha}) \iff \lim\limits_{\lambda\rightarrow 0}\lambda^{-\alpha}\|f_{\lambda}\|_{V} <\infty.
	\end{equation*}
	Similarly, we define the \lq\lq small$-o$\rq\rq\: notation to be
	\begin{equation*}
		\|f_{\lambda}\|_{V} = o(\lambda^{\alpha}) \iff \lim\limits_{\lambda\rightarrow 0}\lambda^{-\alpha}\|f_{\lambda}\|_{V} =0.
	\end{equation*}
	%
	%
	%
%	and in the case $\alpha=0$, the decay may only be logarithmic.
	
	\item Let $L^2(\R)$ be the usual space of square integrable functions with norm $\|f\|_2 = \sqrt{\int_{\R} |f(x)|^2 \: dx}$. Also, for any $f,g \in L^2(\R)$ we denote the scalar product by $\big{(} f,g \big{)}_{L^2} = \int_{\R} f(x) \overline{g(x)} \: dx$.

	\item For any tempered distribution $f$, the operator $\mathcal{F}$ denoting the Fourier transform, applied to $f$, will be written as $\hat{f}(\xi)$ or $\mathcal{F}f(\xi)$.

	\item Let $m:\R \rightarrow \mathbb{R}$ be a smooth function. Then we will use the notation $m(D)$ for a multiplier defined in frequency by $\widehat{m(D) f}(\xi) = m(\xi) \hat{f}(\xi)$.

	\item For any $s \in \mathbb{R}$ we call the multiplier $ \widehat{D^{s}f} (\xi)= |\xi|^s \hat{f}(\xi)$ the Riesz potential of order $-s$. One should note that $D^1 = \mathcal{H} \partial_x$, where $\widehat{\mathcal{H}f}(\xi) = -i\: \text{sgn}(\xi)\hat{f}(\xi)$  is the Hilbert transform.

	\item For any $s \in \mathbb{R}$ we call the multiplier $J^s = (1+D^2)^{\frac{s}{2}} = \langle D \rangle^s$ the Bessel potential of order $-s$. Moreover, the Sobolev space $H^s(\mathbb{R})$  is equivalent to the weighted $L^2-$space; $\|f\|_{H^s} = \|J^s f\|_{L^2}$. We also find it  convenient to define $J^{\frac{1}{2}}_{\mu}$ which is a multiplier assosiated to the symbol:
	\begin{equation}\label{scaled J_mu}
		\mathcal{F}(J^{\frac{1}{2}}_{\mu}f)(\xi) = (1+\mu\xi^2)^{\frac{1}{4}}\hat{f}(\xi).
	\end{equation}

	\item We say $f$ is a  Schwartz function $\mathscr{S}(\mathbb{R})$, if $f \in C^{\infty}(\mathbb{R})$ and satisfies for all $\alpha, \beta \in \mathbb{N}$,
	\begin{equation*}
		\sup \limits_{x} |x^{\alpha} \partial_x^{\beta} f | < \infty.
	\end{equation*}
	
	\item If  $A$ and $B$ are two operators, then we denote the commutator between them to be $[A,B] = AB - BA$.	
\end{itemize}

\section{Preliminary results}\label{PointW}

\subsection{Pointwise estimates}

The first result concerns the properties of the dispersive part of the equation. Namely, we deduce pointwise estimates for the multipliers \eqref{mult K} and \eqref{Op T} that are needed to obtain the coercivity of the energy (see, for instance, equation \eqref{equiv. full dispersion} below). Moreover, these estimates will prove essential when dealing with the nonlinear parts of the equation that appear in the energy estimates.

\begin{lemma}\label{Pointwise est.} Let $\mu\in(0,1)$. Then we have the following pointwise estimates on the kernel $K_{\mu}(\xi):$
	\begin{itemize}
		\item For $\beta\geq 0$, we have the upper bound
		\begin{equation}\label{sqrt K}
			K_{\mu}(\xi)\lesssim 1 + \beta (1+\beta \sqrt{\mu }| \xi|).
		\end{equation}
		
		\item If $\beta \geq \frac{1}{3}$, then for all $h_0 \in (0,1)$ we have the lower bound
		\begin{equation}\label{K pos}
			K_\mu(\xi)  \geq (1-\frac{h_0}{2}) + c  \sqrt{\mu} |\xi|,
		\end{equation}
		whereas, if $0<\beta < \frac{1}{3}$, we have the lower bound
		\begin{equation}\label{K pos beta}
			K_\mu(\xi) \geq \beta + c \beta \sqrt{\mu} |\xi |.
		\end{equation}

		\item The derivative of the symbol $K_{\mu}(\xi)$ satisfies
		\begin{equation}\label{derivative Symbol}
			\bigg{|}\frac{d}{d\xi}\sqrt{K_{\mu}(\xi)}\bigg{|} \lesssim \langle \xi \rangle^{-1}+ \sqrt{\beta} \mu^{\frac{1}{4}} \langle \xi \rangle^{-\frac{1}{2}}.
		\end{equation}
	\item We have the following comparison of $\sqrt{\mathcal{K}_{\mu}(\xi)}$ by
	\begin{equation}\label{Comparison}
		\big{|} 
		\sqrt{K_{\mu}(\xi)} - \sqrt{\beta} \mu^{\frac{1}{4}} |\xi|^{\frac{1}{2}}
		\big{|} \lesssim\sqrt{\beta} + \beta.
	\end{equation}
	\item There holds
	\begin{equation}\label{K point. for commutator}
		\sqrt{K_{\mu}(\xi)} \langle \xi \rangle^{s-1}  |\xi| \lesssim (\sqrt{\beta} + \beta) \langle \xi \rangle^{s} + \sqrt{\beta} \mu^{\frac{1}{4}}  \langle \xi \rangle^{s}   |\xi|^{\frac{1}{2}}.
	\end{equation}
	\end{itemize}
\end{lemma}

\begin{remark}\label{Constant FIX}
	%One should note that we are not interested in tracing the dependence in $\beta$ for any of the upper bounds presented in the lemmas above. Though, 
	For inequality \eqref{K pos beta}, it is crucial to specify the  dependence in $\beta$ as it will provide the coercivity of the energy when $0<\beta<\frac{1}{3}$. The same is true for \eqref{K pos}, whose importance will be revealed  in the proof of Proposition \ref{Energy fully disp Boussinesq} below. Though, we note that \eqref{K pos beta} does not agree with \eqref{K pos} when $\beta = \frac{1}{3}$. This is because the lower bound in  \eqref{K pos beta} is not optimal, but it does not play a role for the overall result.

\end{remark}

\begin{remark}
		We also trace the dependence in $\beta$ for the first pointwise estimate \eqref{sqrt K}, and it will sometimes be replaced with $c_{\beta}^2$ given by \eqref{Constant beta}. This constant will again appear when we prove the energy estimates which will provide the size of the time of existence (see Lemma \ref{Lemma time of existence} in the proof Theorem \ref{Well-posedness long time full dispersion}). 
\end{remark}

	The proof of Lemma \ref{Pointwise est.} is technical and  postponed to the Appendix in Section \ref{A2}. A corollary of Proposition \ref{Pointwise est.} may now be stated.

\begin{cor}\label{Coericivity and continuity} Take $f \in \mathscr{S}(\mathbb{R})$, $\mu\in(0,1)$ and $s \in \mathbb{R}$. Then in the case $\beta\geq \frac{1}{3}$ and for all $h_0\in(0,1)$  we have
	\begin{equation}\label{estimate K in Hs}
	(1-\frac{h_0}{2})	\|f\|_{H^s}^2 + c \sqrt{\mu}  \| D^{\frac{1}{2}}f\|_{H^s}^2 
		\leq
		| |\sqrt{\mathcal{K}_{\mu}}(D)  f\|_{H^s}^2 \leq c_{\beta}^2 \|f\|_{H^s}^2 +c \beta\sqrt{\mu}   \|D^{\frac{1}{2}}f\|_{H^{s}}^2.
	\end{equation}
	Similarly, in the case $0<\beta < \frac{1}{3}$ there holds
	\begin{equation}\label{estimate K Low}
		\beta\|f\|_{H^s}^2 + c \beta\sqrt{\mu}  \| D^{\frac{1}{2}}f\|_{H^{s}}^2 \leq| |\sqrt{\mathcal{K}_{\mu}}(D)  f\|_{H^s}^2\leq c_{\beta}^2 \|f\|_{H^s}^2 + c\sqrt{\mu}   \|D^{\frac{1}{2}}f\|_{H^{s}}^2.
	\end{equation}

\end{cor}
\begin{proof}
	The upper bound in \eqref{estimate K in Hs} follows by Plancherel's identity and the pointwise estimate \eqref{sqrt K}, while the lower bound is a consequence  of \eqref{K pos}.
	
	 In the same way, for $0<\beta <\frac{1}{3}$, then \eqref{estimate K Low} is deduced from \eqref{K pos beta}.
 \end{proof}

Similarly, we state some useful pointwise estimates on $T_{\mu}(\xi)$ and the scaled Bessel potential $J^{\frac{1}{2}}_{\mu}$, where the proof is presented in Appendix \ref{A2}.

\begin{lemma}\label{Pointwise est. on T} Let $\mu\in(0,1)$. Then we have the following pointwise estimates on the kernel $T_{\mu}(\xi):$
	\begin{itemize}
		\item For all $h_0 \in (0,1)$ there holds
		\begin{equation}\label{Inverse T_mu}
			(1-\frac{h_0}{2})+ c \sqrt{\mu} |\xi| \leq (T_{\mu}(\xi))^{-1} \lesssim 1+\sqrt{\mu} |\xi|.
		\end{equation}
		\item There holds
		\begin{equation}\label{T_mu J_mu equiv}
			1	\lesssim T_{\mu} (\xi) \langle \sqrt{\mu} \xi \rangle \lesssim 1.
		\end{equation}
		\item For $s\in \R$ there holds
		\begin{equation}\label{Derivative J^sJ_mu}
			\bigg{|}\frac{d}{d\xi}
			\langle\xi\rangle^s\langle \sqrt{\mu} \xi \rangle^{\frac{1}{2}}
			\bigg{|} \lesssim \langle \xi \rangle^{s-1} \langle \sqrt{\mu} \xi \rangle^{\frac{1}{2}}.
		\end{equation}
		\item For $s\in \R$ there holds
		\begin{equation}\label{Derivative sqrtT Js J_mu}
			\bigg{|}\frac{d}{d\xi}
			\sqrt{T_{\mu}(\xi)}\langle\xi\rangle^s\langle \sqrt{\mu} \xi \rangle^{\frac{1}{2}}
			\bigg{|}\lesssim \langle \xi \rangle^{s-1}.
		\end{equation}
		\item There holds
			\begin{equation}\label{Comparison J_mu D_mu}
			\bigg{|}
			 \langle \sqrt{\mu}  \xi \rangle^{\frac{1}{2}}- \mu^{\frac{1}{4}}|\xi|^{\frac{1}{2}}
			\bigg{|}
			\lesssim 1.
		\end{equation}

	\end{itemize}
\end{lemma}
A direct consequence of the above estimates can now be given. 
\begin{cor}
	Let $f \in \mathscr{S}(\R)$, $\mu\in(0,1)$, $s \in \R$  and $c>0$. Then for all $h_0\in(0,1)$ there holds
	\begin{equation}\label{boundedness of T}
		\|\sqrt{\mathcal{T}_{\mu}}(D)f\|_{L^2} \leq \|f\|_{L^2}.
	\end{equation}
	\begin{equation}\label{Inverse T_mu in Hs}
		(1-\frac{h_0}{2})\|f\|_{H^s}^2  + c\sqrt{\mu}\|D^{\frac{1}{2}}f\|_{H^{s}}^2 \leq \| \sqrt{\mathcal{T}_{\mu}}^{-1}(D) f\|^2_{H^s}\lesssim 	\|f\|_{H^s}^2  + c\sqrt{\mu}\|D^{\frac{1}{2}}f\|_{H^{s}}^2.
	\end{equation}
	\begin{equation}\label{Equiv sqrt T}
		\|f\|_{H^s} \lesssim\| \sqrt{\mathcal{T}_{\mu}}(D)J^{\frac{1}{2}}_{\mu} f\|_{H^s}\lesssim \|f\|_{H^s}.
	\end{equation}
	\begin{equation}\label{J_mu}
	\|f\|_{H^s}^2 + \sqrt{\mu} \|D^{\frac{1}{2}}f\|_{H^{s}}^2 \lesssim	\|J^{\frac{1}{2}}_{\mu} f \|_{H^s}^2 \lesssim \|f\|_{H^s}^2 + \sqrt{\mu} \|D^{\frac{1}{2}}f\|_{H^{s}}^2.
	\end{equation}
	
\end{cor}

\subsection{Commutator estimates}

To handle derivatives in the nonlinear parts of the equations, we  need commutator estimates on $\mathcal{K}_{\mu}(D)$ and $\mathcal{T}_{\mu}(D)$.  
\begin{lemma}\label{Commutator K at level s}
Let $f,g \in \mathscr{S}(\R)$, $\mu \in (0,1)$, $s\geq 1$,  and $t_0 > \frac{1}{2}$. Then we have the following commutator estimate
\begin{align}\label{Commutator low freq}
	\notag
	\| [  \sqrt{\mathcal{K}_{\mu}}(D) J^s , f]\partial_x g \|_{L^2} & \lesssim (c_{\beta}^2\|f\|_{H^s} + \sqrt{\beta}\mu^{\frac{1}{4}}\| D^{\frac{1}{2}}f \|_{H^{s}})\|\partial_x g\|_{H^{t_0}}
	\\
	 & \hspace{0.5cm}+ (c_{\beta}^2\|g\|_{H^s} + \sqrt{\beta}\mu^{\frac{1}{4}}\| D^{\frac{1}{2}}g \|_{H^{s}})\| \partial_x f \|_{H^{t_0}} .
\end{align} 
\end{lemma}
In the high regularity setting, the proof will follow the  same lines as in \cite{Wang2020}, but we track the dependence in $\mu$ and $\beta$ using the pointwise estimates above.
\begin{proof}

First, write the commutator as a bilinear form:
\begin{align*}
	\big{\|}  [  \sqrt{ \mathcal{K}_{\mu}}(D)J^s , f]\partial_x g\big{\|}_{L^2}
	=
	\bigg{\|}
	\int_{\R} \Big( 
	\sqrt{K_{\mu}(\xi)} \langle \xi \rangle^s -\sqrt{K_{\mu}(\rho)} \langle \rho \rangle^s 
	\Big) \hat{f}(\xi-\rho)\widehat{\partial_x g}(\rho)\: d\rho \bigg{\|}_{L^2_{\xi}}.
\end{align*}
Then if $a = \min\{|\xi-\rho|, |\rho|\}$ and $b = \max\{|\xi-\rho|, |\rho|\}$, we can  use the mean value theorem, leaving us to estimate the following terms
\begin{align*}
	\Big|  \sqrt{K_{\mu}(\xi)} \langle \xi \rangle^s - \sqrt{K_{\mu}(\rho)} 	\langle \rho \rangle^s  \Big|
	\lesssim
	\sup\limits_{\omega \in (a,b)}|m(\omega)| \:  |\xi - \rho|,
\end{align*}
where
\begin{equation*}
	m(\omega) = m_1(\omega) + m_2(\omega) =\langle \xi \rangle^s \frac{d}{d \xi} \sqrt{K_{\mu}(\xi)} + \langle \xi \rangle^{s-1}\sqrt{K_{\mu}(\xi)}.
\end{equation*}
But using \eqref{Comparison} to estimate $m_1(\omega)$ and \eqref{derivative Symbol} to treat $m_2(\omega)$, we deduce
\begin{equation}\label{est on m}
	m(\omega) \lesssim c_{\beta}^2 \langle\omega\rangle^{s-1} + \sqrt{\beta}\mu^{\frac{1}{4}}\langle\omega \rangle^{s-1} |\omega|^{\frac{1}{2}},
\end{equation}
where the upper bound is increasing for $s\geq 1$. Therefore the supremum is attained at $|\xi-\rho|$ or $|\rho|$ by definition of $b$. In particular, if $b = |\xi- \rho|$ then we may   conclude by Minkowski integral inequality, the Cauchy-Schwarz inequality  and \eqref{est on m} that
\begin{align*}
	\big{\|}  [  \sqrt{ \mathcal{K}_{\mu}}(D)J^s , f]\partial_x g\big{\|}_{L^2}
	&
 	\lesssim c_{\beta}^2
	\bigg{\|}
	\int_{\R}  \langle\xi - \rho \rangle^{s-1} |\xi-\rho| \: |\hat{f}(\xi-\rho) | \: |\widehat{\partial_x g}(\rho)|\: d\rho
	\bigg{\|}_{L^2_{\xi}} 
	\\
	& 
	\hspace{0.5cm}
	+  \sqrt{\beta}\mu^{\frac{1}{4}}
	 \bigg{\|}
	 \int_{\R} \langle \xi - \rho \rangle^{s-1} |\xi - \rho|^{\frac{1}{2}} |\xi-\rho| \: |\hat{f}(\xi-\rho) | \: |\widehat{\partial_x g}(\rho)|\: d\rho
	 \bigg{\|}_{L^2_{\xi}} 
	\\
	& \lesssim (c_{\beta}^2\|f\|_{H^s} + \sqrt{\beta}\mu^{\frac{1}{4}}\| D^{\frac{1}{2}}f \|_{H^{s}})\|\partial_x g\|_{H^{t_0}},
\end{align*}
%\begin{align*}
%	\bigg{\|}
%	\int_{\R} |m_i(\xi-\rho)|\: |\xi-\rho| \: |\hat{f}(\xi-\rho) | \: |\widehat{\partial_x g}(\rho)|\: d\rho
%	\bigg{\|}_{L^2_{\xi}} 
%	& \leq \| m_i(\xi) |\xi|\hat{f} \|_{L^2_{\xi}}  \|\partial_x g \|_{H^{t_0}},
%\end{align*}
%
%
%
for $t_0>\frac{1}{2}$. On the other hand, if $b = |\rho|$, then we make a change of coordinates and argue similarly to deduce,
\begin{align*}
	\big{\|}  [  \sqrt{ \mathcal{K}_{\mu}}(D)J^s , f]\partial_x g\big{\|}_{L^2}
	&
	\lesssim c_{\beta}^2
	\bigg{\|}
	\int_{\R}    \langle \xi - \gamma \rangle^{s-1} |\widehat{\partial_x g}(\xi- \gamma)| \:  |\gamma| \: |\hat{f}(\gamma)| \: d\gamma
	\bigg{\|}_{L^2_{\xi}} 
	\\
	& 
	\hspace{0.5cm}
	+  \sqrt{\beta}\mu^{\frac{1}{4}}
	\bigg{\|}
	\int_{\R}\langle \gamma \rangle^{s-1} | \xi - \gamma|^{\frac{1}{2}}  |\widehat{\partial_x g}(\xi- \gamma)| \:    |\gamma| \: |\hat{f}(\gamma)| \: d\gamma
	\bigg{\|}_{L^2_{\xi}} 
	\\
	& \lesssim (c_{\beta}^2\|g\|_{H^s} + \sqrt{\beta}\mu^{\frac{1}{4}}\| D^{\frac{1}{2}}g \|_{H^{s}})\|\partial_x f\|_{H^{t_0}}.
\end{align*}
Adding the two scenarios, we may conclude  that \eqref{Commutator low freq} holds. 

\end{proof}

We will also need a commutator estimates on $\mathcal{T}_{\mu}(D)$ and $J^{\frac{1}{2}}_{\mu}$. 
\begin{lemma}
	Let $f,g \in \mathscr{S}(\R)$, $s\geq1$, $t_0 > \frac{1}{2}$, $\mu \in (0,1)$ and $J^{\frac{1}{2}}_{\mu}$ as defined in \eqref{scaled J_mu}. 
	\begin{itemize}
		\item 	Then we have a Kato-Ponce type estimate
		\begin{align}\notag
			\| [J^s J^{\frac{1}{2}}_{\mu},f]\partial_x g\|_{L^2} 
			&
			\lesssim
			(\|f\|_{H^s} + \mu^{\frac{1}{4}} \|D^{\frac{1}{2}}f\|_{H^s})\|\partial_x g \|_{H^{t_0}} 
			\\
			&
			\hspace{0.5cm}
			+ \label{Commutator J^sJ_mu}
			(\|g\|_{H^s} + \mu^{\frac{1}{4}} \|D^{\frac{1}{2}}g\|_{H^s})\|\partial_x f \|_{H^{t_0}}.
		\end{align}
		\item There holds
		\begin{equation}\label{Commutator Js T}
			\|[\sqrt{\mathcal{T}_{\mu}}(D)J^s J^{\frac{1}{2}}_{\mu},f]\partial_xg \|_{L^2} \lesssim \| f \|_{H^{s}} \| g\|_{H^{t_0+1}} + \| f \|_{H^{t_0+1}} \| g\|_{H^{s}}.
		\end{equation}

	\end{itemize}

\end{lemma}

\begin{proof}
	The proof is similar to the one of Lemma \ref{Commutator K at level s} and relies on the pointwise estimates established in Lemma \ref{Pointwise est. on T}. Indeed, for \eqref{Commutator J^sJ_mu} we define $a_1(D) (f,g) := [J^s J^{\frac{1}{2}}_{\mu},f]\partial_x g$ and use the mean value theorem combined with \eqref{Derivative J^sJ_mu} to deduce
	\begin{align*}
		|\hat{a}_1(\xi) (f,g)|
		& \leq
		\int_{\R} 
		\Big{|}
		\langle \xi \rangle^s \langle \sqrt{\mu} \xi \rangle^{\frac{1}{2}} -	\langle \rho\rangle^s \langle \sqrt{\mu} \rho \rangle^{\frac{1}{2}} 
		\Big{|}
		|\hat{f}(\xi-\rho)| \: | \widehat{\partial_x g}(\rho)| \: d\rho
		\\ 
		& \lesssim
		\int_{\R} 
		\langle \xi -\rho\rangle^{s-1} \langle \sqrt{\mu} (\xi-\rho) \rangle^{\frac{1}{2}} 
		|\xi - \rho| \:
		|\hat{f}(\xi-\rho)| \: | \widehat{\partial_x g}(\eta)| \: d\rho
		\\
		&
		\hspace{0.5cm}
		+
		\int_{\R} 
		\langle \rho\rangle^{s-1} \langle \sqrt{\mu} \rho \rangle^{\frac{1}{2}}
		|\xi - \rho| \:
		|\hat{f}(\xi-\rho)| \: | \widehat{\partial_x g}(\rho)| \: d\rho.
	\end{align*}
	Then if we apply the $L^2(\R)-$norm with respect to $\xi$, we can argue as in Lemma \ref{Commutator K at level s} that
	\begin{align*}
			\| \hat{a}_1(\xi)(f,g) \|_{L^2_{\xi}} & \lesssim \|J^{\frac{1}{2}}_{\mu}f\|_{H^s} \int_{\R} | \widehat{\partial_x g}(\rho)| \: d\rho + \|J^{\frac{1}{2}}_{\mu}g\|_{H^s} \int_{\R} |\rho| \: |\hat{f}(\rho)| \: d\rho.
	\end{align*}
	Then use the definiton of $a_1(D)(f,g)$ and \eqref{J_mu} to conclude. 
	
	The proof of \eqref{Commutator Js T} is the same, with $a_2(D)(f,g) := [\sqrt{\mathcal{T}_{\mu}}(D)J^s J^{\frac{1}{2}}_{\mu},f]\partial_xg$. We use \eqref{Derivative sqrtT Js J_mu} to find that
	\begin{align*}
		|\hat{a}_2(\xi)(f,g)|
		& \leq
		\int_{\R} 
		\Big{|}
		\sqrt{T_{\mu}(\xi)} \langle \xi \rangle^s \langle \sqrt{\mu} \xi \rangle^{\frac{1}{2}} - \sqrt{T_{\mu}(\rho)} 	\langle \rho \rangle^s \langle \sqrt{\mu} \rho \rangle^{\frac{1}{2}} 
		\Big{|}
		|\hat{f}(\xi-\rho)| \: | \widehat{\partial_x g}(\rho)| \: d\rho
		\\ 
		& \lesssim
		\int_{\R} 
		\langle \xi -\rho\rangle^{s-1} 
		|\xi - \rho| \:
		|\hat{f}(\xi-\rho)| \: | \widehat{\partial_x g}(\rho)| \: d\rho
		\\
     	&
		\hspace{0.5cm}
		+
		\int_{\R} 
		\langle \rho\rangle^{s-1}
		|\xi - \rho| \:
		|\hat{f}(\xi-\rho)| \: | \widehat{\partial_x g}(\rho)| \: d\rho,
	\end{align*}
	and the result follows.

\end{proof}

Next, we state the classical Kato-Ponce commutator estimate. We will use it repeatedly to commute the Bessel potential with functions to obtain the desired energy estimates in the coming sections.

\begin{lemma}[Kato - Ponce commutator estimates \cite{KatoPonce1988}] \label{Kato ponce}Let $s\geq 0$, $p,p_2,p_3 \in (1,\infty)$ and $p_1, p_4 \in (1,\infty]$ such that $\frac{1}{p} = \frac{1}{p_1} + \frac{1}{p_2} = \frac{1}{p_3} + \frac{1}{p_4}$. Then
\begin{equation}\label{Prod. Kato-Ponce}
		\|J^s(fg)\|_{L^p} \lesssim \| f\|_{L^{p_1}} \|J^{s}g\|_{L^{p_2}} + \|J^s f \|_{L^{p_3}} \|g\|_{L^{p_4}}
\end{equation}
and
\begin{equation}\label{K-P}
	\|[J^s,f]g\|_{L^p} \lesssim \|\partial_x f\|_{L^{p_1}} \|J^{s-1}g\|_{L^{p_2}} + \|J^s f \|_{L^{p_3}} \|g\|_{L^{p_4}}.
\end{equation}
\end{lemma}

Similar commutator estimates also hold for more general multipliers. In fact, by splitting the frequency domain into two parts using smooth cut-off functions defined in frequency, we can obtain sharper commutator estimates specific to equation  \eqref{full dispersion}.

\begin{Def}\label{multiplier}
	We define the smooth cut-off functions $\chi^{(i)}\in\mathscr{S}(\mathbb{R})$ as Fourier multipliers
	\begin{equation*}
		\mathcal{F}(\chi^{(i)}(D) f)(\xi) = \chi^{(i)}(|\xi|)\hat{f}(\xi),
	\end{equation*}
	for any $f \in \mathscr{S}(\mathbb{R})$ with the following properties: 
	\begin{equation*}
		0\leq \chi^{(i)}(\xi) \leq 1, \qquad 	(\chi^{(1)}(\xi))^2 + (\chi^{(2)}(\xi))^2 = 1 \quad \text{on} \quad  \R,
	\end{equation*}
	and
	%	
	%
	% 
%	\begin{equation*}
%		\chi^{(1)} + \chi^{(2)} = 1
	%\end{equation*}
	%
	%
	%
	\begin{equation*}
	\text{supp} \:  \chi^{(1)} \subset 
	[-1,1]
	,
	\quad
	\text{supp} \: \chi^{(2)} \: \subset \R \backslash \big[-\frac{1}{2}, \frac{1}{2} \big].
\end{equation*}
		Moreover, we denote the scaled version in $\mu$ by $\chi^{(i)}_{\mu}(\xi) = \chi^{(i)}(\sqrt{\mu} \xi)$.
	
\end{Def}

We have the results:

\begin{lemma}\label{Commutator L2} Let $s>\frac{3}{2}$, $\mu \in (0,1)$ and $f,g \in \mathscr{S}(\mathbb{R})$. 
	\begin{itemize}
		\item Let $(\chi^{(1)}_{\mu} \sqrt{\mathcal{K}_{\mu}})(D)$ be the multiplier of the symbol $(\chi^{(1)}_{\mu} \sqrt{{K}_{\mu}})(\xi)$. Then
		\begin{equation}\label{Product chi K}
			\| (\chi^{(1)}_{\mu} \sqrt{\mathcal{K}_{\mu}})(D)f \|_{L^2} \lesssim_{\beta} \| f \|_{L^2},
		\end{equation}
		and
		\begin{equation}\label{Commutator chi K}
			\| [(\chi^{(1)}_{\mu} \sqrt{\mathcal{K}_{\mu}})(D),f ]\partial_x g \|_{L^2} \lesssim_{\beta} \|f\|_{H^{s}} \| g\|_{L^2}.
		\end{equation}

	\item We define the symbol 
	\begin{align}\label{Sigma 1/2}
		\sigma_{\mu, \frac{1}{2}}(D) 
		& : =  \bigg{(}  \frac{1}{\sqrt{\mu} |D|} + \beta \sqrt{\mu}  |D| \bigg{)}^{\frac{1}{2}}.
	\end{align}
	%
	%
	%
	%	and let $$\tilde{\sigma}_{\mu, \frac{1}{2}}(D) : = 	\mu^{-\frac{1}{4}}\sigma_{\mu, \frac{1}{2}}(D) .$$
	%
	%
	%
	Then  
	\begin{equation}\label{Product sigma 1/2}
		\|(\chi^{(2)}_{\mu}  \sigma_{\mu, \frac{1}{2}})(D)f \|_{L^2}  \lesssim_{\beta}
		\|f\|_{L^2} +\mu^{\frac{1}{4}} \|  D^{\frac{1}{2}}  f\|_{L^2} 
	\end{equation}
	and %there holds,
	\begin{equation}\label{Commutator D}
		\|[  (\chi^{(2)}_{\mu}  \sigma_{\mu, \frac{1}{2}})(D), f] \partial_x g\|_{L^2} \lesssim_{\beta}  \mu^{\frac{1}{4}} \|f\|_{H^{s}}  \| g\|_{H^{\frac{1}{2}}}.
	\end{equation}

	 \item Lastly, we define the symbol $\sigma_{\mu,0}(D)$ to be
	\begin{equation}\label{sigmaDiff}
		\sigma_{\mu,0}(D) : =  \bigg{(}  \frac{1}{\sqrt{\mu} |D|} +  \beta \sqrt{\mu} |D| - \mathcal{K}_{\mu}(D) 
		\bigg{)}^{\frac{1}{2}}.
	\end{equation}
    Then 
    \begin{equation}\label{Product sigma}
    	\| (\chi^{(2)}_{\mu}\sigma_{\mu,0})(D)f \|_{L^2} \lesssim_{\beta} \| f \|_{L^2}
    \end{equation}
    and 
	\begin{equation}\label{Commutator sigma}
		\|[  (\chi^{(2)}_{\mu}\sigma_{\mu,0})(D), f] \partial_x g\|_{L^2} \lesssim_{\beta}  \|f\|_{H^{s}} \| g\|_{L^2}.
	\end{equation}

\end{itemize}	
\end{lemma}
The proof is postponed to Appendix \ref{A3}, where we also will prove the following commutator estimates at the $L^2(\R)-$level:

\begin{lemma}\label{T_mu and J_mu Commutator L2} Let $s> \frac{3}{2}$, $\mu \in (0,1)$ and $f,g \in \mathscr{S}(\R)$.
	
	\begin{itemize}
		\item For the composition of $\sqrt{T_{\mu}}(D)$ and $J^{\frac{1}{2}}_{\mu}$ there holds,
		\begin{equation}\label{L2 Commutator T_mu J_mu}
			\|[\sqrt{T_{\mu}}(D)J^{\frac{1}{2}}_{\mu}, f]\partial_x g\|_{L^2} \lesssim \|f\|_{H^{s}}\|g\|_{L^2}.
		\end{equation}
		\item While for the usual Bessel potential there holds,
		\begin{equation}\label{Bessel Commutator T_mu J}
			\|[\sqrt{T_{\mu}}(D)J^s, f]\partial_x g\|_{L^2} \lesssim \|f\|_{H^{s}}\|J^sg\|_{L^2}.
	\end{equation}
	
	\item Similarly, when the operator $J^s$ is the identity we have
	\begin{equation}\label{Bessel Commutator T_mu J=1}
		\|[\sqrt{T_{\mu}}(D), f]\partial_x g\|_{L^2} \lesssim \|f\|_{H^{s}}\|g\|_{L^2}.
	\end{equation}

		\item The derivative of the following commutator satisfies
		\begin{equation}\label{Commutator dx T}
			\|\partial_x[\sqrt{\mathcal{T}_{\mu}}(D), f]g \|_{L^2}
			\lesssim
			\|f\|_{H^{s}}\|g\|_{L^2}.
		\end{equation}

		\item Lastly, we can commute $J^{\frac{1}{2}}_{\mu}$ by
		\begin{equation}\label{L2 commutator J_mu}
			\|[J^{\frac{1}{2}}_{\mu}, f] \partial_xg \|_{L^2} \lesssim \|f\|_{H^{s}} 	\|J^{\frac{1}{2}}_{\mu}g\|_{L^2}.
		\end{equation}
	\end{itemize}

\end{lemma}

\subsection{Classical estimates}

Before turning to the proof of the energy estimates, we state some necessary results that will also be used throughout the paper. First, recall the embeddings (see, for example \cite{LinaresPonce2014}).
\begin{lemma}[Sobolev embeddings]
	Let $f \in \mathscr{S}(\R)$ and $s\in (0,\frac{1}{2})$. Then $H^s(\R) \hookrightarrow L^p(\R)$ with $p = \frac{2}{1-2s}$, and there holds
	\begin{equation}\label{Sobolev embedding s small}
		\| f \|_{L^p} \lesssim \| D^sf\|_{L^2}.
	\end{equation}
	Moreover, In the case $s>\frac{1}{2}$, then $H^s(\R)$ is continuously embedded in $L^{\infty}(\R)$. 
	
\end{lemma}

We also will  use the Leibniz rule for the Riesz potential on multiple  occasions.

\begin{lemma}[Fractional Leibniz rule \cite{KenigPonceVega1993}]  Let $\sigma = \sigma_1 + \sigma_2 \in (0,1)$ with $\sigma_i\in[0,\sigma]$ and $p,p_1,p_2\in(1,\infty)$ satisfy $\frac{1}{p} = \frac{1}{p_1} + \frac{1}{p_2}$. Then, for $f,g \in \mathscr{S}(\R)$
	\begin{equation}\label{FracLeibnizRule}
		\|D^{\sigma}(fg) - fD^{\sigma}g - g D^{\sigma}f \|_{L^p} \lesssim \| D^{\sigma_1} f \|_{L^{p_1}} \|D^{\sigma_2}g\|_{L^{p_2}}.
	\end{equation}
	Moreover, the case $\sigma_2 = 0$, $p_2 = \infty$ is also allowed. 
\end{lemma}

Finally, we recall the following results for the Bona-Smith argument (provided in the classical paper \cite{BonaSmith1975}) on the multiplier $\varphi_{\delta}(D)$ defined by:
\begin{Def}\label{regularisation}
Let $\varphi \in \mathscr{S}(\mathbb{R})$ such that $\int \varphi = 1$ and for $\delta>0$ define the
regularization operators $\varphi_{\delta}(D)$ in frequency by
\begin{align*}
    \forall f \in L^2(\mathbb{R}), \quad \forall \xi \in \mathbb{R}, \quad \widehat{\varphi_{\delta} f}(\xi) : = \varphi(\delta \xi) \hat{f}(\xi), 
\end{align*}
where $\varphi$ is a real valued and $\varphi (0)  = 1$. 
\end{Def}
\noindent
We give the version of the regularization estimates as presented in \cite{LinaresPonce2014} (Proposition $9.1$).
\begin{prop}\label{Rate of Decay in norm} Let $s>0$, $\delta>0$ and $f \in \mathscr{S}(\mathbb{R})$. Then
%\begin{itemize}
%	\item For all  $\alpha >0$ and as $\delta \searrow 0$, there holds
	%
	%
	%
	\begin{equation}\label{reg 3}
		\|\varphi_{\delta} (D) f\|_{H^{s+\alpha}} \lesssim \delta^{-\alpha}\| f\|_{H^s}, \quad \forall \alpha > 0,
	\end{equation}
	and
	\begin{equation}\label{reg 2}
		\|\varphi_{\delta} (D) f - f\|_{H^{s-\beta}} \lesssim \delta^{\beta} \|f\|_{H^s}, \quad \forall \beta \in [0,s].
	\end{equation}
	Moreover, there holds
%	\item For all $\beta \in [0,s]$  and as $\delta \searrow 0$, there holds
	%
	%
	%
	\begin{equation}\label{reg 4}
		\|\varphi_{\delta}(D)  f - f\|_{H^{s-\beta}} \underset{\delta \rightarrow 0}{=}o(\delta^{\beta} ), \: \: \: \quad \forall \beta \in [0,s].
	\end{equation}
%\end{itemize}
\end{prop}

\section{A priori estimates}\label{Energy est}

In this section, we give  \textit{a priori} estimates for solutions of the three systems \eqref{full dispersion}, \eqref{Whitham Boussinesq}, and \eqref{2nd Whitham Boussinesq}.

\subsection{Estimates for  system \eqref{full dispersion}}

As noted in the introduction, we  revisit the energy estimate in \cite{Wang2020}  to keep track of the parameters $\beta, \ve$ and $\mu$. For simplicity, we adopt the notation $\bold{U} = (\eta, u)^T = \ve (\zeta, v)^T$, where we write  \eqref{full dispersion} on the compact form:
\begin{equation}\label{Eq 1 M}
	\partial_t \bold{U} + M(\bold{U},D) \bold{U} = \bold{0},
\end{equation}
with
\begin{equation}\label{M 1}
	M(\bold{U},D) = 
	\begin{pmatrix}
		u \partial_x  & (\mathcal{K}_{\mu}(D) + \eta)\partial_x \\
		\partial_x & u \partial_x
	\end{pmatrix}.
\end{equation}
Also, we simplify the notation for the energy given in Definition \ref{Def Energy} by introducing the symmetrizer
\begin{equation}\label{S: symmetrizer}
	Q(\bold{U},D) =  Q^{(1)}(\bold{U},D)+ Q^{(2)}(\bold{U},D)
	=
	\begin{pmatrix}
		1 & 0 \\
		0 &  \eta 
	\end{pmatrix}
	+ 
	\begin{pmatrix}
		0 & 0 \\
		0 & \mathcal{K}_{\mu}(D)
	\end{pmatrix}.
\end{equation}
Then the energy given in Definition \ref{Def Energy} can be rewritten as
\begin{equation*}
	E_s( \bold{U},D) = \big{(} J^s \bold{U}, Q(\bold{U}, D)J^s \bold{U} \big{)}_{L^2}.
\end{equation*}

\begin{prop}\label{Energy fully disp Boussinesq} Let $s> 2$, $\ve, \mu \in (0,1) $ and $ (\eta,u) =  \ve (\zeta, v) \in C([0,T_0]; V^s_{\sqrt{\mu}}(\mathbb{R}) )$ be a solution to \eqref{Eq 1 M} on a time interval $[0,T_0]$ for some $T_0 >0 $.  Moreover, assume there exist $h_0 \in (0,1)$ and $h_1>0$ such that 
	\begin{equation}\label{cond. sol.}
		 h_0 - 1 \leq \eta(x,t), \quad   \forall (x,t) \in \mathbb{R} \times [0,T_0] \quad  \text{and} \quad \sup\limits_{t\in [0,T_0]}\|(\eta, u)\|_{H^s\times H^s} \leq h_1,
	\end{equation}
	when $\beta  \geq  \frac{1}{3}$, and that
	\begin{equation}\label{cond. sol. beta}
			  -\frac{\beta}{2} \leq \eta(x,t), \quad   \forall (x,t) \in \mathbb{R} \times [0,T_0] \quad  \text{and} \quad \sup\limits_{t\in [0,T_0]} \|(\eta, u)\|_{H^s\times H^s} \leq h_1,
	\end{equation}
	when $0<  \beta < \frac{1}{3}$. 
	
	Then, for the energy given in Definition  \ref{Def Energy} and $c_{\beta}^{i}$ defined by \eqref{Constant beta},
	\begin{equation}\label{Energy full dispersion}
		\frac{d}{dt} E_s(\bold{U}) \leq c_{\beta}^2\:  \big{(}E_s(\bold{U})\big{)}^{\frac{3}{2}},
	\end{equation}
	for all $0<t<T_0$, and
	\begin{equation}\label{equiv. full dispersion}
	 	c_{\beta}^1\|(\eta, u)\|^2_{V^s_{\sqrt{\mu}}} \leq E_s(\bold{U}) \leq c_{\beta}^2   \|(\eta, u)\|^2_{V^s_{\sqrt{\mu}}},
	\end{equation}
	for all $0<t<T_0$.
%	Furthermore, for $0<  \beta < \frac{1}{3}$ the estimates above will still remain true if $\eta$ satisfies the $\beta-$dependent surface condition \eqref{nonCavitationSurfaceTension} for all time. \\
	%
	%
	%
\end{prop}

\begin{remark}
	Note that we aim to prove \eqref{Energy full dispersion} with power $\frac{3}{2}$ on the right-hand side. This result will prove essential in getting the time of existence $T \sim \frac{1}{\ve}$ in the proof of Theorem \ref{Well-posedness long time full dispersion}. One should also note that if we have \eqref{equiv. full dispersion}, then it is enough to show 
	\begin{equation*}
		\frac{d}{dt} E_s(\bold{U}) \lesssim_{\beta} \|(\eta,u) \|^3_{V^s_{\sqrt{\mu}}},
	\end{equation*}
	to obtain \eqref{Energy full dispersion}. With this in mind, in the proof of the proposition, we will repeatedly use assumption \eqref{cond. sol.}$-$\eqref{cond. sol. beta}  to discard higher powers in the norm of the solution than $3$. Meaning the terms of form  $\|(\eta, u)\|_{V^s_{\sqrt{\mu}}}^{3+n}$ for $n \in \N$ will be bounded by $\|(\eta, u)\|_{V^s_{\sqrt{\mu}}}^{3}$ since this seems to be the best we can hope for when using the current method. 
\end{remark}

\begin{proof}[Proof of Proposition \ref{Energy fully disp Boussinesq}]
We first prove estimate \eqref{equiv. full dispersion} in the case $\beta \geq 1/3$.  By definition, we have  that
\begin{align*}
	E_s(\bold{U}) = \|J^s \eta \|_{L^2}^2 + \big( J^s u,( \mathcal{K}_{\mu}(D) + \eta)J^su \big)_{L^2}.
\end{align*}
Thus, as a result of the non-cavitation condition \eqref{cond. sol.} and the estimate \eqref{estimate K in Hs}, there holds
\begin{align*}
	\big( J^s u,( \mathcal{K}_{\mu}(D) + \eta) J^su \big)_{L^2}  
	& \geq 
	\frac{h_0}{2}\|u\|_{H^s}^2 +c\sqrt{\mu} \| D^{\frac{1}{2}}u \|_{H^{s}}^2.
\end{align*}
The reverse inequality holds for any $\beta>0$ and is a consequence of \eqref{estimate K in Hs}, Hölder's inequality, the Sobolev embedding with $s>\frac{3}{2}$, and conditions \eqref{cond. sol.}$-$\eqref{cond. sol. beta}. Indeed, we observe that%Hence, we have the equivalence, and from \eqref{energy with u},  we deduce the energy estimate \eqref{Energy full dispersion}.
\begin{equation*}
	E_s(\bold{U}) \leq \| \eta \|_{H^s}^2 + \|\sqrt{\mathcal{K}_{\mu}}(D) u\|_{H^s}^2 + \|\eta\|_{L^{\infty}}\|u\|_{H^s}^2\leq c_{\beta} \|(\eta,u)\|_{V^s_{\sqrt{\mu}}}^2.
\end{equation*}

In the case $0<\beta<\frac{1}{3}$, we impose the $\beta-$dependent surface condition \eqref{cond. sol. beta}, leaving less to be absorbed for the coercivity and in conjunction with \eqref{estimate K Low}. This implies
\begin{align*}
	\big( J^s u,( \mathcal{K}_{\mu}(D) + \eta) J^su \big)_{L^2}   \geq  \frac{\beta}{2} \|u\|^2_{H^s} + c\sqrt{\mu} \| D^{\frac{1}{2}}u \|_{H^{s}}^2.
\end{align*}	
As a consequence, we have that \eqref{equiv. full dispersion} is established for all $\beta>0$.

Next, we prove \eqref{Energy full dispersion}. By using \eqref{Eq 1 M} and the fact that $Q(\bold{U},D)$ is self-adjoint, we compute
\begin{align*}
	\frac{1}{2}\frac{d}{dt} E_s(\bold{U})
	& = \big(J^s \partial_t \bold{U}, Q(\bold{U},D)J^s \bold{U} \big)_{L^2} +\frac{1}{2}\big(J^s  \bold{U}, (\partial_t Q(\bold{U},D)) J^s \bold{U} \big)_{L^2} 
	\\
	& = 
	-
	\big(J^s M(\bold{U},D) \bold{U} , Q(\bold{U},D) J^s \bold{U}\big)_{L^2}
	+
	\frac{1}{2}\big( J^s \bold{U}, ( \partial_t Q(\bold{U},D)) J^s \bold{U}\big)_{L^2}
	\\
	& = : - I + II.
\end{align*}

\noindent
\underline{Control of $I$}. We may write
\begin{align*}
	I
	& = 
	\big([J^s , M(\bold{U},D)]\bold{U}, Q^{(1)}(\bold{U},D) J^s \bold{U}\big)_{L^2}
	+
	 \big( Q^{(1)}(\bold{U},D) M(\bold{U},D)J^s\bold{U},  J^s \bold{U}\big)_{L^2}\\
	& \hspace{0.4cm} 
	+
	 \big(J^s M(\bold{U},D) \bold{U} , Q^{(2)}(\bold{U},D)J^s \bold{U}\big)_{L^2}
	\\ 
	& =: I_1 + I_2 + I_3.
\end{align*}
\\ 
\textit{Control of $I_1$.} It follows from the Cauchy-Schwarz inequality that
\begin{equation*}
	| I_1 | \leq \|[J^s, M(\bold{U},D)]\bold{U}\|_{L^2} \| Q^{(1)}(\bold{U},D) J^s \bold{U}\|_{L^2}.
\end{equation*}
The second term is easily treated, %Sobolev embedding and for the last terms use the Leibniz rule:
\begin{align*}
	\| Q^{(1)}(\bold{U},D) J^s \bold{U} \|_{L^2}  
	\lesssim
	 \| J^s \eta \|_{L^2} + \|\eta\|_{L^\infty} \|J^s u\|_{L^2} 
	\lesssim 
	\|(\eta,u)\|_{V^s_{\sqrt{\mu}}},
\end{align*}
by Hölder's inequality, the Sobolev embedding with $s>\frac{1}{2}$, and assumption \eqref{cond. sol.}. Furthermore, using the Kato-Ponce commutator estimate  \eqref{K-P} yields
\begin{align*}
	 \|[J^s, M(\bold{U},D)]\bold{U}\|_{L^2} 
	 & \leq 
	  \|[J^s,u] \partial_x \eta\|_{L^2} 
	  +
	  \|[J^s,\eta] \partial_x u\|_{L^2} 
	  +
	   \|[J^s,u] \partial_x u\|_{L^2} \\
	   & \leq \| \eta\|_{H^s} \| u \|_{H^s} + \|u\|_{H^s}^2
	   \\
	   &\leq
	   \| (\eta,u)\|_{V^s_{\sqrt{\mu}}}^2
	   .
\end{align*}
The desired bound on $I_1$ follows:
\begin{equation*}
	|I_{1}|  \lesssim \|(\eta, u)\|_{V^s_{\sqrt{\mu}}}^3.
\end{equation*}
\noindent
\textit{Control of  $I_2 + I_3$.}  First note that $(a_{ij}) = Q^{(1)}(\bold{U},D) M(\bold{U},D)$ is given by,
\begin{align*}
	%A(\bold{U},D) 
	(a_{ij})
	= 
	\begin{pmatrix}
	u \partial_x & (\mathcal{K}_{\mu}(D) + \eta)\partial_x %+ u^2\partial_x 
	 \\
	 \eta \partial_x %+
	 %   u^2 \partial_x 
	 &  \eta u \partial_x  %+ u(\mathcal{K}_{\ve}(D) + \eta)\partial_x 
	\end{pmatrix}.
\end{align*}
We must estimate each piece below,
\begin{align*}
	& \big(Q^{(1)}(\bold{U},D)  M(\bold{U},D)J^s\bold{U}, J^s \bold{U}\big)_{L^2} 
	\\
	&
	= 
	\big( a_{11} J^s \eta, J^s \eta\big)_{L^2}
	+
	\big( a_{12} J^s u, J^s \eta\big)_{L^2}
	+
	\big( a_{21}J^s \eta, J^s u \big)_{L^2}
	+
	\big( a_{22} J^s u , J^s u \big)_{L^2}
	\\ 
	& =: A_{11} + A_{12} + A_{21} + A_{22}.
\end{align*}
As we will shortly see, $A_{12} + A_{21}$ needs to be compensated by $B_{21}$, that is defined  by the remaining part:
\begin{align*}
	 \big( J^s M(\bold{U},D)\bold{U}, Q^{(2)}(\bold{U},D)  J^s \bold{U}\big)_{L^2} 
	& = 
	\big( \partial_x J^s \eta, \mathcal{K}_{\mu}(D) J^s u \big)_{L^2}
	+
	\big( J^s (u \partial_x u) , \mathcal{K}_{\mu}(D)J^s u \big)_{L^2}
	\\
	& = :
	B_{21} + B_{22},
\end{align*}
while $B_{22}$ is the price we pay for symmetry. \\

\noindent
\textit{Control of  $A_{11}$.} Integration by part and the Sobolev embedding yields
\begin{align*}
	|A_{11}| & \leq  \frac{1}{2} \big |\big{(} \partial_x uJ^s\eta,J^s \eta \big{)}_{L^2}\big | \leq\frac{1}{2}\|\partial_x u\|_{L^\infty} \|\eta\|_{H^s}^2  \lesssim \|(\eta,u)\|_{V^s_{\sqrt{\mu}}}^3.
\end{align*}

\noindent
\textit{Control of  $A_{12} + A_{21} + B_{21}$.}   By definition, consideration is given to the expression
\begin{align*}
	A_{12} + A_{21} +  B_{21}
	& = 
	\big((\mathcal{K}_{\mu}(D) + \eta)\partial_xJ^s u , J^s \eta \big)_{L^2}
	+
	\big((\mathcal{K}_{\mu}(D) + \eta)  \partial_x J^s \eta , J^s u \big)_{L^2}.
\end{align*}
Observe, after integration by parts that
\begin{align*}
	A_{12}
	& =
	-
	\big( J^s u, (\mathcal{K}_{\mu}(D) + \eta) \partial_x J^s \eta \big)_{L^2}  
	-
	 \big( J^s u, \partial_x \eta J^s \eta \big)_{L^2}.
\end{align*}
The first term cancels with $(A_{21} + B_{21})$, while the Sobolev embedding easily controls the remaining part,
\begin{equation*}
	| \big( J^s u, \partial_x \eta J^s \eta \big)_{L^2}  | \leq \| \partial_x \eta \|_{L^{\infty}} \| \eta \|_{H^s} \| u \|_{H^s} \lesssim \| (\eta, u) \|_{V^s_{\sqrt{\mu}}}^3.
\end{equation*}

\noindent
\textit{Control of  $A_{22} $.} We simply use integration by parts as above together with \eqref{cond. sol.}$-$\eqref{cond. sol. beta} to deduce
\begin{equation*}
	|A_{22}| \leq |	\big(  \eta u  \partial_x J^s u , J^s u \big)_{L^2} | \leq c_{\beta}^2\| (\eta, u ) \|_{V^s_{\sqrt{\mu}}}^3.
\end{equation*}
\textit{Control of  $B_{22}$.} We observe, after integrating by parts that
\begin{align*}
	B_{22}
	& =
	\big( J^s( u \partial_x u) ,\mathcal{K}_{\mu}(D)J^s u \big)_{L^2}
	\\
	& = 
	\big( [\sqrt{\mathcal{K}_{\mu}}(D)J^s, u]\partial_x u,\sqrt{\mathcal{K}_{\mu}}(D)J^s u \big)_{L^2} 
	-\frac{1}{2}
	\big( (\partial_x u) \sqrt{\mathcal{K}_{\mu}} (D)J^s u, \sqrt{\mathcal{K}_{\mu}}(D)J^s u\big)_{L^2}.
\end{align*}
Thus, we deduce by using  Hölder's inequality, estimates  \eqref{Commutator low freq} and  \eqref{estimate K in Hs} that
\begin{align*}
	|B_{22} | \leq c_{\beta}^2 \|(\eta,u)\|_{V_{\sqrt{\mu}}^s}^3.
\end{align*}

\noindent
\underline{Control of  $II$.} First we claim that $\|\mathcal{K}_{\mu}(D)\partial_x u \|_{L^{\infty}} \lesssim_{\beta}\| (\eta, u)\|_{V_{\sqrt{\mu}}^s}$ for $s>2$. Indeed, it follows from \eqref{Pointwise est.} and the Sobolev embedding $H^{\frac{1}{2}^+}(\R)\hookrightarrow L^{\infty}(\R)$ that
\begin{align}\notag 
	\|\mathcal{K}_{\mu}(D)\partial_x u \|_{L^{\infty}} 
	&
	\leq
	\|\partial_x u\|_{H^{s-\frac{3}{2}}}  +\beta \sqrt{\mu}  \| D^1 \partial_x u \|_{H^{s-\frac{3}{2}}}
	\\ 
	& 
	\leq c_{\beta}^2 (\| u \|_{H^s} + \sqrt{\mu}  \| D^{\frac{1}{2}} u \|_{H^s}).  \label{Dependence in beta}
\end{align}
Then we observe by using equation  \eqref{Eq 1 M} yields,
\begin{align*}
	II
	& = 
	\big( J^s u,  (\partial_t \eta) J^s u \big)_{L^2} 
	=
	-\big( J^s u, (\mathcal{K}_{\mu}(D)\partial_x u) J^s u \big)_{L^2} - \big(  J^s u,  (\partial_x(\eta u))J^s u \big)_{L^2}.
\end{align*}
Consequently, the desired estimate follows from Hölder's inequality, the Sobolev embedding, and the above claim that,
\begin{align}\label{time der}\
	II
  \lesssim
	 \|\mathcal{K}_{\mu}(D)\partial_x u \|_{L^{\infty}} \|u\|_{H^s}^2 + \|\partial_x (\eta u ) \|_{L^{\infty}} \| u\|_{H^s}^2
\lesssim_{\beta} 
	 \| (\eta, u)\|_{V_{\sqrt{\mu}}^s}^3.
\end{align}

Adding together all the estimates, combined with \eqref{equiv. full dispersion} yields,
\begin{equation*}%\label{energy with u}
	\frac{d}{dt} E_{s}(\bold{U}) \leq   c_{\beta}^2\:  \big{(}E_s(\bold{U})\big{)}^{\frac{3}{2}},
\end{equation*} 
and completes the proof of Proposition \ref{Energy fully disp Boussinesq}.

\end{proof}

%
%
%
%Next, we will repeat the same type of estimates  for \eqref{Whitham Boussinesq}.

\subsection{Estimates for system \eqref{Whitham Boussinesq}}
As in the former subsection we define $\bold{U} = (\eta, u)^T = \ve (\zeta, v)^T$ and we write the system on a compact form:
\begin{equation}\label{Eq M for T_mu}
	\partial_t \bold{U} + \mathcal{M}(\bold{U},D) \bold{U} = \bold{0},
\end{equation}
with
\begin{equation}\label{Math M}
	\mathcal{M}(\bold{U},D) = 
	\begin{pmatrix}
		u \partial_x  & (1+ \eta)\partial_x \\
		\mathcal{T}_{\mu}(D)\partial_x & u \partial_x
	\end{pmatrix}.
\end{equation}
We define the symmetrizer associated to \eqref{Eq M for T_mu} to be
\begin{equation}\label{Sym Q}
	\mathcal{Q}(\bold{U}, D) =
	\begin{pmatrix}
		\mathcal{T}_{\mu}(D) & 0 \\
		0 & 1 + \eta	
	\end{pmatrix}.
\end{equation}
Then the energy given in Definition \ref{Energy Whitham Boussinesq} can be written as
\begin{equation}\label{Energy Whitham Boussinesq Q}
	\mathcal{E}_s(\bold{U}) = \big{(}J^sJ^{\frac{1}{2}}_{\mu} \bold{U}, \mathcal{Q}(\bold{U}, D) J^sJ^{\frac{1}{2}}_{\mu} \bold{U}\big{)}_{L^2},
\end{equation}
%In the proof of Proposition \ref{Compactness L2}, some additional terms will appear when looking at the difference between two solutions due to the nonlinearity. Though, the estimation is similar and is demonstrated below. 
%
%
%
and the \textit{a priori estimate} for \eqref{Whitham Boussinesq} is stated in the following proposition.

\begin{prop}\label{A priori Whitham Boussinesq} Let $s> \frac{3}{2}$, $\ve, \mu \in (0,1)$ and $(\eta,u) =  \ve (\zeta, v) \in C([0,T_0]; V^s_{\sqrt{\mu}}(\mathbb{R}) )$ be a solution to \eqref{Whitham Boussinesq} on a time interval $[0,T_0]$ for some $T_0 >0 $.  Moreover, assume there exist $h_0 \in (0,1)$ and $h_1>0$ such that 
	\begin{equation}\label{cond. sol. 2}
		h_0 - 1 \leq \eta(x,t), \quad   \forall (x,t) \in \mathbb{R} \times [0,T_0] \quad  \text{and} \quad \sup\limits_{t\in [0,T_0]} \|(\eta, u)\|_{H^s\times H^s} \leq h_1.
	\end{equation}

 	Then, for the energy given in Definition  \ref{Energy Whitham Boussinesq}, there holds
	\begin{equation}\label{Energy Est Whitham Boussinesq}
		\frac{d}{dt} \mathcal{E}_s(\bold{U}) \lesssim  \big{(}\mathcal{E}_s(\bold{U})\big{)}^{\frac{3}{2}},
	\end{equation}
	for all $0<t<T_0$, and
	\begin{equation}\label{Equivalence Whitham Boussinesq}
		\|(\eta, u)\|^2_{V^s_{\sqrt{\mu}}} \lesssim \mathcal{E}_s(\bold{U}) \lesssim  \|(\eta, u)\|^2_{V^s_{\sqrt{\mu}}},
	\end{equation}
	for all $0<t<T_0$.
	
	%	Furthermore, for $0<  \beta < \frac{1}{3}$ the estimates above will still remain true if $\eta$ satisfies the $\beta-$dependent surface condition \eqref{nonCavitationSurfaceTension} for all time. \\
	%
	%
	%
\end{prop}

\begin{proof}[Proof of Proposition \ref{A priori Whitham Boussinesq}] We begin by proving \eqref{Equivalence Whitham Boussinesq}.  By Definition \eqref{Energy Whitham Boussinesq} of the energy, the non-cavitation condition \eqref{cond. sol. 2}, \eqref{J_mu}, and  \eqref{Equiv sqrt T}  we obtain the lower bound
	\begin{align*}
		\mathcal{E}_{s}(\bold{U}) 
		& = 
		\| \sqrt{\mathcal{T}_{\mu}}(D)J^{\frac{1}{2}}_{\mu} \eta  \|_{H^s}^2 
		+
		\big{(}
		J^{\frac{1}{2}}_{\mu} J^s u, (1+\eta)J^{\frac{1}{2}}_{\mu} J^s u 
		\big{)}_{L^2} 
		\\
		& 
		\geq  c\|  \eta  \|_{H^s}^2 
		+
		h_0\|J^{\frac{1}{2}}_{\mu}  u\|_{H^s}^2
		\\
		& 
		\geq  c\|  \eta  \|_{H^s}^2 
		+
		c \cdot	h_0 (\|  u\|_{H^s}^2 + \sqrt{\mu} \|D^{\frac{1}{2}} u\|_{H^s}^2),
	\end{align*}
	for some $c>0$. The reverse inequality follows by the estimates \eqref{Equiv sqrt T}, \eqref{J_mu}, Hölder's inequality, the Sobolev embedding, and  \eqref{cond. sol. 2}: 
	\begin{align*}
		\mathcal{E}_{s}(\bold{U}) \leq 	\| \sqrt{\mathcal{T}_{\mu}}(D)J^{\frac{1}{2}}_{\mu} \eta  \|_{H^s}^2 
		+
		\| J^{\frac{1}{2}}_{\mu} u\|_{H^s}^2 + \|\eta\|_{L^{\infty}} \|J^{\frac{1}{2}}_{\mu} u \|^2_{H^s}
		\lesssim \| (\eta,u)\|_{V^s_{\sqrt{\mu}}}^2.
	\end{align*}

	Next, we prove \eqref{Energy Est Whitham Boussinesq}. There follows by using \eqref{Eq M for T_mu} and the self-adjointness of $\mathcal{Q}(U,D)$ that
	\begin{align*}
		\frac{1}{2}\frac{d}{dt} \mathcal{E}_s(\bold{U})
		& =
		-
		\big(J^{s}  J_{\mu}^{\frac{1}{2}} \mathcal{M}(\bold{U},D)  \bold{U} ,  \mathcal{Q}(\bold{U},D)J^{s} J_{\mu}^{\frac{1}{2}} \bold{U}\big)_{L^2}
		\\
		& \hspace{1cm}+
		\frac{1}{2}\big( J^{s} J_{\mu}^{\frac{1}{2}} \bold{U},  (\partial_t \mathcal{Q}(\bold{U},D)) J^{s} J_{\mu}^{\frac{1}{2}} \bold{U}\big)_{L^2}
		\\
		& = :- \mathcal{I} + \mathcal{II}.
	\end{align*}
	\underline{Control of $\mathcal{I}$.} By definition of \eqref{Energy Whitham Boussinesq Q} we decompose $\mathcal{I}$ in four pieces,
	\begin{align*}
		\mathcal{I}
		& = 
		\big(J^s J^{\frac{1}{2}}_{\mu}(u \partial_x \eta) ,  J^s J^{\frac{1}{2}}_{\mu}\mathcal{T}_{\mu}(D) \eta\big)_{L^2}
		+
		\big(J^s J^{\frac{1}{2}}_{\mu}\big{(}(1+\eta) \partial_x u\big{)} ,  J^s J^{\frac{1}{2}}_{\mu}\mathcal{T}_{\mu}(D) \eta\big)_{L^2}
		\\
		& \hspace{0.4cm}
		+
		\big(J^s J^{\frac{1}{2}}_{\mu}\mathcal{T}_{\mu}(D) \partial_x \eta , (1 + \eta)J^s J^{\frac{1}{2}}_{\mu} u\big)_{L^2}
		+
		\big(J^s J^{\frac{1}{2}}_{\mu} u \partial_x u , (1 + \eta)J^s J^{\frac{1}{2}}_{\mu}u\big)_{L^2}
		\\
		& = : \mathcal{A}_{11} + \mathcal{A}_{12} + \mathcal{A}_{21} + \mathcal{A}_{22}.
	\end{align*}
\textit{Control of $\mathcal{A}_{11}$.} We aim to exploit symmetries, and we first write $\mathcal{A}_{11}$
as
\begin{align*}
	\mathcal{A}_{11} 
	& =
	\big([J^s J^{\frac{1}{2}}_{\mu}  \sqrt{\mathcal{T}_{\mu}} (D), u] \partial_x \eta , J^s J^{\frac{1}{2}}_{\mu} \sqrt{\mathcal{T}_{\mu}}(D)  \eta\big)_{L^2}
	\\
	& 
	\hspace{0.5cm}
	+
	\big(uJ^s J^{\frac{1}{2}}_{\mu} \sqrt{\mathcal{T}_{\mu}} (D)\partial_x \eta , J^s J^{\frac{1}{2}}_{\mu}  \sqrt{\mathcal{T}_{\mu}}(D)  \eta\big)_{L^2}
	\\
	& =: 	\mathcal{A}_{11}^1 + 	\mathcal{A}_{11}^2. 
\end{align*} 
The first term is treated by the commutator estimate \eqref{Commutator Js T} with $s>\frac{3}{2}$, the Cauchy-Schwarz inequality and \eqref{Equiv sqrt T}. Thus, there 
holds
\begin{equation*}
	|	\mathcal{A}_{11}^1 | \leq \|[J^s J^{\frac{1}{2}}_{\mu}  \sqrt{\mathcal{T}_{\mu}} (D), u] \partial_x \eta\|_{L^2} \| J^s J^{\frac{1}{2}}_{\mu} \sqrt{\mathcal{T}_{\mu}}(D)  \eta \|_{L^2} \lesssim \|u\|_{H^s}\|\eta\|_{H^s}^2.
\end{equation*}
Similar to previous estimates, we use integration by parts and exploit the symmetries of $\mathcal{A}_{11}^2$, then conclude by  \eqref{Equiv sqrt T}, and the Sobolev embedding $H^{\frac{1}{2}^+}(\R)\hookrightarrow L^{\infty}(\R)$ that
\begin{equation*}
	|\mathcal{A}_{11}| \leq \frac{1}{2} \big| \big((\partial_xu)J^s J^{\frac{1}{2}}_{\mu} \sqrt{\mathcal{T}_{\mu}} (D) \eta , J^s J^{\frac{1}{2}}_{\mu}  \sqrt{\mathcal{T}_{\mu}}(D)  \eta\big)_{L^2} \big{|}
	 \lesssim \|(\eta,u)\|_{V^s_{\sqrt{\mu}}}^3.
\end{equation*}
\textit{Control of $\mathcal{A}_{12} + \mathcal{A}_{21}$.} We first decompose $\mathcal{A}_{12}$ in two parts 
 \begin{align*}
 	\mathcal{A}_{12} & = 
 	\big{(}
 	[J^s J^{\frac{1}{2}}_{\mu}  \sqrt{\mathcal{T}_{\mu}} (D), \eta] \partial_x u, J^s J^{\frac{1}{2}}_{\mu}  \sqrt{\mathcal{T}_{\mu}} (D) \eta
 	\big{)}_{L^2}  
 	\\
 	& \hspace{0.5cm}+
 	\big{(}(1+\eta)J^s J^{\frac{1}{2}}_{\mu}  \sqrt{\mathcal{T}_{\mu}} (D)\partial_x u, J^s J^{\frac{1}{2}}_{\mu}  \sqrt{\mathcal{T}_{\mu}} (D)\eta \big{)}_{L^2}
 	\\ 
 	& = 
 	\mathcal{A}_{12}^1 + \mathcal{A}_{12}^2. 
 \end{align*}
We estimate $\mathcal{A}_{12}^1$ the same way we did for $\mathcal{A}_{11}^1$ and obtain
\begin{equation*}
	|\mathcal{A}_{12}^1| \lesssim \| u\|_{H^s}\| \eta\|_{H^s}^2.
\end{equation*}
For the second term, after integration by parts, we find
\begin{align*}
	\mathcal{A}_{12}^2 
	& = 
	-
	\big{(}(\partial_x\eta)J^s J^{\frac{1}{2}}_{\mu}  \sqrt{\mathcal{T}_{\mu}} (D)u, J^s J^{\frac{1}{2}}_{\mu}  \sqrt{\mathcal{T}_{\mu}} (D)\eta \big{)}_{L^2}
	\\
	& \hspace{0.5cm} 
	-
	 \big{(}(1+\eta)J^s J^{\frac{1}{2}}_{\mu}  \sqrt{\mathcal{T}_{\mu}} (D) u, J^s J^{\frac{1}{2}}_{\mu}  \sqrt{\mathcal{T}_{\mu}} (D)\partial_x\eta \big{)}_{L^2}
	\\ 
	& = 
	\mathcal{A}_{12}^{2,1} + \mathcal{A}_{12}^{2,2}.  
\end{align*}
By using the Sobolev embedding and \eqref{Equiv sqrt T}, we find that
\begin{equation*}
	|\mathcal{A}_{12}^{2,1}| \leq \|\partial_x \eta \|_{L^{\infty}}\| J^{\frac{1}{2}}_{\mu}  \sqrt{\mathcal{T}_{\mu}}u\|_{H^s} \| J^{\frac{1}{2}}_{\mu}  \sqrt{\mathcal{T}_{\mu}}\eta\|_{H^s}
	\lesssim
	\|u\|_{H^s}\|\eta\|_{H^s}^2.
\end{equation*}
On the other hand, we cannot estimate $\mathcal{A}_{12}^{2,2}$ on its own. We must therefore cancel it with $\mathcal{A}_{21}$. Observe
\begin{align*}
	\mathcal{A}_{21} 
	& = 
	\big{(} J^s J^{\frac{1}{2}}_{\mu}  \sqrt{\mathcal{T}_{\mu}} (D)\partial_x \eta, [\sqrt{\mathcal{T}_{\mu}}(D),\eta]J^s J^{\frac{1}{2}}_{\mu}  u\big{)}_{L^2}
	\\
	& \hspace{0.5cm}
	+
	\big{(}
	J^s J^{\frac{1}{2}}_{\mu}  \sqrt{\mathcal{T}_{\mu}} (D) \partial_x \eta, (1+\eta)J^s J^{\frac{1}{2}}_{\mu}  \sqrt{\mathcal{T}_{\mu}} (D) u
	\big{)}_{L^2}
	\\
	& = 
	\mathcal{A}_{21}^1 + \mathcal{A}_{21}^2. 
\end{align*}
First, by using integration by parts, the Cauchy-Schwarz inequality, \eqref{Equiv sqrt T}, \eqref{Commutator dx T} and \eqref{J_mu} we find that
\begin{align}\label{A_21^1}
	|\mathcal{A}_{21}^1|  = \| J^s J^{\frac{1}{2}}_{\mu}  \sqrt{\mathcal{T}_{\mu}} (D) \eta\|_{L^2}
	\| \partial_x[\sqrt{\mathcal{T}_{\mu}}(D),\eta]J^s J^{\frac{1}{2}}_{\mu}  u\|_{L^2}  
	\lesssim 
	\|(\eta, u) \|_{V^s_{\sqrt{\mu}}}^3.
\end{align}
On the other hand, we observe that $\mathcal{A}_{21}^2 = - \mathcal{A}_{21}^{2,2}$. We may therefore conclude that the sum satisfies:
\begin{equation*}
	|\mathcal{A}_{12} + \mathcal{A}_{21}|\lesssim \|(\eta, u )\|_{V^s_{\sqrt{\mu}}}^3.
\end{equation*}
\textit{Control of $\mathcal{A}_{22}$.} Similar to $\mathcal{A}_{11}$ we write the expression with the good commutator:
\begin{align*}
	\mathcal{A}_{22}
	& =
	\big([J^s J^{\frac{1}{2}}_{\mu}, u ] \partial_x u , (1 + \eta)J^s J^{\frac{1}{2}}_{\mu}u\big)_{L^2}
	+
	\big(uJ^s J^{\frac{1}{2}}_{\mu}  \partial_x u , (1 + \eta)J^s J^{\frac{1}{2}}_{\mu}u\big)_{L^2}
	\\
	& = 
	\mathcal{A}_{22}^1 + 	\mathcal{A}_{22}^2.
\end{align*}
Then use the  Cauchy-Schwarz inequality, \eqref{cond. sol. 2}, \eqref{Commutator J^sJ_mu} with $s > \frac{3}{2}$, and the Sobolev embedding  to get
\begin{equation*}
	| \mathcal{A}_{22}^1 | \lesssim \|[J^s J^{\frac{1}{2}}_{\mu}, u ] \partial_x u\|_{L^2}  (1+\|\eta \|_{L^{\infty}}) \|J^s J^{\frac{1}{2}}_{\mu}u\|_{L^2}\lesssim \|(\eta, u)\|_{V^s_{\sqrt{\mu}}}^3.
\end{equation*}
While for $\mathcal{A}_{22}^2$ we integrate by parts, apply the Sobolev embedding, and again bound each term by the $V^s_{\sqrt{\mu}}-$norm of $(\eta,u)$ to obtain that
\begin{equation*}
	|\mathcal{A}_{22}^2| \lesssim (\|\partial_x u\|_{L^{\infty}} + \| \partial_x \eta \|_{L^{\infty}} )\|J^{\frac{1}{2}}_{\mu}u\|_{H^s}^2 \lesssim \|(\eta,u)\|_{V^s_{\sqrt{\mu}}}^3.
\end{equation*}

Gathering all these estimates, we conclude that
\begin{equation}\label{Wh B one}
	|\mathcal{I}|\lesssim\|(\eta, u)\|_{V^s_{\sqrt{\mu}}}^3.
\end{equation}

\noindent 
\underline{Control of $\mathcal{II}$.} By defintion of \eqref{Sym Q} and  \eqref{Eq M for T_mu} we get that,
\begin{align*}
	\mathcal{II}
		& = 
		\big( J^{s}J^{\frac{1}{2}}_{\mu} u,  (\partial_t \eta) J^{s} J^{\frac{1}{2}}_{\mu} u \big)_{L^2} 
		\\
		&
		=
		-
		\big( J^{s}J^{\frac{1}{2}}_{\mu}  u, (\partial_x u) J^{s}J^{\frac{1}{2}}_{\mu} u \big)_{L^2} 
		-
		\big( J^{s}J^{\frac{1}{2}}_{\mu} u,  (\partial_x(\eta u) )J^{s}J^{\frac{1}{2}}_{\mu} u \big)_{L^2} 
\end{align*}
Then, by using Hölder's inequality, the Sobolev embedding, \eqref{cond. sol. 2} and \eqref{J_mu}, we deduce that
\begin{equation}\label{Wh B two}
	|\mathcal{II}|\lesssim (\|\partial_x u\|_{L^{\infty}} + \|\partial_x(\eta u )\|_{L^{\infty}})\|J^{\frac{1}{2}}_{\mu} u\|_{H^s}^2 \lesssim \|(\eta, u)\|_{V^s_{\sqrt{\mu}}}^3.
\end{equation}

Consequently, we may add \eqref{Wh B one} and  \eqref{Wh B two}, then apply \eqref{Equivalence Whitham Boussinesq} to conclude the proof of estimate \eqref{Energy Est Whitham Boussinesq}.

\end{proof}

\subsection{Estimates for  system \eqref{2nd Whitham Boussinesq}}

As in the former subsections we let $\bold{U} = (\eta, u)^T = \ve (\zeta, v)^T$ and write the system on the form
\begin{equation}\label{2nd Eq M for T_mu}
	\partial_t \bold{U} + \mathscr{M}(\bold{U},D) \bold{U} = \bold{0},
\end{equation}
with
\begin{equation}\label{2nd Math M}
	\mathscr{M}(\bold{U},D) = 
	\begin{pmatrix}
		\mathcal{T}_{\mu}(D)(u \partial_x \boldsymbol{\cdot} )& \partial_x+ \mathcal{T}_{\mu}(D)(\eta\partial_x \boldsymbol{\cdot} )\\
		\mathcal{K}_{\mu}(D)\partial_x & \mathcal{T}_{\mu}(D)(u \partial_x\boldsymbol{\cdot} )
	\end{pmatrix}.
\end{equation}
The  symmetrizer is defined by
\begin{equation}\label{2nd Sym Q}
	\mathscr{Q}(\bold{U}, D) =
	\begin{pmatrix}
		\mathcal{T}_{\mu}^{-1}(D)\mathcal{K}_{\mu}(D) & 0 \\
		0 & \mathcal{T}_{\mu}^{-1}(D)+ \eta	
	\end{pmatrix}.
\end{equation}
Then the energy given in Definition \ref{2nd Energy Whitham Boussinesq} can be written as
\begin{equation}\label{2nd Energy Whitham Boussinesq Q}
	\mathscr{E}_s(\bold{U}) = \big{(}J^s \bold{U}, \mathscr{Q}(\bold{U}, D) J^s \bold{U}\big{)}_{L^2},
\end{equation}
%In the proof of Proposition \ref{Compactness L2}, some additional terms will appear when looking at the difference between two solutions due to the nonlinearity. Though, the estimation is similar and is demonstrated below. 
%
%
%
and the \textit{a priori estimate} for \eqref{2nd Whitham Boussinesq} is stated in the following proposition.

\begin{prop}\label{2nd A priori Whitham Boussinesq} Let $s> \frac{3}{2}$, $\ve, \mu \in (0,1) $, $\beta>0$, and let $(\eta,u) =  \ve (\zeta, v) \in C([0,T_0]; X^s_{\beta,\mu}(\mathbb{R}) )$ be a solution to \eqref{2nd Whitham Boussinesq} on a time interval $[0,T_0]$ for some $T_0 >0 $.  Moreover, assume  that there exist $h_0 \in (0,1)$ and $h_1>0$ such that
	\begin{equation}\label{cond. sol. 3}
		h_0 - 1 \leq \eta(x,t), \quad   \forall (x,t) \in \mathbb{R} \times [0,T_0] \quad  \text{and} \quad \sup\limits_{t\in [0,T_0]} \|(\eta, 	u)\|_{H^s\times H^s} \leq h_1.
	\end{equation}

	Then, for the energy given in Definition  \ref{2nd Energy Whitham Boussinesq}, there holds,
	\begin{equation}\label{2nd Energy Est Whitham Boussinesq}
		\frac{d}{dt} \mathscr{E}_s(\bold{U}) \lesssim c_{\beta}^2 \big{(}\mathscr{E}_s(\bold{U})\big{)}^{\frac{3}{2}},
	\end{equation}
	and the energy is coercive:
	\begin{equation}\label{2nd Equivalence Whitham Boussinesq}
		\|(\eta, u)\|^2_{X^s_{\beta,\mu}} \lesssim \mathscr{E}_s(\bold{U}) \lesssim  \|(\eta, u)\|^2_{X^s_{\beta,\mu}}.
	\end{equation}
\end{prop}

\begin{proof}[Proof of Proposition \ref{2nd A priori Whitham Boussinesq}] We will first provide the coercivity estimate \eqref{2nd Equivalence Whitham Boussinesq}. By Definition \ref{2nd Energy Whitham Boussinesq} for the energy, the non-cavitation condition \eqref{cond. sol. 3} and \eqref{Inverse T_mu in Hs}  we obtain the lower bound
	\begin{align*}
		\mathscr{E}_{s}(\bold{U}) 
		& = 
		\| \langle \sqrt{\beta \mu } D \rangle \eta  \|_{H^s}^2 
		+
		\big{(}
		J^s u, (\mathcal{T}^{-1}_{\mu}(D)+\eta) J^s u 
		\big{)}_{L^2} 
		\\
		& 
		\geq   \|  \eta  \|_{H^s}^2 +\beta \mu  \|D^1\eta\|_{H^{s}}^2
		+
		\frac{h_0}{2}\| u\|_{H^s}^2  + c\sqrt{\mu} \|D^{\frac{1}{2}}u\|_{H^s}^2,
	\end{align*}
	for some $c>0$. The reverse inequality follows by the upper bound in \eqref{Inverse T_mu in Hs},  the Sobolev embedding   and   \eqref{cond. sol. 3}: 
	\begin{align*}
		\mathscr{E}_{s}(\bold{U}) 
		\leq 
		\| \langle \sqrt{\beta \mu } D \rangle \eta  \|_{H^s}^2 
		+
		\|  \mathcal{T}^{-\frac{1}{2}}_{\mu}(D)J^s u\|^2  +\| \eta\|_{L^{\infty}}\| J^s u\|^2 
		\lesssim
		\| (\eta,u)\|_{X^s_{\beta,\mu}}^2.
	\end{align*}

	We may now prove \eqref{2nd Energy Est Whitham Boussinesq}. To do so, we use \eqref{2nd Eq M for T_mu} and the self-adjointness of $\mathscr{Q}(\bold{U},D)$ to write
	\begin{align*}
		\frac{1}{2}\frac{d}{dt} \mathscr{E}_s(\bold{U})
		& =
		-
		\big(J^{s}   \mathscr{M}(\bold{U},D)  \bold{U} ,  \mathscr{Q}(\bold{U},D)J^{s}  \bold{U}\big)_{L^2}
		+
		%\\
		%& \hspace{1cm}+
		\frac{1}{2}\big( J^{s}  \bold{U},  (\partial_t \mathscr{Q}(\bold{U},D)) J^{s}  \bold{U}\big)_{L^2}
		\\
		& = : \mathscr{I} + \mathscr{II}.
	\end{align*}
	\underline{Control of $\mathscr{I}$.} By definition of \eqref{2nd Energy Whitham Boussinesq Q} we must estimate the following terms:
	\begin{align*}
		\mathscr{I}
		& = 
		\big(J^s (u \partial_x \eta) , \mathcal{K}_{\mu}(D) J^s  \eta\big)_{L^2}
		+
		\big(J^s \partial_x u+J^s\mathcal{T}_{\mu}(D) (\eta \partial_x u) , \mathcal{T}_{\mu}^{-1}(D)\mathcal{K}_{\mu}(D) J^s  \eta\big)_{L^2}
		\\
		& \hspace{0.4cm}
		+
		\big(J^s \mathcal{K}_{\mu}(D) \partial_x \eta , (\mathcal{T}_{\mu}^{-1}(D) + \eta)J^s u\big)_{L^2}
		+
		\big(J^s \mathcal{T}_{\mu}(D) (u \partial_x u) , (\mathcal{T}_{\mu}^{-1}(D) + \eta)J^s u\big)_{L^2}
		\\
		& = : \mathscr{A}_{11} + \mathscr{A}_{12} + \mathscr{A}_{21} + \mathscr{A}_{22}.
	\end{align*}
	\textit{Control of $\mathscr{A}_{11}$.} We rewrite $\mathscr{A}_{11}$ as
	\begin{align*}
		\mathscr{A}_{11} 
		& =
		\big([J^s \sqrt{\mathcal{K}_{\mu}} (D), u] \partial_x \eta , J^s  \sqrt{\mathcal{K}_{\mu}}(D)  \eta\big)_{L^2}
	%	\\
	%	& 
	%	\hspace{0.5cm}
		+
		\big(uJ^s  \sqrt{\mathcal{K}_{\mu}} (D)\partial_x \eta , J^s  \sqrt{\mathcal{K}_{\mu}}(D)  \eta\big)_{L^2}
		\\
		& =: 	\mathscr{A}_{11}^1 + 	\mathscr{A}_{11}^2. 
	\end{align*} 
	For the first term, we first note by interpolation and Young's inequality that
	\begin{equation}\label{interpolation}
		\sqrt{\beta}\mu^{\frac{1}{4}} \|\eta\|_{H^{s+ \frac{1}{2}}} \leq \| \eta \|_{H^s}^{\frac{1}{2}} 
		\big{(}\beta
		\sqrt{\mu}\| \eta\|_{H^{s+1}}
		\big{)}^{\frac{1}{2}} 
		\lesssim_{\beta} \|\eta\|_{H^s} + \beta\sqrt{\mu} \| D^1 \eta\|_{H^{s}},
	\end{equation}
	and thus $\mathscr{A}_{11}^1 $ is treated by the Cauchy-Schwarz inequality,  the commutator estimate \eqref{Commutator low freq} with $s>\frac{3}{2}$, \eqref{estimate K in Hs}, and \eqref{interpolation}:
	\begin{equation*}
		|	\mathscr{A}_{11}^1 | \lesssim c_{\beta}^2(\|u\|_{H^s} + \mu^{\frac{1}{2}} \|u\|_{H^{s+\frac{1}{2}}})(\|\eta\|_{H^s} +  \sqrt{\beta}\mu^{\frac{1}{4}} \| \eta \|_{H^{s+\frac{1}{2}}})^2 \lesssim c_{\beta}^2 \|(\eta,u)\|_{X^s_{\beta,\mu}}^3.
	\end{equation*}
	On the other hand, for $\mathscr{A}_{11}^2$ we conclude by integration by parts,   \eqref{estimate K in Hs}, \eqref{interpolation},  and the Sobolev embedding with $s>\frac{3}{2}$ that
	\begin{equation*}
		|\mathscr{A}_{11}|
		\lesssim 
		|\mathscr{A}_{11}^1|
		+
		\|\partial_xu\|_{L^{\infty}} \|  \sqrt{\mathcal{K}_{\mu}} (D) \eta\|_{H^s}
		\| \sqrt{\mathcal{K}_{\mu}}(D)  \eta\|_{H^s} 
		\lesssim
		c_{\beta}^2 \|(\eta,u)\|_{X^s_{\beta,\mu}}^3.
	\end{equation*}
	\textit{Control of $\mathscr{A}_{12} + \mathscr{A}_{21}$.} By using integration by parts we  write, 
	\begin{align*}
		\mathscr{A}_{12} 
		& = 
			\big(J^s( \eta \partial_x u) , \mathcal{K}_{\mu}(D) J^s  \eta\big)_{L^2}
			-
			\big(J^s  u , \mathcal{T}_{\mu}^{-1}(D)\mathcal{K}_{\mu}(D) J^s \partial_x  \eta\big)_{L^2}
			\\ 
			& = 
			\mathscr{A}_{12}^1 + \mathscr{A}_{12}^2. 
	\end{align*}
	For $\mathscr{A}_{12}^1$, observe
	\begin{align*}
			\mathscr{A}_{12}^1 
			& =
			\big([J^s, \eta] \partial_x u , \mathcal{K}_{\mu}(D) J^s  \eta\big)_{L^2}
			-
			\big( (\partial_x\eta) J^s  u , \mathcal{K}_{\mu}(D) J^s  \eta\big)_{L^2}
			-
			\big( \eta J^s  u , \mathcal{K}_{\mu}(D) J^s  \partial_x \eta\big)_{L^2}
			\\
			& = 
			\mathscr{A}_{12}^{1,1} + \mathscr{A}_{12}^{1,2} + \mathscr{A}_{12}^{1,3}. 
	\end{align*}
	Then the Kato-Ponce commutator estimate \eqref{Kato ponce}, the Sobolev embedding, and the pointwise estimate \eqref{sqrt K} combined with Plancherel imply that
	\begin{align*}
		|\mathscr{A}_{12}^{1,1} + \mathscr{A}_{12}^{1,2}| 
		\lesssim c_{\beta}^2
		\| u\|_{H^s} \|\eta\|_{H^s} (\|\eta\|_{H^s} + \beta\sqrt{\mu} \| D^1 \eta \|_{H^{s}}) \lesssim c_{\beta}^2 \|(\eta,u) \|_{X^s_{\beta,\mu}}^3.
	\end{align*}
	While the contribution of remaining terms, $\mathscr{A}_{12}^{1,3} +  \mathscr{A}_{12}^2$, will be canceled by $\mathscr{A}_{21}$. Indeed, we observe that
	\begin{align*}
		\mathscr{A}_{21}
		= 
		\big(J^s \mathcal{K}_{\mu}(D) \partial_x \eta , \eta  J^s u\big)_{L^2}
		+
		\big(J^s \mathcal{K}_{\mu}(D) \partial_x \eta ,  \mathcal{T}^{-1}_{\mu}(D)J^s u\big)_{L^2}
		 = -\mathscr{A}_{12}^{1,3} -  \mathscr{A}_{12}^2.
	\end{align*}
	Hence, combining these identities and estimates gives the bound
	\begin{equation*}
		|\mathscr{A}_{12} + \mathscr{A}_{21}| \lesssim c_{\beta}^2 \|(\eta, u )\|_{X^s_{\beta,\mu}}^3.
	\end{equation*}
	\textit{Control of $\mathscr{A}_{22}$.} Consider the two terms:
	\begin{align*}
		\mathscr{A}_{22}
		 =
		\big(J^s(u \partial_x u) ,J^s u\big)_{L^2}
		+
		\big(J^s \mathcal{T}_{\mu}(D) (u \partial_x u) ,  \eta J^s u\big)_{L^2}
		= 
		\mathscr{A}_{22}^1 + 	\mathscr{A}_{22}^2.
	\end{align*}
	The control of $\mathscr{A}_{22}^1$ is a direct consequence of the Kato-Ponce commutator estimate \eqref{Kato ponce} and integration by parts. Since $s>\frac{3}{2}$, we have that
	\begin{equation*}
		|\mathscr{A}_{22}^1| \leq 
		\big{|}
		\big{(}
		[J^s,u]\partial_x u,J^su
		\big{)}_{L^2}
		\big{|}  
		+
		\frac{1}{2}
		\big{|}
		\big{(}
		(\partial_x u)J^s u,J^su
		\big{)}_{L^2}
		\big{|} 
		\lesssim \| u\|_{H^s}^3.
	\end{equation*}
	To deal with $\mathscr{A}_{22}^2$, we make the decomposition
	\begin{align*}
		\mathscr{A}_{22}^2
		& = 
		\big( [J^s \sqrt{\mathcal{T}_{\mu}}(D), u ]  \partial_x u ,\sqrt{\mathcal{T}_{\mu}}(D)   \eta J^s u\big)_{L^2}
		+
		\big( u J^s \sqrt{\mathcal{T}_{\mu}}(D)  \partial_x u ,[\sqrt{\mathcal{T}_{\mu}}(D),   \eta] J^s u\big)_{L^2} 
		\\
		&
		\hspace{0.5cm}+
		\big( u J^s \sqrt{\mathcal{T}_{\mu}}(D)  \partial_x u ,\eta \sqrt{\mathcal{T}_{\mu}}(D)    J^s u\big)_{L^2} 
		%\big(u J^s \sqrt{\mathcal{T}_{\mu}}(D)  \partial_x u ,\sqrt{\mathcal{T}_{\mu}}(D)   \eta J^s u\big)_{L^2}
		\\
		& %= \mathscr{A}_{22}^{2,1} + \mathscr{A}_{22}^{2,2}.
		=
		\mathscr{A}_{22}^{2,1} + \mathscr{A}_{22}^{2,2} + \mathscr{A}_{22}^{2,3} .
	\end{align*}
	Then for $ \mathscr{A}_{22}^{2,1} $ we employ the Cauchy-Schwarz inequality, \eqref{Bessel Commutator T_mu J}, \eqref{cond. sol. 3}, \eqref{boundedness of T}, and the Sobolev embedding to deduce 
	\begin{align*}
		|\mathscr{A}_{22}^{2,1}| \leq \| [J^s \sqrt{\mathcal{T}_{\mu}}(D), u ]  \partial_x u\|_{L^2} \|\sqrt{\mathcal{T}_{\mu}}(D)  ( \eta J^s u)\|_{L^2} \lesssim \|(\eta,u)\|_{X^s_{\beta,\mu}}^3.
	\end{align*}
	Before we treat   $\mathscr{A}_{22}^{2,2,1}$, we note that $\| [\sqrt{\mathcal{T}_{\mu}}(D),   \eta] J^s u \|_{L^2} \lesssim \|\eta\|_{H^s}\|u\|_{H^s}$. Indeed, using \eqref{boundedness of T} and the Sobolev embedding we find that
	\begin{align}\label{claim 3}
		\| [\sqrt{\mathcal{T}_{\mu}}(D),   \eta] J^s u \|_{L^2} \lesssim \|  \eta \|_{L^{\infty}}\|J^s u \|_{L^2}
		\lesssim \|\eta\|_{H^s}\|u\|_{H^s}.
	\end{align}
	Consequently, using integration by parts, the Cauchy-Schwarz inequality, \eqref{claim 3} and \eqref{Commutator dx T} we get 
	\begin{align*}
		|\mathscr{A}_{22}^{2,2}| & = \big{|}\big(( \partial_x u) J^s \sqrt{\mathcal{T}_{\mu}}(D)  u ,[\sqrt{\mathcal{T}_{\mu}}(D),   \eta] J^s u\big)_{L^2} \big{|}
		\\
		& 
		\hspace{0.5cm}+
		\big|
		\big(u J^s \sqrt{\mathcal{T}_{\mu}}(D)  u ,\partial_x [\sqrt{\mathcal{T}_{\mu}}(D),   \eta] J^s u\big)_{L^2} \big|
		\\
		& 
		\lesssim  \|\eta\|_{H^s} \|u\|_{H^s}^3,
	\end{align*}
	 then use \eqref{cond. sol. 3} on one term. Similarly, for $\mathscr{A}_{22}^{2,1}$ we use integration by parts, the Sobolev embedding, and \eqref{boundedness of T} to get the bound
	 \begin{equation*}
	 	|\mathscr{A}_{22}^{2,3}| \lesssim   \|\partial_x(\eta u)\|_{L^{\infty}} \|u \|_{H^s}^2. %\lesssim  \|(\eta,u)\|_{X^s_{\beta,\mu}}^3.
	 \end{equation*}
	 
	  Therefore, we conclude by  \eqref{cond. sol. 3} and  gathering all these estimates that
	 \begin{equation*}
	 		|\mathscr{A}_{22}| \lesssim\|(\eta, u)\|_{X^s_{\beta,\mu}}^3,% \|u\|^2_{H^s}\|\eta\|_{H^s}^2,
	 \end{equation*}
	and by extension, we have the bound
	$$|\mathscr{I}|\lesssim_{\beta}\|(\eta, u)\|_{X^s_{\beta,\mu}}^3.$$

	\noindent 
	\underline{Control of $\mathscr{II}$.} By defintion of \eqref{2nd Sym Q} and  \eqref{2nd Eq M for T_mu} we get that,
	\begin{align*}
		\mathscr{II}
		& = 
		\big( J^{s}u,  (\partial_t \eta) J^{s} u \big)_{L^2} 
		\\
		&
		=
		-
		\big( J^{s} u, (\partial_x u) J^{s} u \big)_{L^2} 
		-
		\big( J^{s} u, (\mathcal{T}_{\mu}(D) \partial_x(\eta u)) J^{s}u \big)_{L^2},
	\end{align*}
	Then the final estimate follows  by the Cauchy-Schwarz inequality, \eqref{cond. sol. 3} and the fact that $\mathcal{T}_{\mu}(D)$ is bounded on $L^2(\R)$, then apply Hölder's inequality, and the Sobolev embedding to deduce %for $s>\frac{3}{2}$. 
	\begin{equation*}
		|	\mathscr{II} |
		\leq 
		\|\partial_x u\|_{L^{\infty}} \|u\|_{H^s}^2 + \|\mathcal{T}_{\mu}(D)\partial_x(\eta u)\|_{L^{\infty}}\|u\|_{H^s}^2 \lesssim \|(\eta,u)\|_{X^s_{\beta,\mu}}^3.
	\end{equation*}

\end{proof}

\section{Estimates for the difference of two solutions}\label{Energy est diff}

\subsection{Estimates for system  \eqref{full dispersion}}

We will now estimate the difference between two solutions of \eqref{full dispersion} given by $\bold{U}_1= (\eta_1, u_1)^T = \ve (\zeta_1, v_1 )^T$ and $\bold{U}_2 = (\eta_2, u_2)^T = \ve (\zeta_2, v_2)^T$. For   convenience,  we define  $(\psi,w) = (\eta_1 - \eta_2, u_1-u_2)$. Then $\bold{W} = (\psi,w)^T$ solves 
\begin{equation}\label{lin W}
	\partial_t \bold{W} + M(\bold{U}_1,D)\bold{W} = \bold{F}, 
\end{equation}
with $M$ defined as  in \eqref{M 1} and $\bold{F} =- \big(M(\bold{U}_1,D) - M(\bold{U}_2,D)  \big) \bold{U}_2$. Specifically, the source term is given by
\begin{equation} \label{F: source term}
	\bold{F} 
	=
	-
	\begin{pmatrix}
		w \partial_x \eta_2 + \psi \partial_x u_2 \\
		w \partial_x u_2
	\end{pmatrix} .
\end{equation}
The energy associated to \eqref{lin W} is given in terms of the symmetrizer $Q(\bold{U}_1,D)$ defined in \eqref{S: symmetrizer} and reads
\begin{equation}\label{tilde Energy s}
	\tilde{E}_{s}(\bold{W}) 
	 : =
	\big(J^s \bold{W}, Q(\bold{U}_1,D)J^s \bold{W}\big)_{L^2}. 
\end{equation}

The main result of this section reads:

\begin{prop}\label{Compactness L2}
	Take $s>2$ and $\ve, \mu \in (0,1)$.  Let $(\eta_1,u_1),(\eta_2,u_2) \in C([0,T_0]: V^s_{\sqrt{\mu}} (\mathbb{R}))$ be two solutions of \eqref{full dispersion} on a time interval $[0,T_0]$ for some $T_0 >0$. Moreover,    assume there exist $h_0 \in (0,1)$ and $h_1>0$ such that %for $i\in\{1,2\}$
	\begin{equation}\label{cond. sol. diff}
		h_0 - 1 \leq \eta_1(x,t), \quad   \forall (x,t) \in \mathbb{R} \times [0,T_0] \quad  \text{and} \quad \sup\limits_{t\in[0,T_0]}\|(\eta_1, u_1)\|_{H^s\times H^s} \leq h_1,
	\end{equation}
	when $\beta  \geq  \frac{1}{3}$, and that
	\begin{equation}\label{cond. sol. beta diff}
		-\frac{\beta}{2} \leq \eta_1(x,t), \quad   \forall (x,t) \in \mathbb{R} \times [0,T_0] \quad  \text{and} \quad \sup\limits_{t \in [0,T_0]} \|(\eta_1, u_1)\|_{H^s\times H^s} \leq h_1,
	\end{equation}
	when $0<  \beta < \frac{1}{3}$. 
	
	 Define the difference to be $\bold{W}=(\psi,w) = (\eta_1 - \eta_2, u_1 - u_2)$. Then, for the energy defined by \eqref{tilde Energy s}, there holds
	\begin{equation}\label{Energy 1}
		\frac{d}{dt} \tilde{E}_0(\bold{W}) \lesssim_{\beta} \max \limits_{i = 1,2} \|(\eta_i,u_i)\|_{V^{s}_{\sqrt{\mu}}} \|(\psi,w)\|_{V^{0}_{\sqrt{\mu}}}^2,
	\end{equation}
	and
	\begin{equation}\label{equiv 1}
		\|(\psi,w)\|^2_{V^0_{\sqrt{\mu}}} \lesssim_{\beta} \tilde{E}_0(\bold{W}) \lesssim_{\beta} \|(\psi,w)\|^2_{V^0_{\sqrt{\mu}}}.
	\end{equation}
	
	Furthermore, we have the following estimate at the  $V^s_{\sqrt{\mu}}-$ level:
	\begin{equation}\label{Energy 2}
		\frac{d}{dt} \tilde{E}_s(\bold{W}) \lesssim_{\beta} |\big(J^s \bold{F}, Q(\bold{U}_1,D) J^s \bold{W}\big)_{L^2}| +  \max \limits_{i = 1,2} \|(\eta_i,u_i)\|_{V^{s}_{\sqrt{\mu}}}  \|(\psi,w)\|_{V^s_{\sqrt{\mu}}}^2,
	\end{equation}
	and
	\begin{equation}\label{equiv 2}
		\|(\psi,w)\|^2_{V^s_{\sqrt{\mu}}} \lesssim_{\beta} \tilde{E}_s(\bold{W}) \lesssim_{\beta} \|(\psi,w)\|^2_{V^s_{\sqrt{\mu}}}.
	\end{equation}
\end{prop}

\begin{remark}\label{Remark tilde}

	The source term corresponding to $\bold{F}$ given by \eqref{Energy 2} will be treated  in the proof of Theorem \ref{Well-posedness long time full dispersion} by using  regularisation estimates and a classical Bona-Smith argument.

\end{remark}
\begin{proof}[Proof of Proposition \ref{Compactness L2}] First, the proofs of  \eqref{equiv 1} and \eqref{equiv 2} are similar to the one of \eqref{equiv. full dispersion}.
	
		Next, we only prove \eqref{Energy 1}, where \eqref{Energy 2} is more straightforward and follows the same line, utilizing   similar estimates to those applied for the proof of Proposition \ref{Energy fully disp Boussinesq}.

	To prove \eqref{Energy 1}, we use \eqref{lin W} and the self-adjointness of $Q(\bold{U}_1,D)$ to write 
	\begin{align*}
	\frac{1}{2}\frac{d}{dt} \tilde{E}_{0}(\bold{W})
	& =
	\frac{1}{2}\big(  \bold{W}, ( \partial_t Q(\bold{U}_1,D))  \bold{W}\big)_{L^2}
	 +
	\big( \bold{F}, Q(\bold{U}_1,D)  \bold{W}\big)_{L^2}
	\\
	&
	\hspace{0.4cm}
	-
	\big(  M(\bold{U}_1,D) \bold{W} , Q(\bold{U}_1,D)  \bold{W}\big)_{L^2}
	\\
	& = : I - II - III.
\end{align*} \\ 
\underline{Control of	$I$.} We estimate the first term for $s>2$ by arguing similarly to estimate \eqref{time der}. Indeed, we have that
\begin{align*}
I
	   = ( w , (\partial_{t} \eta_1)  w) 
	\lesssim 
	\|\partial_t \eta_1 \|_{L^{\infty}} \| w \|_{L^{2}}^2
	\lesssim_{\beta} 
	\| (u_1,\eta_1) \|_{V_{\sqrt{\mu}}^s} \|(\psi, w) \|_{V_{\sqrt{\mu}}^0}^2.
\end{align*}

\noindent
\underline{Control of $II$.} For $II$, we write
\begin{align*}
	II & = 
	\big(
	w \partial_x \eta_2, \psi
	\big)_{L^2}
	+
	\big(
	\psi \partial_x \eta_2, \psi
	\big)_{L^2}
%	\\
%	& \hspace{0.4cm} 
	+
	\big(w \partial_x u_2, \eta_1  w
	\big)_{L^2}
	+
	\big(w \partial_x u_2, \mathcal{K}_{\mu}(D) w
	\big)_{L^2}
	\\
	& = :
	II_1 + II_2 + II_3 + II_4 
	.
\end{align*}
The first three terms are treated by the Cauchy-Schwarz inequality and the Sobolev embedding. Take, for instance, $II_1$:
\begin{align*}
	|II_1|
	 \lesssim 
	\|w \partial_x \eta_2 \|_{L^2}  \|  \psi \|_{L^2} 
	 \lesssim \| \eta_2 \|_{H^s} \| (\psi, w) \|_{V^0_{\sqrt{\mu}}}^2,
\end{align*}
for $s>\frac{3}{2}$. Then estimating $II_2 + II_3$ similarly gives
\begin{equation*}
	|II_1 + II_2 + II_3|
	\lesssim \max \limits_{i = 1,2} \|(\eta_i,u_i)\|_{V^{s}_{\sqrt{\mu}}}
	\| (\psi, w) \|_{V^0_{\sqrt{\mu}}}^2.
\end{equation*} 
Regarding the term containing the multiplier $\mathcal{K}_{\mu}(D)$, we write 
\begin{align*}
	II_4   \leq 
	\|  \sqrt{\mathcal{K}_{\mu}}(D)
	(w \partial_x u_2)\|_{L^2}  \| \sqrt{\mathcal{K}_{\mu}}(D) w \|_{L^2}
	=:
	II_{4}^1 \cdot II_4^2,
\end{align*}
and make the observation
\begin{align*}
	II_{4}^1
	 &
	  \leq
	  \|(\sqrt{\mathcal{K}_{\mu}}(D) -  \sqrt{\beta}\mu^{\frac{1}{4}}D^{\frac{1}{2}})(w \partial_xu_2)\|_{L^2}
	 % \\ 
	 %&
	 %\hspace{1cm}
	 +
	  \sqrt{\beta}\mu^{\frac{1}{4}}  \| D^{\frac{1}{2}}(w \partial_x u_2)\|_{L^2}
	  \\
	  &=: II_4^{1,1}+ \sqrt{\beta} II_4^{1,2}.
\end{align*}
For the first term, we note that $(\sqrt{\mathcal{K}_{\mu}}(D) - \sqrt{\beta}\mu^{\frac{1}{4}}D^{\frac{1}{2}})$ is bounded on $L^2(\mathbb{R})$ by \eqref{Comparison}, and we can conclude by the Sobolev embedding that
\begin{align*}
	%\|(\sqrt{\mathcal{K}_{\mu}}(D) -  \sqrt{\beta}\mu^{\frac{1}{4}}D^{\frac{1}{2}})\partial_{x}(w \partial_xu_2)\|_{L^2} 
	 II_4^{1,1} \lesssim_{\beta} \| w \partial_x \eta_2\|_{L^2}  
	\lesssim_{\beta} \|\eta_2\|_{H^s}  \|(\psi,w) \|_{V^0_{\sqrt{\mu}}}. 
\end{align*}
For the remaining term, $ II_4^{1,2}$,   we first make an observation. Let $\nu = \frac{1}{2}^-$ and $(p_1,p_2) = (\frac{1}{\nu}, \frac{2}{1-2\nu})$ then by  \eqref{Sobolev embedding s small}  there holds
\begin{equation}\label{observation ...}
	\| D^{\frac{1}{2}} \partial_x u_2 \|_{L^{p_1}} \mu^{\frac{1}{4}} \|  w\|_{L^{p_2}} \lesssim \|u_2\|_{H^{2-\nu}} \mu^{\frac{1}{4}}\|D^{\frac{1}{2}}w\|_{L^2}.
\end{equation}
Moreover, by the fractional Leibniz rule \eqref{FracLeibnizRule}, the triangle inequality and Hölder's inequality yields the bound
\begin{align*}
	II^{1,2}_4 
	&
	\lesssim 
	\mu^{\frac{1}{4}} \| D^{\frac{1}{2}}(w\partial_x u_2) - w D^{\frac{1}{2}} \partial_x u_2 - (\partial_x u_2)D^{\frac{1}{2}}w \|_{L^2 } 
	+
	\mu^{\frac{1}{4}}\| w D^{\frac{1}{2}}\partial_x u_2 \|_{L^2} 
	\\
	&
	\hspace{0.5cm} 
	+
	\mu^{\frac{1}{4}}\| (\partial_x u_2) D^{\frac{1}{2}}w\|_{L^2} 
	\\
	&
	\lesssim \|D^{\frac{1}{2}} \partial_x u_2 \|_{L^{p_1}} \mu^{\frac{1}{4}} \|w\|_{L^{p_2}} + \|w\|_{L^2}\mu^{\frac{1}{4}}\|D^{\frac{1}{2}}\partial_x u_2\|_{L^{\infty}} + \|\partial_x u_2 \|_{L^{\infty}} \mu^{\frac{1}{4}}\|D^{\frac{1}{2}} w\|_{L^2}.
\end{align*}
Now, since $\frac{1}{p_1} + \frac{1}{p_2} = \nu + \frac{1-2\nu}{2} = \frac{1}{2}$, we may apply \eqref{observation ...} to deal with the first term, and combined with the Sobolev embedding $H^{\frac{1}{2}^+}(\R) \hookrightarrow L^{\infty}(\R)$ we deduce that
\begin{align*}
		II^{1,2}_4
		\lesssim 
		\|u_2\|_{H^s} \mu^{\frac{1}{4}}\|D^{\frac{1}{2}} w \|_{L^2} 
		+
		\|w\|_{L^2} \mu^{\frac{1}{4}} \|D^{\frac{1}{2}}u_2\|_{H^s},
\end{align*}
with $s>\frac{3}{2}$.  Consequently, the bound on  $II_4$ is given by
\begin{align*}
	II_4  %& \leq 
%	\|  \sqrt{\mathcal{K}_{\mu}}(D)
%	\partial_x(w \partial_x u_2)\|_{L^2}  \| \sqrt{\mathcal{K}_{\mu}}(D)\partial_x w \|_{L^2}
%	\\
%	&
	\lesssim_{\beta}
	\max \limits_{i = 1,2} \|(\eta_i,u_i)\|_{V^{s}_{\sqrt{\mu}}} \|(\psi,w)\|_{V^0_{\sqrt{\mu}}}^2,
\end{align*}
which allows us to conclude that
\begin{equation*}
	II \lesssim_{\beta}  \max \limits_{i = 1,2} \|(\eta_i,u_i)\|_{V^{s}_{\sqrt{\mu}}} \| (\psi,w)\|_{V^0_{\sqrt{\mu}}}^2.
\end{equation*}

\noindent
\underline{Control of $III$.} By definition, we must estimate:
\begin{align*}
	III 
	& =
	\big(
	u_1 \partial_x \psi, \psi
	\big)_{L^2}
	+
	\big((\mathcal{K}_{\mu}(D) + \eta_1)\partial_x w , \psi
	\big)_{L^2}
	\\
	& \hspace{0.4cm} 
	+
	\big(
	\partial_x \psi,(\mathcal{K}_{\mu}(D) + \eta_1) w
	\big)_{L^2}
	+
	\big( 
	u_1 \partial_x w, 
	(\mathcal{K}_{\mu}(D) + \eta_1)  w
	\big)_{L^2}
	\\
	& = 
	A_1 + A_2 + A_3 + A_4.   
\end{align*}
The first term is handled by integration by parts and the Sobolev embedding
\begin{align*}
	A_1 
	& \lesssim
	\|\partial_x u_1 \|_{L^{\infty}} \|  \psi \|_{L^2}^2 \lesssim \| u_1 \|_{H^s} \|  \psi \|_{L^2}^2. 
\end{align*}
Next, we observe a cancelation in the off-diagonal terms due to the symmetry. Indeed, we see after integrating by parts that
\begin{align*}
 	A_2  = 	-\big((\partial_x\eta_1) w , \psi
 	\big)_{L^2} - A_3.
\end{align*}
Consequently, we observe after using Hölder's inequality and the Sobolev embedding that
\begin{align*}
	|A_2 + A_3| \lesssim	\|\partial_x\eta_1\|_{L^{\infty}} \| w \|_{L^2}\| \psi\|_{L^2}.
\end{align*}
The only term remaining is $A_4$, which contains the multiplier that will need some more care. In particular, we write
\begin{align*}
	A_4
	& = 
	\big( 
	u_1 \partial_x w, 
	\eta_1  w
	\big)_{L^2}
	+ 
	\big(
u_1 \partial_x w, 
\mathcal{K}_{\mu}(D)  w
\big)_{L^2}
	\\
	& = : A_4^1 + A_4^2.
\end{align*}
The first term is again treated by integration by parts, and we obtain the bound
\begin{equation*}
	|A_4^1| \lesssim \|  u_1 \|_{H^s} \| \eta_1 \|_{H^s} \| w \|_{L^2}^2.
\end{equation*}
Lastly, to estimate $A_4^2$, we split the kernel $\mathcal{K}_{\mu}(D)$ into several pieces that are localized in low and high frequencies:
\begin{align}%\notag
	\mathcal{K}_{\mu}(D) 
%	& =
	% ((\chi^{(1)}_{\mu})^2\mathcal{K}_{\mu})(D)  + ((\chi_{\mu}^{(2)})^2 \sigma^2_{\mu,\frac{1}{2}})(D)- ((\chi^{(2)}_{\mu})^2(\sigma_{\mu,\frac{1}{2}}^2- \mathcal{K}_{\mu}))(D)
%	\\
	%&
	 = \label{Split multiplier K}
	((\chi^{(1)}_{\mu})^2\mathcal{K}_{\mu})(D) + ((\chi_{\mu}^{(2)})^2 (\sigma_{\mu,\frac{1}{2}})^2)(D) - ((\chi^{(2)}_{\mu})^2(\sigma_{\mu,0})^2)(D),
\end{align}
where $\sigma_{\mu,\frac{1}{2}}(D)$ is defined in \eqref{Sigma 1/2}, $\sigma_{\mu,0}(D)$ is defined in \eqref{sigmaDiff} and $\chi^{(i)}_{\mu}(D)$ with its porperties given by Definition \ref{multiplier}. Then,  we get that
\begin{align*}
	\big(
	 u_1 \partial_xw, \mathcal{K}_{\mu}(D) w
	\big)_{L^2} 
	& = 
	\big((\chi^{(1)}_{\mu}
\sqrt{\mathcal{K}_{\mu}})(D) 
(u_1 \partial_x w), (\chi^{(1)}_{\mu}\sqrt{\mathcal{K}_{\mu}})(D) w
\big)_{L^2} 
\\
&
\hspace{0.4cm}
+
\big((\chi^{(2)}_{\mu}\sigma_{\mu,\frac{1}{2}})(D) 
(u_1 \partial_x w), (\chi^{(2)}_{\mu}\sigma_{\mu,\frac{1}{2}})(D)  w
\big)_{L^2} 
\\
&
\hspace{0.4cm}
-
\big((\chi^{(2)}_{\mu}\sigma_{\mu,0})(D) 
(u_1 \partial_x w), (\chi^{(2)}_{\mu}\sigma_{\mu,0})(D)  w
\big)_{L^2} 
\\
& = : A^{2,1}_4 + A^{2,2}_4 - A^{2,3}_4.
\end{align*}
We treat each term individually using the commutator estimates in Lemma \ref{Commutator L2}, where the remaining part is symmetric and is treated by using integration by parts and the Sobolev embedding in the usual way.\\

\noindent
\textit{Control of $A^{2,1}_4$.} Proceeding as explained  above, we have that
\begin{align*}
	A^{2,1}_4 & = 	\big([(\chi^{(1)}_{\mu}
	\sqrt{\mathcal{K}_{\mu}})(D) ,
	u_1 ]\partial_x w, (\chi^{(1)}_{\mu}\sqrt{\mathcal{K}_{\mu}})(D) w
	\big)_{L^2}
	\\
	& 
	\hspace{0.4cm} +
	\big(u_1(\chi^{(1)}_{\mu}
	\sqrt{\mathcal{K}_{\mu}})(D) 
	\partial_xw, (\chi^{(1)}_{\mu}\sqrt{\mathcal{K}_{\mu}})(D) w
	\big)_{L^2}  
	\\
	& = 
	A^{2,1,1}_4 + A^{2,1,2}_4.
\end{align*}
For $A^{2,1,1}_4$ we use the Cauchy-Schwarz inequality, \eqref{Product chi K}, and  \eqref{Commutator chi K} to obtain the bound
\begin{align*}
	A^{2,1,1}_4 
	& \lesssim
	 \|[(\chi^{(1)}_{\mu}
	\sqrt{\mathcal{K}_{\mu}})(D), u_1 ]\partial_x w\|_{L^2} \| (\chi^{(1)}_{\mu}\sqrt{\mathcal{K}_{\mu}})(D) w \|_{L^2} 
	\\
	& \lesssim
	\| u_1 \|_{H^s} \|  w\|_{L^2}^2.
	%\\
%	& 
%	\lesssim \| u_1 \|_{H^s}  \| (\xi,w)\|_{V^1_{\sqrt{\mu}}}^2.
\end{align*}
For the remaining term, we deduce from \eqref{Product chi K} that
\begin{align*}
	A^{2,1,2}_4 
	& =
	 -\frac{1}{2}\big((\partial_xu_1)(\chi^{(1)}_{\mu}
	\sqrt{\mathcal{K}_{\mu}})(D) 
	 w, (\chi^{(1)}_{\mu}\sqrt{\mathcal{K}_{\mu}})(D) w
	\big)_{L^2}
	\\ 
	& \lesssim
	\| \partial_x u_1 \|_{L^{\infty}} \| w\|_{L^2}^2.
%	\\
%	& 
%	\lesssim \| u_1 \|_{H^s}  \| (\xi,w)\|_{V^1_{\sqrt{\mu}}}^2.  
\end{align*}

\noindent
\textit{Control of $A^{2,2}_4$.} Similarly we get from the estimates \eqref{Product sigma 1/2},   \eqref{Commutator D} and the Sobolev embedding that
\begin{align*}
	A^{2,2}_4 
	& = 
	\big([(\chi^{(2)}_{\mu}\sigma_{\mu,\frac{1}{2}})(D)  ,
	u_1 ]\partial_x w, (\chi^{(2)}_{\mu}\sigma_{\mu,\frac{1}{2}})(D)  w
	\big)_{L^2}
	\\
	& 
	\hspace{0.4cm}-
	\frac{1}{2}\big((\partial_xu_1)(\chi^{(2)}_{\mu}\sigma_{\mu,\frac{1}{2}})(D) 
	 w, (\chi^{(2)}_{\mu}\sigma_{\mu,\frac{1}{2}})(D)  w
	\big)_{L^2}
	\\ 
	& 
	\lesssim
	\| u_1\|_{H^s} 
	(\|w\|_{L^2} +  \mu^{\frac{1}{4}}\| D^{\frac{1}{2}}  w \|_{L^2})^2.
%	+
	% \|\partial_x u_1 \|_{L^{\infty}} \sqrt{\mu}\| \partial_x w\|_{H^{\frac{1}{2}}}^2 
	%\\
%	&
	%\lesssim \| u_1 \|_{H^s} 
	%\| (\xi,w)\|_{V^1_{\sqrt{\mu}}}^2.  
\end{align*}

\noindent
\textit{Control of $A^{2,3}_4$.} By the  same approach as above, combined with estimates  \eqref{Product sigma} and \eqref{Commutator sigma} leaves us with the bound
\begin{align*}
	A^{2,3}_4 
	& = 
	\big([(\chi^{(2)}_{\mu}\sigma_{\mu,0})(D)  ,
	u_1 ]\partial_x w, (\chi^{(2)}_{\mu}\sigma_{\mu,0})(D)  w
	\big)_{L^2}
	\\
	& 
	\hspace{0.4cm}-
	\frac{1}{2}\big((\partial_xu_1)(\chi^{(2)}_{\mu}\sigma_{\mu,0})(D) 
	  w , (\chi^{(2)}_{\mu}\sigma_{\mu,0})(D)  w
	\big)_{L^2}
	\\ 
	&
	\lesssim \| u_1 \|_{H^s} \|  w \|_{L^2}^2 + \|\partial_x u_1 \|_{L^{\infty}} \| w\|_{L^2}^2.
\end{align*}

Gathering all these estimates, we obtain the result
\begin{align*}
	A_4 = 
A_{4}^{1} +A_{4}^{2} +  A_{4}^{3}
	& \lesssim_{\beta} (\| u_1 \|_{H^s} + \|\eta_1\|_{H^s})
	 	\| (\psi,w)\|_{V^0_{\sqrt{\mu}}}^2.
\end{align*}
Adding $I + II + III$ concludes the proof. 
\end{proof}

\subsection{Estimates for system \eqref{Whitham Boussinesq}}

As in the prvious subsection, we let $\bold{U}_1= (\eta_1, u_1)^T = \ve (\zeta_1, v_1 )^T$ and $\bold{U}_2 = (\eta_2, u_2)^T = \ve (\zeta_2, v_2)^T$ be two solutions of \eqref{Whitham Boussinesq} and define the difference  $(\psi,w) = (\eta_1 - \eta_2, u_1-u_2)$. Then $\bold{W} = (\psi,w)^T$ solves 
\begin{equation}\label{W T_mu}
	\partial_t \bold{W} + \mathcal{M}(\bold{U}_1,D)\bold{W} = \bold{F}, 
\end{equation}
with $\mathcal{M}$ defined as  in \eqref{Math M} and $\bold{F}$ will remain the same as previously defined by \eqref{F: source term}.
Then the energy associated to \eqref{W T_mu} is given in terms of the symmetrizer \eqref{Sym Q}:
\begin{equation}\label{tilde Energy s 2}
	\tilde{\mathcal{E}}_{s}(\bold{W}) 
	: =
	\big(J_{\mu}^{\frac{1}{2}}J^s \bold{W}, \mathcal{Q}(\bold{U}_1,D)J_{\mu}^{\frac{1}{2}}J^s \bold{W}\big)_{L^2}. 
\end{equation}
\begin{prop}\label{prop diff 2}
	Take $s>\frac{3}{2}$ and $\ve, \mu \in (0,1)$.  Let $(\eta_1,u_1),(\eta_2,u_2) \in C([0,T_0]: V^s_{\sqrt{\mu}} (\mathbb{R}))$ be two solutions of \eqref{Whitham Boussinesq} on a time interval $[0,T_0]$ for some $T_0 >0$. Moreover, assume there exists $h_0 \in (0,1)$ and $h_1>0$ such that % for $i\in\{1,2\}$
	\begin{equation}\label{cond. sol. diff 2}
		h_0 - 1 \leq \eta_1(x,t), \quad   \forall (x,t) \in \mathbb{R} \times [0,T_0] \quad  \text{and} \quad  \sup\limits_{t \in [0,T_0] }\|(\eta_1, u_1)\|_{H^s\times H^s} \leq h_1.
	\end{equation}

	Define the difference to be $\bold{W}=(\psi,w) = (\eta_1 - \eta_2, u_1 - u_2)$. Then, for the energy defined by \eqref{tilde Energy s 2}, there holds
	
	\begin{equation}\label{Diff Energy 2 1}
		\frac{d}{dt} \tilde{\mathcal{E}}_0(\bold{W}) \lesssim \max \limits_{i = 1,2} \|(\eta_i,u_i)\|_{V^{s}_{\sqrt{\mu}}}   \|(\psi,w)\|_{V^{0}_{\sqrt{\mu}}}^2,
	\end{equation}
	and
	\begin{equation}\label{Diff Equiv 2 1}
		\|(\psi,w)\|^2_{V^0_{\sqrt{\mu}}} \lesssim \tilde{\mathcal{E}}_0(\bold{W}) \lesssim \|(\psi,w)\|^2_{V^0_{\sqrt{\mu}}}.
	\end{equation}
	
	Furthermore, we have the following estimate at the  $V^s_{\sqrt{\mu}}-$ level:
	\begin{equation}\label{Diff Energy 2 s}
		\frac{d}{dt} \mathcal{\tilde{E}}_s(\bold{W}) \lesssim  |\big(J^s \bold{F}, \mathcal{Q}(\bold{U}_1,D)  J^s \bold{W}\big)_{L^2}| +   \max \limits_{i = 1,2} \|(\eta_i,u_i)\|_{V^{s}_{\sqrt{\mu}}} \|(\psi,w)\|_{V^s_{\sqrt{\mu}}}^2,
	\end{equation}
	and
	\begin{equation}\label{Diff Equiv 2 s}
		\|(\psi,w)\|^2_{V^s_{\sqrt{\mu}}} \lesssim  \mathcal{\tilde{E}}_s(\bold{W}) \lesssim \|(\psi,w)\|^2_{V^s_{\sqrt{\mu}}}.
	\end{equation}
\end{prop}

\begin{proof}The proofs of \eqref{Diff Equiv 2 1} and \eqref{Diff Equiv 2 s} are similar to the proof of \eqref{Equivalence Whitham Boussinesq}.

	Also, we only prove \eqref{Diff Energy 2 1} since the control of \eqref{Diff Energy 2 s} follows by the proof of Proposition \ref{A priori Whitham Boussinesq}.

	To prove \eqref{Diff Energy 2 1}, we use \eqref{W T_mu} and the self-adjointness of $\mathcal{Q}(\bold{U}_1,D)$ to write
	\begin{align*}
		\frac{1}{2} \frac{d}{dt}  \tilde{\mathcal{E}}_{0}(\bold{W})  
		&= 
		\frac{1}{2}\big( J_{\mu}^{\frac{1}{2}} \bold{W}, ( \partial_t \mathcal{Q}(\bold{U}_1,D)) J_{\mu}^{\frac{1}{2}}  \bold{W}\big)_{L^2}
		%\\
		%& \hspace{0.4cm}
		+
		\big(J_{\mu}^{\frac{1}{2}} \bold{F}, \mathcal{Q}(\bold{U}_1,D) J_{\mu}^{\frac{1}{2}}  \bold{W}\big)_{L^2}
		\\
		&
		\hspace{0.4cm}
		-
		\big(J_{\mu}^{\frac{1}{2}}  \mathcal{M}(\bold{U}_1,D) \bold{W} , \mathcal{Q}(\bold{U}_1,D) J_{\mu}^{\frac{1}{2}}  \bold{W}\big)_{L^2}
		\\
		& = : \mathcal{I} - \mathcal{II} - \mathcal{III}.
	\end{align*}
	\underline{Control of $\mathcal{I}$}. Using \eqref{W T_mu}, \eqref{J_mu}, the Sobolev embedding and \eqref{cond. sol. diff 2} yields 
	\begin{align*}
		|\mathcal{I}|
		 = 
		\frac{1}{2} |\big( J_{\mu}^{\frac{1}{2}}  w, (\partial_t \eta_1) J_{\mu}^{\frac{1}{2}}  w \big)_{L^2}|
		\lesssim 
		\|(\eta_1,u_1)\|_{V^s_{\sqrt{\mu}}}\|J_{\mu}^{\frac{1}{2}}  w \|_{L^2}^2, %\lesssim 	\|(\psi,w)\|_{V^0_{\sqrt{\mu}}}^2,
	\end{align*}
	since $s >\frac{3}{2}$.\\
	
	\noindent
	\underline{Control of $\mathcal{II}$.} The contribution of the source term is given by
	\begin{align*}
		\mathcal{II} 
		& =
		\big(  J_{\mu}^{\frac{1}{2}}  (w\partial_x \eta_2), \mathcal{T}_{\mu}(D) J_{\mu}^{\frac{1}{2}}  \psi \big)_{L^2}  
		+
		\big( J_{\mu}^{\frac{1}{2}} (\psi \partial_x u_2),  \mathcal{T}_{\mu}(D) J_{\mu}^{\frac{1}{2}}   \psi \big)_{L^2}  
		\\
		&
		\hspace{0.4cm}
		+
		\big( J_{\mu}^{\frac{1}{2}}   (w\partial_x u_2),  J_{\mu}^{\frac{1}{2}}   w \big)_{L^2}  
		+
		\big(   J_{\mu}^{\frac{1}{2}}(w\partial_x u_2),  \eta_1 J_{\mu}^{\frac{1}{2}}  w \big)_{L^2}  
		\\ 
		& 
		=:
		\mathcal{II}_1 + \mathcal{II}_2 + \mathcal{II}_3 +  \mathcal{II}_4.
	\end{align*}
	\textit{Control of $\mathcal{II}_1 + \mathcal{II}_2$.} The estimate of $\mathcal{II}_1$ is a direct consequence of the Cauchy-Schwarz inequality, \eqref{Equiv sqrt T} and the Sobolev embedding. Indeed, since $s>\frac{3}{2}$, we get
	\begin{align*}
		|\mathcal{II}_1| 
		\leq
		\|\sqrt{\mathcal{T}_{\mu}}(D) J_{\mu}^{\frac{1}{2}} (w\partial_x \eta_2)\|_{L^2} \|\sqrt{\mathcal{T}_{\mu}}(D)J^{\frac{1}{2}}_{\mu}\psi\|_{L^2}
		\lesssim 
		\| \eta_2\|_{H^s} \|w\|_{L^2}\|\psi\|_{L^2}.
	\end{align*}
	Next, the control of $\mathcal{II}_2$ follows by the same estimates and gives
	\begin{align*}
		|\mathcal{II}_2| 
		 \lesssim 
		\|\eta_2\|_{H^s} \|\psi\|_{H^1}^2.
	\end{align*}

	\noindent
	\textit{Control of $\mathcal{II}_3+ \mathcal{II}_4$.} We first deduce from  \eqref{Comparison J_mu D_mu} that
	\begin{align*}
		\| J^{\frac{1}{2}}_{\mu} (w\partial_x u_2) \|_{L^2}
		&
		\leq
		\|  w\partial_x u_2 \|_{L^2} + \mu^{\frac{1}{4}}\| D^{\frac{1}{2}} (w\partial_x u_2) \|_{L^2} .
	%	\\
	%	&
	%	\lesssim
	%	\|(\psi,w)\|_{V^1_{\sqrt{\mu}}}
	\end{align*}
	The first term is estimated by the Sobolev embedding, while the second term is equal to the term $II^{1,2}_4$ in the proof of Proposition \ref{Compactness L2}. Since the terms $w$ and $u_2$ in $II^{1,2}_4$ belong to the same function space, we can apply the same estimates. Thus, there holds for $s>\frac{3}{2}$ that
	\begin{equation}\label{Claim Jmu}
		\| J^{\frac{1}{2}}_{\mu} (w\partial_x u_2) \|_{L^2} \lesssim  \max \limits_{i = 1,2} \|(\eta_i,u_i)\|_{V^{s}_{\sqrt{\mu}}}   \|(\psi,w)\|_{V^0_{\sqrt{\mu}}}^2.
	\end{equation}
	 Therefore, by using the Cauchy-Schwarz inequality, \eqref{Claim Jmu}, \eqref{J_mu}, \eqref{cond. sol. diff 2}, and the Sobolev embedding implies
	\begin{align*}
		 |\mathcal{II}_3|+|\mathcal{II}_4|  
		 \lesssim 
		(1+ \| \eta_1\|_{L^{\infty}})
		 \|   J_{\mu}^{\frac{1}{2}}(w\partial_x u_2)\|_{L^2} \|J_{\mu}^{\frac{1}{2}}  w \|_{L^2}  
		 \lesssim \max \limits_{i = 1,2} \|(\eta_i,u_i)\|_{V^{s}_{\sqrt{\mu}}} 
		 \|(\psi,w)\|_{V^0_{\sqrt{\mu}}}^2.
	\end{align*}

	\noindent
	\underline{Control of $\mathcal{III}$.} Lastly, the symmetrized term reads:
	\begin{align*}
		\mathcal{III} 
		& =
		\big( J_{\mu}^{\frac{1}{2}} (u_1 \partial_x \psi), \mathcal{T}_{\mu}(D)  J_{\mu}^{\frac{1}{2}}\psi \big)_{L^2} 
		+
		\big(  J_{\mu}^{\frac{1}{2}} \big{(}(1+\eta_1) \partial_x w\big{)}, \mathcal{T}_{\mu}(D) J_{\mu}^{\frac{1}{2}} \psi \big)_{L^2} 
		\\
		& \hspace{0.4cm}
		+
		\big( \mathcal{T}_{\mu} (D) J_{\mu}^{\frac{1}{2}} \partial_x \psi, (1+\eta_1)  J_{\mu}^{\frac{1}{2}}  w \big)_{L^2} 
		+
		\big(J_{\mu}^{\frac{1}{2}}  ( u_1 \partial_x w), (1+\eta_1)  J_{\mu}^{\frac{1}{2}}  w \big)_{L^2} 
		\\
		& =  \mathcal{A}_{11} + \mathcal{A}_{12} + \mathcal{A}_{21} + \mathcal{A}_{22}.
	\end{align*}
	Each term is treated by using integration by parts and suitable commutator estimates. \\

	\noindent
	\textit{Control of $\mathcal{A}_{11}$.}  For $\mathcal{A}_{11}$,  we use  integration by parts to find that
	\begin{align*}\label{Yolo}
		\mathcal{A}_{11}
		 = 
		\big(  [\sqrt{\mathcal{T}_{\mu}}(D)J_{\mu}^{\frac{1}{2}},u_1]  \partial_x \psi,  \sqrt{\mathcal{T}_{\mu}}(D)  J_{\mu}^{\frac{1}{2}}  \psi \big)_{L^2} 
		%	\\ 
		%	& \hspace{0.5cm}
		-\frac{1}{2}
		\big( (\partial_x u_1) \sqrt{\mathcal{T}_{\mu}}(D)J_{\mu}^{\frac{1}{2}} \psi,  \sqrt{\mathcal{T}_{\mu}}(D) J_{\mu}^{\frac{1}{2}}  \psi \big)_{L^2}.
	\end{align*} 
	Thus, it follows from  the commutator estimate \eqref{L2 Commutator T_mu J_mu} with $s>\frac{3}{2}$ and estimate \eqref{Equiv sqrt T} that
	\begin{equation*}
		|\mathcal{A}_{11}| \lesssim \max \limits_{i = 1,2} \|(\eta_i,u_i)\|_{V^{s}_{\sqrt{\mu}}} \|\psi\|^2_{L^2}.
	\end{equation*}

	\noindent
	\textit{Control of $\mathcal{A}_{12} + \mathcal{A}_{21}$.} Treating the off-diagonal terms we first observe, 
	\begin{align*}
		\mathcal{A}_{12} 
		& = 
		\big(  [\sqrt{\mathcal{T}_{\mu}} (D)J_{\mu}^{\frac{1}{2}},\eta_1 ]\partial_x w, \sqrt{\mathcal{T}_{\mu}}(D)  J_{\mu}^{\frac{1}{2}} \psi \big)_{L^2}   
		\\
		& 
		\hspace{0.5cm}
		+
		\big(  (1+\eta_1)  \sqrt{\mathcal{T}_{\mu}}(D)J_{\mu}^{\frac{1}{2}} \partial_x w, \sqrt{\mathcal{T}_{\mu}}(D)J_{\mu}^{\frac{1}{2}} \psi \big)_{L^2} 
		\\
		& = 
		\mathcal{A}_{12}^{1} + \mathcal{A}_{12}^{2}.
	\end{align*}
	The commutator estimate \eqref{L2 Commutator T_mu J_mu} and estimate \eqref{Equiv sqrt T} deals with the first term. Indeed, we get the bound
	\begin{align*}
		|\mathcal{A}_{12}^{1}|
		\leq 
		\| [\sqrt{\mathcal{T}_{\mu}} (D)J_{\mu}^{\frac{1}{2}},\eta_1 ]\partial_x w\|_{L^2}
		\|\sqrt{\mathcal{T}_{\mu}}(D)  J_{\mu}^{\frac{1}{2}}  \psi \|_{L^2}  \lesssim\max \limits_{i = 1,2} \|(\eta_i,u_i)\|_{V^{s}_{\sqrt{\mu}}}  \|(\psi,w)\|_{V^0_{\sqrt{\mu}}}^2.
	\end{align*}
	Next, we  integrate $\mathcal{A}_{12}^{2}$ by parts to obtain two new terms
	\begin{align*}
		\mathcal{A}_{12}^{2} 
		& = 
		-
		\big{(}(\partial_x\eta_1) J^{\frac{1}{2}}_{\mu}  \sqrt{\mathcal{T}_{\mu}} (D)  w,J^{\frac{1}{2}}_{\mu}  \sqrt{\mathcal{T}_{\mu}} (D) \psi \big{)}_{L^2}
		\\
		& \hspace{0.5cm} 
		-
		\big{(}(1+\eta_1) J^{\frac{1}{2}}_{\mu}  \sqrt{\mathcal{T}_{\mu}} (D)  w,  J^{\frac{1}{2}}_{\mu}  \sqrt{\mathcal{T}_{\mu}} (D)\partial_x\psi \big{)}_{L^2}
		\\ 
		& = 
		\mathcal{A}_{12}^{2,1} + \mathcal{A}_{12}^{2,2}.  
	\end{align*}
	Arguing as above, we find that
	\begin{align*}
		|\mathcal{A}_{12}^{2,1}|
		\leq
		\|\partial_x\eta_1\|_{L^{\infty}}\| J^{\frac{1}{2}}_{\mu}  \sqrt{\mathcal{T}_{\mu}} (D)  w\|_{L^2} 
		\|J^{\frac{1}{2}}_{\mu}  \sqrt{\mathcal{T}_{\mu}} (D) \psi \|_{L^2} \lesssim \|\eta_1\|_{H^s} \|(\psi,w)\|_{V^0_{\sqrt{\mu}}}^2,
	\end{align*}
	for $s>\frac{3}{2}$. On the other hand, the term $\mathcal{A}_{12}^{2,2}$,  is absorbed by  $\mathcal{A}_{21}$. Indeed,
	\begin{align*}
		\mathcal{A}_{21} 
		&=
		-\big( \sqrt{\mathcal{T}_{\mu}}(D) J^{\frac{1}{2}}_{\mu}  \psi ,\partial_x[\sqrt{\mathcal{T}_{\mu}}(D), \eta_1] J^{\frac{1}{2}}_{\mu}   w\big)_{L^2} 
		\\
		&
		\hspace{0.5cm}
		+
		\big( \sqrt{\mathcal{T}_{\mu}}(D)J^{\frac{1}{2}}_{\mu} \partial_x \psi ,  (1+ \eta_1)\sqrt{\mathcal{T}_{\mu}}(D)J^{\frac{1}{2}}_{\mu} w\big)_{L^2}
		\\
		& = \mathcal{A}_{21}^1 + \mathcal{A}_{21}^2,
	\end{align*}
	with $\mathcal{A}_{21}^2 = -  \mathcal{A}_{12}^{2,2}$.  We estimate $\mathcal{A}_{21}^1$ by using the Cauchy-Schwarz inequality, \eqref{Equiv sqrt T}, \eqref{Commutator dx T}, and \eqref{J_mu} to get
	\begin{align*}
		|\mathcal{A}_{21}^1|
		\leq 
		\| \sqrt{\mathcal{T}_{\mu}}(D) J^{\frac{1}{2}}_{\mu}   \psi\|_{L^2}
		\|\partial_x[\sqrt{\mathcal{T}_{\mu}}(D), \eta_1] J^{\frac{1}{2}}_{\mu}   w\|_{L^2} 
		\lesssim \|\eta_1\|_{H^s}
		\|(\psi,w)\|_{V^0_{\sqrt{\mu}}}^2.
	\end{align*}

	Thus, we deduce by gathering all these estimates that
	\begin{equation*}
	 	|\mathcal{A}_{12} + \mathcal{A}_{21}| \lesssim \max \limits_{i = 1,2} \|(\eta_i,u_i)\|_{V^{s}_{\sqrt{\mu}}}  \|(\psi,w)\|_{V^0_{\sqrt{\mu}}}^2.
	 \end{equation*}

	\noindent
	\textit{Control of $\mathcal{A}_{22}$.} Lastly, the term $\mathcal{A}_{22}$ is estimated by  \eqref{L2 commutator J_mu} for $s>\frac{3}{2}$,  \eqref{cond. sol. diff 2}, and integration by parts 
	\begin{align*}
		|\mathcal{A}_{22}^2| 
		& \leq 
		|\big(  [J_{\mu}^{\frac{1}{2}} ,u_1] \partial_x w, (1+\eta_1)  J_{\mu}^{\frac{1}{2}}  w \big)_{L^2} |
		+|
		\big(  u_1J_{\mu}^{\frac{1}{2}}  \partial_x w, (1+\eta_1)   J_{\mu}^{\frac{1}{2}}  w \big)_{L^2} |
		\\
		& \lesssim \max \limits_{i = 1,2} \|(\eta_i,u_i)\|_{V^{s}_{\sqrt{\mu}}} 
		\|(\psi,w)\|_{V^0_{\sqrt{\mu}}}^2.
	\end{align*}

	Therefore, we deduce that
	\begin{align*}
		\frac{d}{dt}  \tilde{\mathcal{E}}_{0}(\bold{W})  \lesssim |\mathcal{I}| + |\mathcal{II}| + |\mathcal{III}| \lesssim \max \limits_{i = 1,2} \|(\eta_i,u_i)\|_{V^{s}_{\sqrt{\mu}}} 	\|(\psi,w)\|_{V^0_{\sqrt{\mu}}}^2,
	\end{align*}
	which concludes the proof of Proposition \ref{prop diff 2}. 

\end{proof}

\subsection{Estimates for system  \eqref{2nd Whitham Boussinesq}}

Again, we let $\bold{U}_1= (\eta_1, u_1)^T = \ve (\zeta_1, v_1 )^T$ and $\bold{U}_2 = (\eta_2, u_2)^T = \ve (\zeta_2, v_2)^T$ be two solutions of \eqref{2nd Whitham Boussinesq} and define the difference  $(\psi,w) = (\eta_1 - \eta_2, u_1-u_2)$. Then $\bold{W} = (\psi,w)^T$ solves 
\begin{equation}\label{2nd W T_mu}
	\partial_t \bold{W} + \mathscr{M}(\bold{U}_1,D)\bold{W} = \bold{F}, 
\end{equation}
with $\mathscr{M}$ defined as  in \eqref{2nd Math M} and $\bold{F}$ is defined by
\begin{equation}\label{2nd F}
	\bold{F} = -
	\begin{pmatrix}
		\mathcal{T}_{\mu}(D)(w \partial_x \eta_2) + \mathcal{T}_{\mu}(D)(\psi \partial_x u_2)
		\\
		\mathcal{T}_{\mu}(D)(w \partial_x u_2)
	\end{pmatrix}.
\end{equation}
The energy associated to \eqref{2nd W T_mu} is given in terms of the symmetrizer \eqref{2nd Sym Q} by
\begin{equation}\label{tilde Energy s 3}
	\tilde{\mathscr{E}}_{s}(\bold{W}) 
	: =
	\big(J^s \bold{W}, \mathscr{Q}(\bold{U}_1,D)J^s \bold{W}\big)_{L^2}. 
\end{equation}
\begin{prop}\label{2nd diff prop}
	Take $s> \frac{3}{2}$, $\ve, \mu \in (0,1)$ and $\beta>0$.  Let $(\eta_1,u_1),(\eta_2,u_2) \in C([0,T_0]: X^s_{\beta,\mu} (\mathbb{R}))$ be two solutions of \eqref{2nd Whitham Boussinesq} on a time interval $[0,T_0]$ for some $T_0 >0$. Moreover,    assume there exist $h_0 \in (0,1)$ and $h_1>0$ such that % for $i\in\{1,2\}$
	\begin{equation}\label{cond. sol. 4}
		h_0 - 1 \leq \eta_1(x,t), \quad   \forall (x,t) \in \mathbb{R} \times [0,T_0] \quad  \text{and} \quad \sup\limits_{t \in [0,T_0]}\|(\eta_1, 	u_1)\|_{H^s\times H^s} \leq h_1.
	\end{equation}

	Define the difference to be $\bold{W}=(\psi,w) = (\eta_1 - \eta_2, u_1 - u_2)$. Then, for the energy defined by \eqref{tilde Energy s 3}, there holds
	
	\begin{equation}\label{Energy last one 0}
		\frac{d}{dt} \tilde{\mathscr{E}}_0(\bold{W}) \lesssim_{\beta} \max \limits_{i = 1,2} \|(\eta_i,u_i)\|_{V^{s}_{\sqrt{\mu}}}  \|(\psi,w)\|_{X^0_{\beta,\mu}}^2,
	\end{equation}
	and
	\begin{equation}\label{Diff Equiv 2 2}
		\|(\psi,w)\|^2_{X^0_{\beta,\mu}} \lesssim \tilde{\mathscr{E}}_0(\bold{W}) \lesssim \|(\psi,w)\|^2_{X^0_{\beta,\mu}}.
	\end{equation}
	
		Furthermore, we have the following estimate at the  $X^s_{\beta,\mu}-$ level:
	\begin{equation}\label{Diff Energy 2 2 s}
		\frac{d}{dt} \mathscr{\tilde{E}}_s(\bold{W}) \lesssim_{\beta}  |\big(J^s \bold{F}, \mathscr{Q}(\bold{U}_1,D)  J^s \bold{W}\big)_{L^2}| +  \max \limits_{i = 1,2} \|(\eta_i,u_i)\|_{V^{s}_{\sqrt{\mu}}}  \|(\psi,w)\|_{X^s_{\beta,\mu}}^2,
	\end{equation}
	and
	\begin{equation}\label{Diff Equiv 2 2s}
		\|(\psi,w)\|^2_{X^s_{\beta,\mu}} \lesssim \mathscr{\tilde{E}}_s(\bold{W}) \lesssim \|(\psi,w)\|^2_{X^s_{\beta,\mu}}.
	\end{equation}
\end{prop}

\begin{proof} By previous arguments, we note that the proofs of \eqref{Diff Equiv 2 2} and \eqref{Diff Equiv 2 2s} are similar to the proof of \eqref{2nd Equivalence Whitham Boussinesq}.
	
	Moreover, we will only prove \eqref{Energy last one 0} since the control of \eqref{Diff Energy 2 2 s} follows by the proof of Proposition \ref{2nd A priori Whitham Boussinesq}.

	We will now prove \eqref{Energy last one 0}. Then we first use \eqref{2nd W T_mu} and the self-adjointness of $\mathscr{Q}(\bold{U}_1,D)$ to write
	\begin{align*}
		\frac{1}{2} \frac{d}{dt}  \tilde{\mathscr{E}}_{0}(\bold{W})  
		&= 
		\frac{1}{2}\big(  \bold{W}, ( \partial_t \mathscr{Q}(\bold{U}_1,D))    \bold{W}\big)_{L^2}
		%\\
		%& \hspace{0.4cm}
		+
		\big( \bold{F}, \mathscr{Q}(\bold{U}_1,D)    \bold{W}\big)_{L^2}
		\\
		&
		\hspace{0.4cm}
		-
		\big(  \mathscr{M}(\bold{U}_1,D) \bold{W} , \mathscr{Q}(\bold{U}_1,D)     \bold{W}\big)_{L^2}
		\\
		& = : \mathscr{I} - \mathscr{II} - \mathscr{III}.
	\end{align*}
	\underline{Control of $\mathscr{I}$.} By \eqref{2nd W T_mu}, Hölder's inequality,  the Sobolev embedding and \eqref{cond. sol. 4} we deduce
	\begin{align*}
		|\mathscr{I}|
		= 
		\frac{1}{2} |\big(    w, (\partial_t \eta_1)  w \big)_{L^2}|
		\lesssim 
		\|u_1\|_{H^s}(1+ \| \eta_1 \|_{H^s}) \| w \|_{L^2}^2 \lesssim \max \limits_{i = 1,2} \|(\eta_i,u_i)\|_{V^{s}_{\sqrt{\mu}}}  \| (\psi,w)\|_{X^0_{\beta,\mu}}^2.
	\end{align*}
	%
	%
	%
	%Then to conclude the first estimate we use  \eqref{J_mu}.\\
	
	\noindent
	\underline{Control of $\mathscr{II}$.} The contribution from the source term is given by,
	\begin{align*}
		\mathscr{II} 
		& =
		\big( w\partial_x \eta_2, \mathcal{K}_{\mu}(D)   \psi \big)_{L^2}  
		+
		\big(  \psi \partial_x \eta_2,  \mathcal{K}_{\mu}(D)    \psi \big)_{L^2}  
		\\
		&
		\hspace{0.4cm}
		+
		\big(   w\partial_x u_2,   w \big)_{L^2}  
		+
		\big(   \mathcal{T}_{\mu}(D)(w\partial_x u_2),  \eta_1 w \big)_{L^2}  
		\\ 
		& 
		=:
		\mathscr{II}_1 + \mathscr{II}_2 + \mathscr{II}_3 +  \mathscr{II}_4.
	\end{align*}
	\textit{Control of $\mathscr{II}_1 + \mathscr{II}_2$.} We first apply the Cauchy-Schwarz inequality, \eqref{sqrt K}, and the Sobolev embedding to deduce that for  $s>\frac{3}{2}$
	\begin{align*}
		|\mathscr{II}_1|+|\mathscr{II}_2| 
		&
		\lesssim
		(\| w\|_{L^2}  + \| \psi \|_{L^2} )\| \eta_2\|_{H^s}\|\mathcal{K}_{\mu}(D)\psi \|_{L^2}
		\\
		& 
		\lesssim \max \limits_{i = 1,2} \|(\eta_i,u_i)\|_{V^{s}_{\sqrt{\mu}}} 
		%|
		%\lesssim
		\|(\psi,w)\|_{X^0_{\beta,\mu}}^2.	
	\end{align*}
	
	\noindent
	\textit{Control of $\mathscr{II}_3 + \mathscr{II}_4$.} Both terms are treated with the Cauchy-Schwatz inequality, \eqref{boundedness of T} and the Sobolev embedding. Consequently, for $s>\frac{3}{2}$ there holds
	\begin{align*}
		|\mathscr{II}_3|  + |\mathscr{II}_4| 
		\lesssim
		(1+\|\eta_1\|_{H^s})\|u_2\|_{H^s} \| w\|_{L^2}^2
		 \lesssim 	
		 \max \limits_{i = 1,2} \|(\eta_i,u_i)\|_{V^{s}_{\sqrt{\mu}}}  \|(\psi,w)\|_{X^0_{\beta,\mu}}^2.
	\end{align*}

	Gathering all these estimates yields
	\begin{equation*}
		|\mathscr{II}|  \lesssim \max \limits_{i = 1,2} \|(\eta_i,u_i)\|_{V^{s}_{\sqrt{\mu}}}  \| (\psi, w)\|_{X^0_{\beta,\mu}}^2.
	\end{equation*}

	\noindent
	\underline{Control of $\mathscr{III}$.} The symmetrized term $\mathscr{III}$ reads:
	\begin{align*}
		 \mathscr{III} = &
		\big(u_1 \partial_x \psi, \mathcal{K}_{\mu}(D) \psi \big)_{L^2} 
	%	\\
	%	& 
	%	\hspace{0.5cm}
		+
		\big(  (1+\mathcal{T}_{\mu}(D)\eta_1) \partial_x w, \mathcal{T}^{-1}_{\mu}(D)\mathcal{K}_{\mu}(D)  \psi \big)_{L^2} 
		\\
		& \hspace{0.0cm}
		+
		\big( \mathcal{K}_{\mu} (D)  \partial_x \psi, (\mathcal{T}_{\mu}^{-1}(D) +\eta_1)   w \big)_{L^2} 
		+
		\big(\mathcal{T}_{\mu}(D) ( u_1 \partial_x w), (\mathcal{T}_{\mu}^{-1}(D) + \eta_1) w \big)_{L^2} 
		\\
		& =  \mathscr{A}_{11} + \mathscr{A}_{12} + \mathscr{A}_{21} + \mathscr{A}_{22}.
	\end{align*}

	\noindent
	\textit{Control of $\mathscr{A}_{11}$.} We decompose $\mathscr{A}_{11}$ as
	\begin{align*}
		\mathscr{A}_{11}
		& =
		\big(u_1 \partial_x \psi, ((\chi^{(1)}_{\mu})^2\mathcal{K}_{\mu})(D) \psi \big)_{L^2} 
		+
		\big(u_1 \partial_x \psi, ((\chi_{\mu}^{(2)})^2 (\sigma_{\mu,\frac{1}{2}})^2)(D)\psi \big)_{L^2}
		\\
		&
		\hspace{0.5cm}
		-
		\big(u_1 \partial_x \psi, ((\chi^{(2)}_{\mu})^2(\sigma_{\mu,0})^2)(D)\psi \big)_{L^2},
	\end{align*}
 	where we have divided the multiplier $\mathcal{K}_{\mu}(D)$ into three pieces in the same way as we did in \eqref{Split multiplier K}. We may therefore  apply the same estimates as for $A_4^3$ in the proof of Proposition \ref{Compactness L2}, where we change the role of $\psi$ and $w$ to obtain
	\begin{align*}
		|\mathscr{A}_{11}|
		%& 
		=
		\big(u_1 \partial_x \psi,\mathcal{K}_{\mu}(D)  \psi \big)_{L^2} 
		%+
		%\big(u_1 \sqrt{\mathcal{K}_{\mu}}(D)\partial^2_x \xi,\sqrt{\mathcal{K}_{\mu}}(D)\partial_x  \xi \big)_{L^2} 
%		\\
		 \lesssim_{\beta} \|u_1\|_{H^s}(
		\| \psi \|_{L^2}^2+\sqrt{\beta}\mu^{\frac{1}{4}} \| \psi\|_{H^{\frac{1}{2}}}^2).%\|(\xi,w)\|_{V^1_{\sqrt{\mu}}}^2.
	\end{align*}
	Then use inequality \eqref{interpolation} to conclude that
	\begin{equation*}
		|\mathscr{A}_{11}| \lesssim  \max \limits_{i = 1,2} \|(\eta_i,u_i)\|_{V^{s}_{\sqrt{\mu}}} \| (\psi, w)\|_{X^0_{\beta,\mu}}^2.
	\end{equation*}
	
	%obtain $\|(\xi,w)\|_{V^1_{\sqrt{\mu}}} \lesssim_{\beta} \|(\xi,w)\|_{X^1_{\beta,\mu}}$.\\

	\noindent
	\textit{Control of $\mathscr{A}_{12} + \mathscr{A}_{21}$.} Treating the off-diagonal terms we first observe by integrating by parts that
	\begin{align*}
		\mathscr{A}_{12} 
		& =
			%\big( (\partial_x \eta_1) \partial_x w , \mathcal{K}_{\mu}(D)\partial_x  \xi\big)_{L^2} 
			-		
			\big( (\partial_x\eta_1) w ,  
			\mathcal{K}_{\mu}(D) \psi \big)_{L^2} 
		%	\\
		%	&
		%	\hspace{0.5cm}
			 -\mathscr{A}_{21}. 
			%\big( (1+ \mathcal{T}_{\mu}(D) \eta_1 )w , \mathcal{T}^{-1}_{\mu}(D)\mathcal{K}_{\mu}(D)\partial_x  \psi \big)_{L^2}.
	\end{align*}
	Therefore, we may apply Hölder's inequality and the Sobolev embedding to deduce\\
	%
	%
	%@
	\begin{equation*}
		|\mathscr{A}_{12} + \mathscr{A}_{21}|\lesssim \|\eta_1\|_{H^s} \|(\psi,w)\|_{X^0_{\beta,\mu}}.
	\end{equation*}

	\noindent
	\textit{Control of $\mathscr{A}_{22}$.} We decompose $\mathscr{A}_{22}$ into two terms
	\begin{align*}
		\mathscr{A}_{22} 
		& = 
		\big(   u_1 \partial_x w,  w \big)_{L^2} 
		+
		\big( \mathcal{T}_{\mu}(D) ( u_1 \partial_x w),  \eta_1 w \big)_{L^2} 
		\\
		&=
		\mathscr{A}_{22}^1 
		+
		\mathscr{A}_{22}^2.
	\end{align*}
	We see that $\mathscr{A}_{22}^1 $ is easily treated by the Cauchy-Schwarz inequality, integration by parts, the Sobolev embedding, and \eqref{cond. sol. 4}. Indeed, there holds
	\begin{equation*}
		|\mathscr{A}_{22}^1 | \lesssim \|u_1\|_{H^s}  \|w\|_{L^2}^2.
	\end{equation*}
	Next,  we decompose $\mathscr{A}_{22}^2$  into three parts
	\begin{align*}
		\mathscr{A}_{22}^2
		& = 
	%	\big( \mathcal{T}_{\mu}(D) ( \partial_x u_1 \partial_x w),  \eta_1 \partial_xw \big)_{L^2} 
	%	+
		 \big(  [\sqrt{\mathcal{T}_{\mu}}(D), u_1] \partial_x w, \sqrt{\mathcal{T}_{\mu}}(D) (\eta_1 w) \big)_{L^2}
		+
		 \big( u_1 \sqrt{\mathcal{T}_{\mu}}(D)  \partial_x w, [ \sqrt{\mathcal{T}_{\mu}}(D),\eta_1] w \big)_{L^2}
		 \\
		 &
		 \hspace{0.5cm}
		+
		\big( u_1 \sqrt{\mathcal{T}_{\mu}}(D)  \partial_x w,  \eta_1 \sqrt{\mathcal{T}_{\mu}}(D) w \big)_{L^2}.
		\\
		& =
		\mathscr{A}_{22}^{2,1} + \mathscr{A}_{22}^{2,2} + \mathscr{A}_{22}^{2,3}.
	\end{align*}
	For $\mathscr{A}_{22}^{2,1}$, we simply apply Hölder's inequality, \eqref{Bessel Commutator T_mu J=1}, \eqref{boundedness of T}, the Sobolev embedding  to find that
%	Next, we treat similarly $ \mathscr{A}_{22}^{2,2}$, but also use  to find that
	%
	%
	%
	\begin{align*}
		|\mathscr{A}_{22}^{2,1}|
		\leq 
		\|[\sqrt{\mathcal{T}_{\mu}}(D), u_1] \partial_x w\|_{L^2} \|\sqrt{\mathcal{T}_{\mu}}(D) (\eta_1 w) \| _{L^2}
		\lesssim \|w\|_{L^2}^2.
	\end{align*}
	For $\mathscr{A}_{22}^{2,2}$, we first remark that 
	\begin{equation}\label{remark 100}
		 \| [ \sqrt{\mathcal{T}_{\mu}}(D),\eta_1] w \|_{L^2} \lesssim \|\eta_1\|_{L^{\infty}}\|w\|_{L^2},
	\end{equation}
	simply by using Hölder's inequality and \eqref{boundedness of T}. Then after integrating by parts, we use Hölder's inequality, the Sobolev embedding,  \eqref{Commutator dx T}, \eqref{cond. sol. 4}, and \eqref{remark 100}  to deduce that
	\begin{align*}
		|\mathscr{A}_{22}^{2,2}|
		& \leq 
		 \| \partial_x u_1 \|_{L^{\infty}} \|\sqrt{\mathcal{T}_{\mu}}(D)  w\|_{L^2} \| [ \sqrt{\mathcal{T}_{\mu}}(D),\eta_1] w \|_{L^2} 
		 \\
		 & 
		 \hspace{0.5cm}+ \|  u_1 \|_{L^{\infty}} \|\sqrt{\mathcal{T}_{\mu}}(D)   w\|_{L^2} \| \partial_x [ \sqrt{\mathcal{T}_{\mu}}(D),\eta_1] w \|_{L^2}
		 \\
		 & 
		 \lesssim \max \limits_{i = 1,2} \|(\eta_i,u_i)\|_{V^{s}_{\sqrt{\mu}}}  \|w\|_{L^2}^2.
	\end{align*}
	Lastly, we use integration by parts, then apply Hölder's inequality, \eqref{boundedness of T}, the Sobolev embedding, and \eqref{cond. sol. 4} to obtain that
	\begin{align*}
		|\mathscr{A}_{22}^{2,3}|
		\leq \frac{1}{2} \|\partial_x( u_1\eta_1)\|_{L^{\infty}}\| \sqrt{\mathcal{T}_{\mu}}(D)   w\|^2_{L^2}  \lesssim  \max \limits_{i = 1,2} \|(\eta_i,u_i)\|_{V^{s}_{\sqrt{\mu}}} \| w \|_{L^2}^2.%\sqrt{\mathcal{T}_{\mu}}(D) \partial_xw \big)_{L^2}
	\end{align*}

	We may now gather these estimates to conclude that
	\begin{equation*}
		|\mathscr{A}_{22} | \lesssim  \max \limits_{i = 1,2} \|(\eta_i,u_i)\|_{V^{s}_{\sqrt{\mu}}}  \| (\psi,w)\|_{X^0_{\beta,\mu}}^2,
	\end{equation*}
	and as a result the proof of Proposition \ref{2nd diff prop} is now complete.
\end{proof}

\section{Proof of Theorem \ref{Well-posedness long time full dispersion}}\label{Main proof of thm}

\begin{proof}
The proof is divided into eigth steps, utilizing the results above. \\ 

\noindent

\noindent
\underline{Step $1$:} \textit{Existence of solutions for a regularized system.} Let $s > \frac{1}{2}$, $\delta>0$, and $\varphi_{\delta}(D)$  be as given in Definition \ref{regularisation}. Then, for initial data on the form $\varphi_{\delta} (D)\bold{U}_{0} :=  (\varphi_{\delta}(D) \eta_0,\varphi_{\delta} (D)u_0)$, we claim that there exist $c_{\beta}>0$, and a time 
\begin{equation}\label{T delta}
	0
	<
	T_{\delta}
	: =
	T_{\delta}\big{(}\|(\eta_0,u_0)\|_{V^s_{\sqrt{\mu}}}\big{)}
	=
	 \frac{c_{\beta}\delta}{1+\|(\eta_0,u_0)\|_{V^s_{\sqrt{\mu}}}}
\end{equation}
such that $\bold{U}^{\delta}:=(\eta^{\delta}, u^{\delta})^T\in C([0,T_{\delta}]; V^s_{\sqrt{\mu}}(\mathbb{R}))%\cap C^{\infty}(\mathbb{R} \times (0,T_{\delta}))
$ is a unique solution of the regularised Cauchy problem:
\begin{equation}\label{regularised Cauchy problem}
	\begin{cases}
		\partial_t 	\eta^{\delta}  + 	u^{\delta} \partial_x \varphi_{\delta}(D) \eta^{\delta}
		+ (\mathcal{K}_{\mu}(D) + \eta^{\delta}) \partial_x	\varphi_{\delta}(D) u^{\delta}= 0
		\\
		\partial_t u^{\delta} +	\partial_x\varphi_{\delta}(D) \eta^{\delta} +  u^{\delta} \partial_x	\varphi_{\delta}(D) u^{\delta} = 0
	\end{cases}
\end{equation}

The proof of the existence of a unique solution is a consequence of the contraction mapping principle. First, define the notation $\varphi_{\delta} (D)\bold{U}^{\delta}(t) := (\varphi_{\delta} (D)\eta^{\delta},\varphi_{\delta} (D) u^{\delta})^T(t)$. Then we shall show for $M$ defined by \eqref{M 1} that 
\begin{align}\label{integral equation}
	\bold{U}^{\delta}(t)
  	  = 
    	\varphi_{\delta} (D)\bold{U}_{0}
    	-
    	\int_0^t   M( \bold{U}^{\delta}(s),D) \cdot
    	\varphi_{\delta} (D)\bold{U}^{\delta}(s) \: ds
	= :
	 \Phi_{\bold{U}_{0}}(\bold{U}^{\delta})(t)
\end{align}
defines a contraction  on the compete metric space
\begin{align*}
    Y^s( T, a) : = \big{\{} %\bold{f}: \mathbb{R} \times [0,T] \rightarrow \mathbb{R}^2 \: 
     (f^1,f^2) \in C([0,T]; V_{\sqrt{\mu}}^s(\mathbb{R})) \: : \:  \| (f^1,f^2) \|_{L^{\infty}_TV_{\sqrt{\mu}}^s}\leq a \big{\}}.
\end{align*}
\noindent
\underline{Step $1.1$:}   \textit{$\Phi$ is well-defined  on  $Y^s(T,a)$.} Indeed, let $\bold{U}^{\delta} = (\eta^{\delta}, u^{\delta}) \in Y^s(T, a)$ and let $\Phi := (\Phi^1, \Phi^2)^T$. Then for $\Phi^1$ we have by \eqref{sqrt K}, \eqref{Prod. Kato-Ponce}, the Sobolev embedding and \eqref{reg 3} that 
\begin{align*}
		\| \Phi^1(\bold{U}^{\delta})\|_{L^{\infty}_{T}H^s} 
		& 
		\leq \|\eta_{0}\|_{H^s}
		+
		T\big{(}
		\delta^{-1}\|u^{\delta}\|_{L^{\infty}_{T}H^s} \|\eta^{\delta}\|_{L^{\infty}_{T}H^s} 
		\\
		&
		\hspace{0.5cm}+ \delta^{-\frac{3}{2}} \| \mathcal{K}_{\mu}(D)u^{\delta}\|_{L^{\infty}_{T}H^{s-\frac{1}{2}}}+ \delta^{-1}\|\eta^{\delta}\|_{L^{\infty}_T H^s}\|u^{\delta}\|_{L^{\infty}_{T}H^s}
		\big{)}
		\\
		&
		\leq  \|\eta_{0}\|_{H^s}+ T(2a^2 \delta^{-1}+ c_{\beta}a\delta^{-\frac{3}{2}}).
\end{align*}
Similarly, for the second coordinate there holds
\begin{align*}
	\| \Phi^2(\bold{U}^{\delta})\|_{H^s}  + \mu^{\frac{1}{4}} \| D^{\frac{1}{2}}\Phi^2(\bold{U}^{\delta}) \|_{H^s}  
%	&
	\leq \|u_0\|_{H^s} + \mu^{\frac{1}{4}}\|D^{\frac{1}{4}}u_0\|_{H^s} 
	%\\
	%& \hspace{0.5cm}
	+
	T
	c_{\beta}(a+a^2)( \delta^{-1}+ \delta^{-\frac{3}{2}}).
\end{align*}
Gathering these estimates, with $\| (\eta_0, u_0)| |_{V_{\sqrt{\mu}}^s} \leq \frac{a}{2}$ and $T \leq \frac{\tilde{c}_{\beta}\delta}{2(a+1)(1+\delta^{-\frac{1}{2}})}$ implies
\begin{align*}
	\| (\Phi^1(\bold{U}^{\delta}),\Phi^2(\bold{U}^{\delta}) )\|_{V_{\sqrt{\mu}}^s} 
	\leq a.%
%	| |  (\varphi_{\delta} (D)\eta_0,\varphi_{\delta} (D) u_0)| |_{V_{\sqrt{\mu}}^s}
%	+
%	T( 3a^2 \delta^{-1} + a\delta^{-1} + a\delta^{-2}).
\end{align*}
% 
% 
% 
%Thus, for $\| (\eta^{\delta}_0, u^{\delta}_0)| |_{V_{\sqrt{\mu}}^s} \leq \frac{a}{2}$ and $T \leq \frac{\delta}{6a + 2+ 2\delta}$ implies $\Phi(\bold{U}^{\delta}) \in Y^s(T, a)$. 
\noindent 
\underline{Step $1.2$:} \textit{$ \Phi$ defines  a contraction map on $Y^s(T, a)$.} Let  $\bold{U}^{\delta}_1 = (\eta_1^{\delta},u_1^{\delta})^T,\bold{U}^{\delta}_2 = (\eta_2^{\delta},u_2^{\delta})^T \in Y^s(T,a)$. Then if we define $\bold{W}^{\delta} =(\psi^{\delta},w^{\delta})^T=\bold{U}^{\delta}_1 - \bold{U}^{\delta}_2 $ we have that
\begin{equation*}
	\Phi(\bold{U}^{\delta}_1) - \Phi(\bold{U}^{\delta}_2) = \int_0^t  \big{(}M(\bold{U}_1^{\delta},D)\cdot \varphi_{\delta}(D)\bold{W}^{\delta} + (M(\bold{U}_2^{\delta},D)-M(\bold{U}_1^{\delta},D))\cdot \varphi_{\delta}(D)\bold{U}^{\delta}_2 \big{)}\: ds.
\end{equation*}
Reapplying the same estimates as in the previous step, we conclude that
\begin{align*}
	%	& \hspace{-2.5cm}
		\|(\Phi^1(\bold{U}_1^{\delta})-\Phi^1(\bold{U}_2^{\delta}),\Phi^2(\bold{U}_1^{\delta})-\Phi^2(\bold{U}_2^{\delta}) )\|_{V_{\sqrt{\mu}}^s} 
	%	\\
	%	& \hspace{1cm}\lesssim_{\beta} (\delta^{-1} +\delta^{-\frac{3}{2}})T
	<	\|(\psi^{\delta},w^{\delta})\|_{V^s_{\sqrt{\mu}}}.
\end{align*}
by choosing $T\lesssim T_{\delta}(a)$.

We  conclude by the Fixed Point Theorem that there exist a unique solution of \eqref{regularised Cauchy problem} in $C([0,T_{\delta}]; V^s_{\sqrt{\mu}}(\mathbb{R}))$.  %a unique fixed point of the  operator $\Phi = \Phi_{\bold{U_0}}$. \\

\begin{remark}\label{long time flow map}
	A consequence of Step $1$, is the continuity of the flow map associated to \eqref{regularised Cauchy problem}. But this is  only for a time $T_{\delta} \lesssim_{\beta}\delta$, and is therefore not usefull for the limit equation.\\
\end{remark}

% $\boldmath{U}^{\delta} = (\eta^{\delta},u^{\delta})\in C([0,T_{\delta}]; V^s_{\sqrt{\mu}}(\R))$  by  Banach Fixed Point Theorem.\\ 

%\noindent
%\underline{Step $2.3$:} The flow map $F^{\delta}$ is continuous. Let 
%
%
%

\noindent
\underline{Step 2:} \textit{The blow-up alternative.} We define the maximal time of existence to be
 $$T^{\star}_{\delta} = \sup\Big{\{} T_{\delta}>0 \: : \: \exists ! \:  \bold{U}^{\delta} = (\eta^{\delta},u^{\delta})^T \: \text{solution of} \: \eqref{regularised Cauchy problem}\: \text{in} \:  C([0,T_{\delta}]; V^s_{\sqrt{\mu}}(\R))\Big{\}}.$$
Then we claim that the solution of \eqref{regularised Cauchy problem} satisfy the blow-up alternative:
\begin{equation}\label{blow-up alt Boussinesq}
	\text{If} \quad T_{\delta}^{\star}< \infty, \quad \text{then} \quad \lim\limits_{t \nearrow T_{\delta}^{\star}} \| (\eta^{\delta}, u^{\delta})(t)\|_{V^s_{\sqrt{\mu}}} = \infty.
\end{equation} 

%The proof is based on the proof of Theorem $5.14$ in \cite{Iorio2001}.
 First, we argue by contradiction that $T^{\star}_{\delta}<\infty $ and there exist $A\in \R^{+}$ such that
\begin{equation}\label{with M}
	\sup\limits_{t\in [0,T^{\star}_{\delta})} \|(\eta^{\delta}, u^{\delta})(t)\|_{V^s_{\sqrt{\mu}}} = A.
\end{equation}
We use \eqref{T delta} to define $\tau_{\delta,A} =T^{\star}_{\delta}- \frac{T_{\delta}(A)}{2}$. Then we have that 
$$a: = \|(\eta^{\delta}, u^{\delta})(\tau_{\delta,A})\|_{V^s_{\sqrt{\mu}}} \leq A.$$
Therefore, if we let  $\bold{V}^{\delta}_0 = (\eta^{\delta}, u^{\delta})^T(\tau_{\delta,A})$ serve as initial data, then \eqref{regularised Cauchy problem} has a unique solution given by 
\begin{align}\label{V delta}
	\bold{V}^{\delta}(t)
	= 
	\bold{V}^{\delta}_0
	-
	\int_0^t   M( \bold{V}^{\delta}(s),D) \cdot
	\varphi_{\delta} (D)\bold{V}^{\delta}(s) \: ds,
\end{align}
with  $\bold{V}^{\delta}\in C([0,T_{\delta}(a)];V^s_{\sqrt{\mu}}(\R)) $. Here, $T_{\delta}(a)$ is given by \eqref{T delta} due to Step $1$. Moreover, we observe that $T_{\delta}(a)\geq T_{\delta}(A)$ by definition \eqref{T delta}, and implies $\tau_{\delta,A}+ T_{\delta}(a) \geq T^{\star}_{\delta} + \frac{T_{\delta}(A)}{2}$. Thus, we define the extension of $\bold{U}^{\delta}=(\eta^{\delta}, u^{\delta})^T$ by the function
\begin{equation*}
	\bold{Z}^{\delta}(t) 
	=
	\begin{cases}
		\bold{U}^{\delta}(t), \quad& \text{if} \quad 0\leq t < \tau_{\delta,A}\\
		\bold{V}^{\delta}(t-\tau_{\delta,A}), \quad& \text{if} \quad \tau_{\delta,A} \leq t \leq \tau_{\delta,A}+T_{\delta}(a),
	\end{cases}
\end{equation*}
and one can verify that it is a solution of \eqref{regularised Cauchy problem} for all $t \in [0,T_{\delta}^{\star} + \frac{T_{\delta}(A)}{2}] \subset [ 0 , \tau_{\delta,A}+T_{\delta}(a)]$. This contradicts the definition of $T^{\star}_{\delta}$. Thus, we conclude that if $T^{\star}_{\delta}<\infty$, then necessarily $A = \infty$ in \eqref{with M}, and implies
\begin{equation}\label{limsup}
	\limsup\limits_{t\nearrow T^{\star}_{\delta}} \|(\eta^{\delta},u^{\delta}(t))\|_{V^s_{\sqrt{\mu}}}=\infty.
\end{equation}

To conlcude the proof of the claim, we use \eqref{limsup} to verify that for all $R>0$ there exists an open interval $(t_R,T^{\star}_{\delta})$ such that $\|(\eta^{\delta},u^{\delta}(t))\|_{V^s_{\sqrt{\mu}}}>R$, for all $t \in (t_R,T^{\star}_{\delta})$. Indeed, we argue by contradiction that there exists $R\in \R^+$ such that for all $0<t_R<T^{\star}$, we have
\begin{equation}\label{bound contradiction}
\|(\eta^{\delta},u^{\delta})(t)\|_{V^s_{\sqrt{\mu}}} \leq R,
\end{equation}
for some  $t\in (t_R,T^{\star}_{\delta})$. By \eqref{limsup} there is a time such that $\tau_{R,0}>T^{\star}_{\delta}-\frac{T_{\delta}(R)}{2}$
 and satisfying
\begin{equation*}%\label{time t0}
\|(\eta^{\delta},u^{\delta})(\tau_{R,0})\|_{V^s_{\sqrt{\mu}}} >R.
\end{equation*}
On the other hand, by assumption \eqref{bound contradiction} we can take $t_R = \tau_{R,0}$ and use the fact that there is a time $\tau_{R,1} \in (t_R,T^{\star}_{\delta})$ such that
\begin{equation*}%\label{time t1}
	\|(\eta^{\delta},u^{\delta})(\tau_{R,1})\|_{V^s_{\sqrt{\mu}}} \leq R.
\end{equation*}
Thus, by the same argument as above we can take  $(\eta^{\delta},u^{\delta})(\tau_{R,1})$ as initial data of \eqref{regularised Cauchy problem} to find an extended solution defined on $[0,T^{\star}_{\delta} + \frac{T_{\delta}(R)}{2}]\subset [0,\tau_{R,1} + T_{\delta}(R)]$. This contradicts the  definition of $T_{\delta}^{\star}$. As a result, we conclude that \eqref{blow-up alt Boussinesq} holds true.\\

\noindent
\underline{Step $3$:} \textit{The existence time is independent of $\delta>0$.} We claim that there exists
$$T = \frac{1}{k_{\beta}^1 \|(\zeta_0, v_0)\|_{V^s_{\sqrt{\mu}}}},$$
 as in \eqref{Time T}, such that the regularized solution $\ve(\zeta^{\delta},  v^{\delta}) = (\eta^{\delta},u^{\delta})$ exists on the interval $[0,\frac{T}{\ve}]$.

The proof relies on a bootstrap argument similar to the proof of Lemma $5.1$ in \cite{KalischPilod2019}.  In fact, the long time existence is a direct consequence of the following lemma. \\

\begin{lemma}\label{Lemma time of existence}
	Let $ s > 2$ and  $\ve$ be as in \eqref{eps beta}. Let $ (\eta^{\delta}, u^{\delta})=\ve(\zeta^{\delta}, v^{\delta}) \in C([0,T_{\delta}^{\star}); V^s_{\sqrt{\mu}}(\mathbb{R}))$ be a solution of \eqref{regularised Cauchy problem} with initial data $\ve(\zeta_0, v_0)= (\eta_0, u_0) \in V^s_{\sqrt{\mu}}(\mathbb{R})$,  defined on its maximal time of existence and satisfying the blow-up alternative \eqref{blow-up alt Boussinesq}. Moreover, let $\eta_0  = \ve\zeta_0$ satisfy either the non-cavitation condition \eqref{NonCav} or the $\beta-$dependent surface condition \eqref{Beta NonCav}, depending on whether $\beta \geq \frac{1}{3}$ or $0<\beta < \frac{1}{3}$, respectivly. Then there  exists a time %Moreover, let the solution be defined on its maximal time of existence and satisfying the blow-up alternative:
	% 
	% 
	% 
%	\begin{equation}\label{blow-up alt Boussinesq}
	%	\text{If} \quad T^{\star}< \infty, \quad \text{then} \quad \lim\limits_{t \nearrow T^{\star}} \| (\eta^{\delta}, u^{\delta})(t)\|_{V^s_{\sqrt{\ve}}} = \infty.
	%\end{equation}
	% 
	% 
	% 
	
	%
	%
	%
	\begin{equation}\label{time T0}
		T_0 = \frac{1}{k_{\beta}^{1}\| (\eta_0, u_0)\|_{V^s_{\sqrt{\mu}}}},
	\end{equation}
	such that $T^{\star}_{\delta}>T_0$ and
	\begin{equation}\label{Bound on solution - data}
		\sup\limits_{t \in [0,T_0]} \| (\eta^{\delta}, u^{\delta})(t)\|_{V^s_{\sqrt{\mu}}} \leq 4 k^{2}_{\beta}\|  (\eta_0^{\delta},  u_0^{\delta}) \|_{V^s_{\sqrt{\mu}}}.
	\end{equation}
	The constants are on the form
	$$k^{2}_{\beta} = \frac{c_{\beta}^{2}}{c_{\beta}^{1}} \quad \text{and} 
	\quad 
	k_{\beta}^{1} = 
	\begin{cases}
		\frac{C_1}{\beta}  \: \:  \quad \quad 
		\text{for} \: \:   0<\beta<\frac{1}{3} 
		\\
		C_2 \beta^2  
		\: \:  \: \quad\text{for} \: \:   \beta\geq \frac{1}{3}
	\end{cases}$$
	where $C_1$ and $C_2$ are two positive constants to be fixed in the proof. 
\end{lemma}

\begin{proof}
We define the set
\begin{equation}\label{tube}
	\tilde{T}_{\delta} = \sup \limits\Big\{  T_{\delta} \in (0, T_{\delta}^{\star}) \: : \: \sup\limits_{t \in [0,T_{\delta}]}\|(\eta^{\delta}, u^{\delta})(t)\|_{V^s_{\sqrt{\mu}}} \leq4k^{2}_{\beta}\|(\eta_0^{\delta},  u^{\delta}_0)\|_{V^s_{\sqrt{\mu}}}\Big\}.
\end{equation}
Then we first note that $\tilde{T} _{\delta}< T^{\star}_{\delta}$, or else it would contradict the blow-up alternative \eqref{blow-up alt Boussinesq}. For the proof we argue by contradiction that $\tilde{T}_{\delta} \leq T_0$.

The main idea is to improve the estimate given in \eqref{tube}.  First, we verify that the solution $(\eta^{\delta},u^{\delta})$ satisfy \eqref{cond. sol.}.  Indeed, recalling assumption \eqref{eps beta}:
$$0<\ve \leq \frac{1 }{k^{2}_{\beta} \|(\zeta_0, v_0 )\|_{V^s_{\sqrt{\mu}}}},$$
implies
\begin{align}\label{bound U in Hs}
	\|(\eta^{\delta},u^{\delta})\|_{H^s} \leq k^{2}_{\beta} \|(\eta_0^{\delta},u_0^{\delta})\|_{V_{\sqrt{\mu}}^s}  = 4\ve k^{2}_{\beta}  \|(\zeta_0, v_0)\|_{V^s_{\sqrt{\mu}}} \leq 4,
\end{align}
for all $t \in [0, \tilde{T}_{\delta}]$. Next, the solution $(\eta^{\delta},u^{\delta})$ satisfy the non-cavitation condition. We will prove this as a consequence of two claims.

\underline{Claim 1:} If $\eta_0$ satisfies the non-cavitation condition, then the same is true for $\eta^{\delta}_0:= \varphi_{\delta}(D)\eta_0$.  Indeed,  by the Sobolev embedding $H^{\frac{1}{2}^+}(\mathbb{R}) \hookrightarrow L^{\infty}(\mathbb{R})$ and \eqref{reg 4} to deduce that
$$\|\eta_0^{\delta} - \eta_0\|_{L^{\infty}}  \lesssim \|\varphi_{\delta}(D)\eta_0 - \eta_0\|_{H^{\frac{1}{2}^{+}}}  \underset{\delta \rightarrow 0}{=} o(\delta),$$
since $s-1>\frac{1}{2}$.

\underline{Claim 2:} We have the following bound on $\partial_t \eta^{\delta}$:
\begin{equation*}
	 \sup \limits_{\tau \in [0, \tilde{T}_{\delta}]} |\partial_t \eta^{\delta}(\tau) | \leq k_{\beta}^2\beta^2.
\end{equation*}
Indeed,  by  similar argument as for \eqref{Dependence in beta}, we use  \eqref{regularised Cauchy problem},  \eqref{bound U in Hs}, and $H^{\frac{1}{2}^+}(\mathbb{R}) \hookrightarrow L^{\infty}(\mathbb{R})$ to find
\begin{align*}
	\|\partial_{t} \eta^{\delta}\|_{L^{\infty}}
	\leq c_{\beta}^{2}
	\|( \eta^{\delta} , u^{\delta} )
	\|_{V^s_{\sqrt{\mu}}} + \|\eta^{\delta}\|_{H^s}\|u^{\delta}\|_{H^s}
	%\\
	%& \leq 
	%	(c_{\beta}^2+1)
	%	\|( \eta^{\delta} , u^{\delta} )
	%	\|_{V^s_{\sqrt{\mu}}} 
	%	\\
	%	&
	\leq  
	4k^{2}_{\beta}c_{\beta}^{2} \|(\eta_0,u_0)\|_{V^s_{\sqrt{\mu}}}.
\end{align*}

Now, let  $\delta$ small such that $0<o(\delta)< \frac{h_0}{4}$. Then, by Claim $1$ and Claim $2$, we  use the Fundamental Theorem of Calculus to obtain
\begin{align} \label{Fundamental thm}
	1 + \eta^{\delta}(x,t) 
	%&
	= 
	1 + \eta_0^{\delta} + \int^t_0 \partial_t \eta^{\delta}(x,s)\:ds
	%\\ 
	%&
	\geq 
	\frac{3}{4}h_0 -k_{\beta}^2\beta^2 \tilde{T}_{\delta},
\end{align}
for all $t \in [0, \tilde{T}_{\delta}]$. On the one hand, if $\beta \geq \frac{1}{3}$, then
\begin{equation*}
	k^{2}_{\beta} = \frac{c^{2}_{\beta}}{c^{1}_{\beta}} = c\beta.
\end{equation*}
Thus, for $C_2>0$ large enough, we get that $k^1_{\beta} \geq  C_2 \beta^2 \geq  \frac{ c \beta^2 }{h_0}$. Moreover, by the assumption $\tilde{T}_{\delta} \leq T_0$, we  conclude from \eqref{Fundamental thm} that
$$1 + \eta^{\delta}(x,t) \geq \frac{3}{4}h_0 -k_{\beta}^2\beta^2 T_0
\geq\frac{h_0}{2},$$ 
for all  $t\in [0,\tilde{T}_{\delta}]$. On the other hand, in the case when $\beta \in (0, \frac{1}{3})$ we need to verify \eqref{cond. sol. beta}. But this can be done the same way by choosing $k_{\beta}^1 \geq \frac{C_1}{\beta} \geq \frac{c}{ \beta h_\beta}$  for $C_1>0$ large enough.

The hypotheses of Proposition \ref{Energy fully disp Boussinesq} are now verified, leaving us \eqref{Energy full dispersion} and \eqref{equiv. full dispersion} at our disposal. With this at hand, we observe that
\begin{equation*}
	E_s(\bold{U}^{\delta})(t)
	\leq 
	E_s(\bold{U}^{\delta})(0)
	+
	c^2_{\beta}\int_{0}^t \big{(}E_s(\bold{U}^{\delta})( s') \big{)}^{\frac{3}{2}}\: ds'
	: = \psi(t).
\end{equation*}
By the above inequality, we then have $\psi'(t) \leq c^2_{\beta} \:  \big{(}E_s(\bold{U}^{\delta})(t)\big{)}^{\frac{3}{2}} \leq c^2_{\beta}\: (\psi(t))^{\frac{3}{2}}$. We solve the differential inequality and use \eqref{equiv. full dispersion} to relate the energy with the $V^s_{\sqrt{\mu}}-$norm of the solution and deduce that
\begin{align}\label{uniform energy}
	c_{\beta}^1\| (\eta^{\delta},u^{\delta}) (t) \|_{V^s_{\sqrt{\mu}}} %= E_s(\bold{U}^{\delta})^{1/2}(t) 
	\leq \frac{c_{\beta}^2\|(\eta_0^{\delta} ,u_0^{\delta} )\|_{V^s_{\sqrt{\mu}}}}{1 -  \frac{(c_{\beta}^2)^2}{2}t \| (\eta_0^{\delta}, u_0^{\delta})\|_{V^s_{\sqrt{\mu}}}},
\end{align}
for all $t \in [0,\tilde{T}_{\delta}]$. Finally, if $C_1,C_2>0$ is large enough then since $\tilde{T}_{\delta} \leq T_0$ we have that
\begin{align*}
	\| (\eta^{\delta},u^{\delta}) (t)\|_{V^s_{\sqrt{\mu}}} 
	& \leq 2k^2_{\beta}
	\| (\eta_0^{\delta} , u_0^{\delta} ) \|_{V^s_{\sqrt{\mu}}}.
\end{align*}
% 
% 
% 
%Though, having \eqref{uniform energy}
Though, by continuity of the solution in time $t \in [0,T_{\delta}^{\star})$, there exists $\tau>0$ such that $\|(\eta^{\delta},u^{\delta}) (\tau)\|_{V^s_{\sqrt{\mu}}} \leq 3 k^2_{\beta} \|(\eta_0,u_0) \|_{V^s_{\sqrt{\mu}}}$ for $\tilde{T}_{\delta} < \tau < T_{\delta}^{\star}$. This contradicts the definition of $\tilde{T}_{\delta}$. Thus, we may conclude $T_0 <  \tilde{T}_{\delta}$ for all $\delta >0$ and that $T_0$ is independent from $\delta$ by its definition in  \eqref{time T0}.\\

\end{proof}

\begin{remark}\label{Remark on beta after lemma}
	For $0<\beta<\frac{1}{3}$ we observe that $k^1_{\beta} \sim k^2_{\beta} \sim \frac{1}{\beta}$ and is due to the appearance of $c_{\beta}^1$ in the coercivity estimate \eqref{equiv. full dispersion}. This will impact the size of the time interval when $\beta$ is small (see  Remark \ref{Remark 1 on T - beta}). On the other hand, for system  \eqref{2nd Whitham Boussinesq} the coercivity estimate \eqref{2nd Equivalence Whitham Boussinesq} is independent from $\beta$ and therefore gives a longer time of existence, as noted in  Remark \ref{Remark beta thm 1.8}.\\
\end{remark}

\noindent
\underline{Step $4$:} \textit{Uniqueness.} Given a solution of \eqref{full dispersion}, then we claim that it must be unique.

We consider two solutions 
$$\ve(\zeta_1, v_1)= (\eta_1,u_1) \: \text{and} \: \ve(\zeta_1, v_1)= (\eta_1,u_1) \: \text{in}\: C([0,T_{0}];V^s_{\sqrt{\mu}}(\R)),$$
with the same initial data. Then define $\bold{W} = (\eta_1-\eta_2, u_1-u_2)^T$, which is associated to the initial datum $\bold{W}(0) = \bold{0}$. Since $(\eta_1,u_1) \in H^s(\R)$, there exist a number $h_1>0$ such that $\|(\eta_1,u_1)\|_{H^s\times H^s}\leq h_1$. Moreover, $\eta_1$ satisfy the non-cavitation condition by the Fundamental Theorem of Calculus and the argument made in the proof of Lemma \ref{Lemma time of existence}. Thus, we may use Proposition \ref{Compactness L2} to deduce
\begin{equation*}
	\frac{d}{dt} \tilde{E}_0(\bold{W}) \lesssim_{\beta} \max \limits_{i = 1,2} \|(\eta_i,u_i)\|_{V^{s}_{\sqrt{\mu}}}  \tilde{E}_0(\bold{W}).
\end{equation*}
Then Gr{\"o}nwall's lemma and \eqref{equiv 1} implies that $\|(\eta_1-\eta_2, u_1-u_2)(t)\|_{V_{\sqrt{\mu}}^s} = 0$ for all $t\in [0,T_0]$. We therefore  conclude the proof of the uniqueness. \\

\noindent
\underline{Step 5:} \textit{Existence of solutions.} We claim that for all $0\leq s'<s$ there exists a solution $(\zeta,v)=\ve^{-1}(\eta, u) \in C([0,T_0]; V_{\sqrt{\mu}}^{s'}(\mathbb{R})) \cap L^{\infty}([0,T_0]; V_{\sqrt{\mu}}^s(\mathbb{R}))$ of \eqref{full dispersion} with $T_0 = \mathcal{O}(\frac{1}{\ve})$ defined by  \eqref{time T0}. \\

Using the change of variable  $(\zeta,v)=\ve^{-1}(\eta, u)$, we see that the claim in Step $5$ is equivalent to proving that $(\eta^{\delta}, u^{\delta})$ solving \eqref{regularised Cauchy problem} will satisfy  system \eqref{Eq 1 M} in the limit $\delta \searrow 0$ on $[0,T_0]$. In fact, the main idea  is to prove the convergence of $\{ (\eta^{\delta}, u^{\delta})\}$ as $\delta \searrow 0$ by considering the difference between two solutions
$$\bold{W}^{\delta} = (\psi^{\delta}, w^{\delta}): = (\eta^{\delta_1} - \eta^{\delta_2}, u^{\delta_1}- u^{\delta_2}).$$ 
Here we let $0<\delta_1<\delta_2 < \delta \leq 1$ and $ (\eta^{\delta_i}, u^{\delta_i})$ be the solution of \eqref{regularised Cauchy problem}, obtained in Step $1$. Moreover, let $\varphi_{\delta}(D) = (\varphi_{\delta_1}-\varphi_{\delta_2})(D)$, then $(\psi^{\delta}, w^{\delta})$ satisfies a regularized version of \eqref{lin W}:
\begin{equation*}
	\partial_t \bold{W}^{\delta}+  M(\bold{U}_1^{\delta},D)\cdot \varphi_{\delta}(D)\bold{W}^{\delta} = \bold{F}^{\delta},
\end{equation*}
with
\begin{equation}
	 \bold{F}^{\delta} = 
	 -
	 \begin{pmatrix}
	 	w^{\delta}\partial_x \varphi_{\delta_2}(D) \eta^{\delta_2} + \psi^{\delta} \partial_x \varphi_{\delta_2}(D) u^{\delta_2}
	 	\\
	 	w^{\delta} \partial_x \varphi_{\delta_2}(D) u^{\delta_2}
	 \end{pmatrix},
\end{equation}
and initial data
\begin{equation}\label{def psi 0 w0}
	(\psi^{\delta}(0), w^{\delta}(0)) = (\varphi_{\delta}(D)\eta_0,\varphi_{\delta}(D)u_0).
\end{equation}
The system also satisfies the estimates of Propositions \ref{Energy full dispersion} and \ref{Compactness L2} by the arguments used in the proof of Lemma \ref{Lemma time of existence}. \\ %With this at hand, we will now demonstrate the proof in three steps through an application of the Bona-Smith argument \cite{BonaSmith1975}.\\

\noindent
\underline{Step 5.1:} \textit{Convergence in $C([0,T_0];V^{0}_{\sqrt{\mu}}(\mathbb{R}))$.}  Define the difference  $(\psi^{\delta}, w^{\delta})$ as above,  then use \eqref{Energy 1} and \eqref{equiv 1}, combined with Gr{\"o}nwall's inequality and \eqref{Bound on solution - data}  to find the estimate
$$\sup\limits_{t\in[0,T_0]}
	\|(\psi^{\delta},w^{\delta})(t)\|_{V^{0}_{\sqrt{\mu}}} \leq e^{c_{\beta}^2 \|(\eta_0, u_0)\|_{V^s_{\sqrt{\mu}}}T_0}\|(\psi^{\delta}(0),w^{\delta}(0))\|_{V^{0}_{\sqrt{\mu}}},$$
with $  \|(\eta_0, u_0)\|_{V^s_{\sqrt{\mu}}}  T_0 \lesssim_{\beta} 1$ by defintion \eqref{time T0}. As a result, we use \eqref{def psi 0 w0}, the triangle inequality and  \eqref{reg 2}  to deduce that
\begin{align}\label{Rate in X0}
	\sup\limits_{t\in[0,T_0]}\|(\psi^{\delta},w^{\delta})(t)\|_{V^{0}_{\sqrt{\mu}}} 
	 \lesssim_{\beta}  \delta^{s} \|(\eta_0,u_0)\|_{V^s_{\sqrt{\mu}}} \underset{\delta \rightarrow 0}{\longrightarrow} 0.
\end{align}
Consequently,  $\{(\eta^{\delta}, u^{\delta})\}_{0<\delta\leq 1}$ defines a Cauchy sequence  in $ C([0,T_0]; V^{0}_{\sqrt{\mu}}(\mathbb{R}))$ and we conclude that there exists a limit $( \eta,  u) \in C([0,T_0];V^{0}_{\sqrt{\mu}}(\mathbb{R}))$  by completeness.

\ 
\\ 
\ 

\noindent
\underline{Step $5.2$:} \textit{Solution in $C([0,T_0]; V^{s'}_{\sqrt{\mu}}(\mathbb{R}))\cap L^{\infty}([0,T_0]; V^s_{\sqrt{\mu}}(\mathbb{R}))$ for  $s'\in [0,s)$.} As a direct consequence of  \eqref{Bound on solution - data} and  \eqref{Rate in X0}, we deduce by interpolation
\begin{align}\notag
	\| (\psi^{\delta},w^{\delta}) \|_{L^{\infty}_{T_0} V^{s'}_{\sqrt{\mu}}} 
	& 
	\lesssim_{\beta} 
	\| (\psi^{\delta},w^{\delta}) \|_{L^{\infty}_{T_0} V^{s}_{\sqrt{\mu}}}^{\frac{s'}{s}}
	\| (\psi^{\delta},w^{\delta}) \|_{L^{\infty}_{T_0} V^{0}_{\sqrt{\mu}}}^{1-\frac{s'}{s}}
	\\
	&\label{interpolation s'}
	\lesssim_{\beta}	\delta^{s-s'} \|(\eta_0,u_0)\|_{V^s_{\sqrt{\mu}}}
  \underset{\delta \rightarrow 0}{\longrightarrow} 0.
\end{align}

\noindent
\underline{Step 6:} \textit{Presistence of the solution.} We claim that there exists a unique solution $(\zeta,v)=\ve^{-1}(\eta, u) \in  C([0,T_0]; V^{s}_{\sqrt{\mu}}(\mathbb{R}))$ of \eqref{full dispersion}.

We again use the scaled variables $(\eta^{\delta}, u^{\delta})$, where we will apply estimate \eqref{Energy 2}, following the Bona-Smith argument \cite{BonaSmith1975}. But  first we must  control $\bold{F}^{\delta}$ in $V_{\sqrt{\mu}}^s(\mathbb{R})$. The elements of $\bold{F}^{\delta}$ are given in \eqref{F: source term}, and we must therefore control the terms given by:
\begin{align*}
	\big(
	J^s \bold{F}^{\delta}, Q(\bold{U}^{\delta_1}, D) J^s \bold{W}^{\delta} 
	\big)_{L^2} 
	\\
	&\hspace{-3.5cm} = 
	\big(
	J^s (w^{\delta} \partial_x \varphi_{\delta_2}(D)\eta^{\delta_2}), J^s \psi^{\delta}
	\big)_{L^2} 
	+
	\big(
	J^s (\psi^{\delta} \partial_x\varphi_{\delta_2}(D) \eta^{\delta_2}), J^s \psi^{\delta}
	\big)_{L^2} \\
	& \hspace{-3.5cm} 
	\hspace{.5cm}+
	\big(
	J^s (w^{\delta} \partial_x\varphi_{\delta_2}(D) u^{\delta_2}), \eta^{\delta_1} J^s w^{\delta}
	\big)_{L^2} 
	+
	\big(
	J^s (w^{\delta} \partial_x \varphi_{\delta_2}(D) u^{\delta_2}), \mathcal{K}_{\mu}(D)J^s w^{\delta}
	\big)_{L^2} 
	\\
	& \hspace{-3.5cm} = :
	A_1 + A_2 + A_3 + A_4.
\end{align*}
The terms $A_1, A_2$ and $A_3$ are treated similarly. For instance, take $A_1$. Then we observe that
\begin{align*}
	A_1 \leq \|J^s(w^{\delta} \partial_x\varphi_{\delta_2}(D) \eta^{\delta_2}) \|_{L^2} \| J^s\psi^{\delta}\|_{L^2}.
\end{align*}
Furthermore, using \eqref{Prod. Kato-Ponce} and the Sobolev embedding  we obtain that
\begin{align}\label{diverging -1}
	\| J^s(w^{\delta} \partial_{x}\varphi_{\delta_2}(D)  \eta^{\delta_2}) \|_{L^2} 
	& \lesssim
	\|w^{\delta}\|_{L^{\infty}} \| J^s\partial_x \varphi_{\delta}(D)\eta^{\delta_2}\|_{L^2} + \|J^sw^{\delta}\|_{L^2} \| \eta^{\delta_2}\|_{H^s} 
	.
\end{align}
Using the triangle inequality, \eqref{reg 3}, and \eqref{Bound on solution - data}   we observe that 
\begin{equation}\label{one of the last 1}
	\| J^s \partial_x \varphi_{\delta}(D) \eta^{\delta_2} \|_{L^{2}}  \leq  \delta^{-1} \| \eta^{\delta_2} \|_{H^{s}},
\end{equation} 
which needs to be compensated to close the estimate. With this in mind, we use the Sobolev embedding $H^{\frac{1}{2}^+}(\R) \hookrightarrow L^{\infty}(\R)$ and \eqref{interpolation s'} to deduce
\begin{align}\label{one of the last 2}
	\| w^{\delta} \|_{L^{\infty}_{T_0}L^{\infty}}
	\lesssim
	\|(\psi^{\delta}, w^{\delta} )\|_{L^{\infty}_{T_0}V_{\sqrt{\mu}}^{\frac{1}{2}^+}}
	\lesssim
	 \delta^{s-\frac{1}{2}^+}
	  \|(\eta_0,u_0)\|_{V^s_{\sqrt{\mu}}}.
\end{align}
Thus, combining \eqref{diverging -1} with \eqref{one of the last 1} and \eqref{one of the last 2} we get that
\begin{align*}
    |A_1| \lesssim \sup\limits_{t\in[0,T_0]} \big{(} \|w^{\delta}\|_{H^s} \|  \eta^{\delta_2}\|_{H^s}  + \delta^{s-\frac{3}{2}^+}\|(\eta_0,u_0)\|_{V^s_{\sqrt{\mu}}} \big{)} \|\psi^{\delta}\|_{H^s},
\end{align*}
as $\delta \searrow 0$. Arguing similarly, and using estimate \eqref{Bound on solution - data}, we deduce that
\begin{align*}
	|A_1|  + |A_2| + |A_3| 
%	& \lesssim
	% \Big(\|(\eta^{\delta_2}, u^{\delta_2})\|_{V^s_{\sqrt{\mu}}} \|(\psi^{\delta}, w^{\delta})\|_{V^s_{\sqrt{\mu}}}+ \|(\psi^{\delta}, w^{\delta})\|_{V^s_{\sqrt{\mu}}} + \ve o(\delta^{s-\frac{3}{2}^+})\Big)\| (\psi^{\delta}, w^{\delta})\|_{V^s_{\sqrt{\mu}}}
	 %\\
	 %&
	  \lesssim\sup\limits_{t\in[0,T_0]}   \|(\eta_0,u_0)\|_{V^s_{\sqrt{\mu}}}\big{(}
	  \|(\psi^{\delta}, w^{\delta})\|_{V^s_{\sqrt{\mu}}}^2+ \delta^{s-\frac{3}{2}^+}\| (\psi^{\delta}, w^{\delta})\|_{V^s_{\sqrt{\mu}}}\big{)}.
\end{align*}
For $A_4$, we write
\begin{align*}
	A_4
	& = 
	\big(
	[J^s \sqrt{\mathcal{K}_{\mu}}(D) , w^{\delta}] \partial_x u^{\delta_2}, J^s \sqrt{\mathcal{K}_{\mu}}(D) w^{\delta} 
	\big)_{L^2}
	\\
	& \hspace{0.4cm} 
	+
	\big( w^{\delta}
	J^s \sqrt{\mathcal{K}_{\mu}}(D) \partial_x u^{\delta_2}, J^s \sqrt{\mathcal{K}_{\mu}}(D)w^{\delta} 
	\big)_{L^2}.
\end{align*}
The commutator is treated by \eqref{Commutator low freq}. While in the second term, we use \eqref{estimate K in Hs} and argue as  above, giving the estimate
\begin{align*}
	|A_4|  \lesssim_{\beta} \sup\limits_{t\in[0,T_0]}  \|(\eta_0,u_0)\|_{V^s_{\sqrt{\mu}}} \big{(}\|(\psi^{\delta}, w^{\delta}) \|_{V^s_{\sqrt{\mu}}}^2 + \delta^{s-\frac{3}{2}^+}  \|(\psi^{\delta}, w^{\delta}) \|_{V^s_{\sqrt{\mu}}}\big{)}.
\end{align*}
We may therefore conclude by \eqref{Energy 2}:
\begin{align*}
	\frac{d}{dt} \tilde E_s(\bold{W}^{\delta})
	& \lesssim_{\beta}
	 \|(\eta_0,u_0)\|_{V^s_{\sqrt{\mu}}} 
	 \big{(}
	\tilde E_s(\bold{W}^{\delta}) + \delta^{s-\frac{3}{2}^+} \tilde E_s(\bold{W}^{\delta})^{\frac{1}{2}}
	\big{)}.
\end{align*}
Then Gr{\"o}nwall's inequality and   \eqref{equiv 2} implies
$$\|(\psi^{\delta},w^{\delta})\|_{L^{\infty}_{T_0} V^s_{\sqrt{\mu}}}  \lesssim_{\beta} \delta^{s-\frac{3}{2}^+} \|(\eta_0,u_0)\|_{V^s_{\sqrt{\mu}}}  
\underset{\delta \rightarrow 0}{\longrightarrow} 0.$$ 
Thus, $(\eta^{\delta} , u^{\delta})$ is a Cauchy sequence in $C([0,T_0]; V^s_{\sqrt{\mu}}(\mathbb{R}))$ and we conclude by the uniqueness of the limit that the solution $(\eta, u)\in  C([0,T_0]; V^s_{\sqrt{\mu}}(\mathbb{R}))$ . \\ %to the uniqueness proved in Step $1$. \\ \ \\

\noindent
\underline{Step 7:} \textit{The solution is bounded by the initial data.} We claim that the solution obtained in Step $6$ satisfy \eqref{bound in terms of initial data}. 

Indeed, using the notation from previous step, we deduce by \eqref{Bound on solution - data} that 
$$\{(\eta^{\delta},u^{\delta})\}_{0<\delta\leq 1} \subset C([0,T_0];V^s_{\sqrt{\mu}}(\mathbb{R}))$$
is a bounded sequence in a reflexive Banach space. As a result, we have by Eberlein-\u{S}umulian's Theorem that $(\eta^{\delta},u^{\delta}) \underset{\delta \rightarrow 0}{\rightharpoonup}  (\eta,u) $ weakly in $V^s_{\sqrt{\mu}}(\mathbb{R})$ for all $t \in [0,T_0]$ and implies 
\begin{equation*}%\label{Bound sol - initial data}
	\sup\limits_{t \in [0,T_0]} \|(\eta,u)\|_{V^s_{\sqrt{\mu}}} \lesssim_{\beta} \| (\eta_0,u_0) \|_{V^s_{\sqrt{\mu}}}.
\end{equation*}
%
%
%
%Now, we recall that $(\eta,u) = (\ve\zeta, \ve v)$. As a consequence, we deduce that $(\zeta,v)$  is a solution of \eqref{full dispersion} on a time of interval $[0,T/\ve]$ as defined in \eqref{Time T}.\\

\noindent
\underline{Step 8:} \textit{Continuous dependence of the flow map data solution.} Consider two sets of initial data $(\zeta_1,v_1)(0),(\zeta_2,v_2)(0) \in V^s_{\sqrt{\mu}}(\R)$.  Then we claim that for all $\lambda >0$, there exists $\gamma>0$  such that having
$$\| (\zeta_1 - \zeta_2, v_1 - v_2)(0) \|_{V^s_{\sqrt{\mu}}} <   \gamma,$$
implies 
$$\| (\zeta_1- \zeta_2, v_1 - v_2) \|_{L^{\infty}_{\frac{T_0}{2}} V^s_{\sqrt{\mu}}} <  \lambda.$$

Equivalently, we will prove that  for $\ve(\zeta_1,\zeta_2,v_1,v_2) =(\eta_1,\eta_2,u_1,u_2)$ such that
$$\| (\eta_1 - \eta_2, u_1 - u_2)(0) \|_{V^s_{\sqrt{\mu}}} <  \ve \gamma,$$
implies 
$$\| (\eta_1- \eta_2, u_1 - u_2) \|_{L^{\infty}_{\frac{T_0}{2}} V^s_{\sqrt{\mu}}} < \ve \lambda.$$
Using the notation in Step $6$, we let  $0<\delta<1$ to be fixed, and $(\eta_1^{\delta},u_1^{\delta}),(\eta_2^{\delta},u_2^{\delta})  \in C([0,\frac{T_0}{2}]; V^s_{\sqrt{\mu}}(\mathbb{R}))$  be two solutions of \eqref{regularised Cauchy problem}
on large time  with corresponding initial data $(\varphi_{\delta}(D)\eta_1 ,\varphi_{\delta}(D)u_1)(0)$ and $(\varphi_{\delta}(D)\eta_2,\varphi_{\delta}(D)u_2)(0)$. Then observe
\begin{align}\label{Continuity of the flow}\notag
	\| (\eta_1- \eta_2, u_1 - u_2) \|_{V^s_{\sqrt{\mu}}}
	& \leq 
%	\sup\limits_{t\in [0,\frac{T_0}{2}]} 
%	\Big(
	\|(\eta_1- \eta_1^{\delta}, u_1 - u_1^{\delta}) \|_{V_{\sqrt{\mu}}^s} 
%	\\
%	&    
%	\hspace{0.4cm} \notag
	+
	\| (\eta_2- \eta_2^{\delta}, u_2 - u_2^{\delta})   \|_{V_{\sqrt{\mu}}^s}
	\\
	&
	\hspace{0.4cm} \notag
	+
	\| (\eta_1^{\delta}-\eta_2^{\delta}, u_1^{\delta} - u_2^{\delta}) \|_{V^s_{\sqrt{\mu}}}
	%\Big)
	\\
	& =:
	B_1 + B_2 + B_3.
\end{align}
For the first two terms we use that $\ve^{-1} (\eta^{\delta},u^{\delta}) = (\zeta^{\delta},v^{\delta}) \rightarrow   (\zeta,v) = \ve^{-1} (\eta,u)$ as $\delta \searrow 0$ by Step $6$. Therefore it follows that $B_1$ and $B_2$ must at least satisfy the estimate,
\begin{equation}\label{B1 B2}
	\sup\limits_{t \in [0 , \frac{T_0}{2}]}( B_1 + B_2)(t) \lesssim_{\beta}
	\ve o_{\delta}(1).%\delta^{s-\frac{3}{2}^+}
	%\max\limits_{i=1,2} \|(\eta_i,u_i)(0)\|_{V^s_{\sqrt{\mu}}} .
\end{equation}
While for $B_3$, we need the continuity of the flow map of the regularized system \eqref{regularised Cauchy problem} on  a long time (see Remark \ref{long time flow map}).

We let $\bold{\tilde W}^{\delta} = (\tilde \psi^{\delta} ,\tilde w^{\delta}) = ( \eta_1^{\delta}-\eta_2^{\delta}, u_1^{\delta}- u_2^{\delta})$. Then staying consistent with previous notation, we have that the difference of two regularized solutions will satisfy the equation:
\begin{equation}\label{reg W}
	\partial_t \bold{\tilde W}^{\delta} + M(\bold{U}_1^{\delta},D)\varphi_{\delta}(D)\bold{\tilde W}^{\delta} = \bold{\tilde F}^{\delta},
\end{equation}
with 
\begin{equation*}
	\bold{\tilde F}^{\delta} =
	-
	\begin{pmatrix}
		\tilde w^{\delta} \partial_x \varphi_{\delta}(D) \eta_2^{\delta} +\tilde  \psi^{\delta} \partial_x \varphi_{\delta}(D)u_2^{\delta}\\
		\tilde w^{\delta} \partial_x \varphi_{\delta}(D) u_2^{\delta}
	\end{pmatrix},
\end{equation*}
and initial data 
$$(\tilde \psi^{\delta}, \tilde w^{\delta})(0)=   (\varphi_{\delta}(D)\eta_1 -\varphi_{\delta}(D)\eta_2,\varphi_{\delta}(D)u_1-\varphi_{\delta}(D)u_2)(0).$$
We will use this information to estimate $B_3$ by suitable energy estimates at $V^0_{\sqrt{\mu}}(\R)$ and $V^s_{\sqrt{\mu}}(\R)-$level.

Similar to Step $5.1$, we first obtain the estimate in $V^0_{\sqrt{\mu}}(\R)$ by using \eqref{equiv 1} and \eqref{equiv 2}. Indeed, there holds
\begin{equation}\label{last energy 0}
	\frac{d}{dt}\tilde{E}_0(\bold{\tilde{W}}^{{\delta}}) \lesssim_{\beta}\max\limits_{i=1,2}
	\|(\eta_i^{\delta},u_{i}^{\delta})\|_{V^s_{\sqrt{\mu}}} \tilde{E}_0(\bold{\tilde{W}}^{{\delta}}).
\end{equation}
For simplicity we let $\|(\eta_1,u_1)(0)\|_{V^s_{\sqrt{\mu}}} = \ve K$. Moreover, we observe that if $ \ve \gamma< \frac{1}{2}\|(\eta_1, u_1)(0)\|_{V^s_{\sqrt{\mu}}}$, then we have by \eqref{Bound on solution - data} that
\begin{align}\notag
	\| (\eta_1^{\delta}, u_1^{\delta}) \|_{L^{\infty}_{\frac{T_0}{2}}V^s_{\sqrt{\mu}}} +  \|  (\eta_2^{\delta}, u_2^{\delta}) | |_{L^{\infty}_{\frac{T_0}{2}}V^s_{\sqrt{\mu}}} 
	& \lesssim_{\beta}
	\|  (\eta_1, u_1)(0)  \|_{V^s_{\sqrt{\mu}}} + \|  (\eta_2, u_2)(0)  \|_{V^s_{\sqrt{\mu}}}
	\\ 
	& \lesssim_{\beta}  \ve K. \label{diff w0}%\|  (\eta_1, u_1)(0)  \|_{V^s_{\sqrt{\mu}}}.
\end{align}
As a result, we have an estimate on the difference in $V^0_{\sqrt{\mu}}(\R)$. Indeed, by Gr{\"o}nwall's inequality, \eqref{last energy 0}, \eqref{diff w0}, the triangle inequality, and \eqref{reg 2} implies 
\begin{align}\label{V0 bound}
 \| (\tilde \psi^{\delta},\tilde  w^{\delta}) \|_{V^0_{\sqrt{\mu}}} 
%	& 
	\lesssim_{\beta}  \| (\tilde \psi^{\delta}, \tilde w^{\delta})(0) \|_{V^0_{\sqrt{\mu}}} 
%	\\
%	& \lesssim_{\beta} \notag
%	\|  (\varphi_{\delta_1}(D)\eta_1- \eta_1, \varphi_{\delta_1}(D) u_1 - u_1)(0)\|_{V^1_{\sqrt{\mu}}} 
%	\\
%	&\notag
%	\hspace{0.5cm}
%	+
%	\| (\eta_2- \varphi_{\delta_2}(D)\eta_2, u_2 - \varphi_{\delta_2}(D)u_2)(0)\|_{V^1_{\sqrt{\mu}}} 
%	+ \gamma_{\ve}
%	\\
	%&
	 \lesssim_{\beta} 
	  \ve K(\delta^{s}+ K^{-1}\gamma).
	    %o(\delta^{s}) \|  (\eta_1, u_1)(0)  \|_{V^s_{\sqrt{\mu}}} + \gamma_{\ve}.
\end{align}

We will now use this decay estimate to deal with  \eqref{Energy 2}, which is at the $V^s_{\sqrt{\mu}}(\R)-$level. Similar to Step $6$, we decompose  the source term \eqref{F: source term} in four pieces
\begin{align*}
	\tilde{A}	
	&
	 : = 
	 \big(J^s \bold{\tilde F}^{\delta}, Q(\bold{U}_1^{\delta},D) J^s \bold{\tilde W}^{\delta}\big)_{L^2}
	 \\
	 & = 
	 \big(
	 J^s (\tilde{w}^{\delta} \partial_x \varphi_{\delta}(D)\eta_2^{\delta}), J^s \tilde{\psi}^{\delta}
	 \big)_{L^2} 
	 +
	 \big(
	 J^s (\tilde{\psi}^{\delta} \partial_x\varphi_{\delta}(D) \eta_2^{\delta}), J^s \tilde{\psi}^{\delta}
	 \big)_{L^2} \\
	 &  
	 \hspace{.5cm}+
	 \big(
	 J^s (\tilde{w}^{\delta} \partial_x\varphi_{\delta}(D) u_2^{\delta}), \eta_2^{\delta} J^s \tilde{w}^{\delta}
	 \big)_{L^2} 
	 +
	 \big(
	 J^s (\tilde{w}^{\delta} \partial_x \varphi_{\delta}(D) u_2^{\delta}), \mathcal{K}_{\mu}(D)J^s \tilde{w}^{\delta}
	 \big)_{L^2} 
	 \\
	 & = :
	 \tilde{A}_1 + \tilde{A}_2 + \tilde{A}_3 + \tilde{A}_4.
\end{align*}
To estimate $\tilde{A}_1$, we  first obtain a bound similar to  \eqref{interpolation s'}. Indeed, using the Sobolev emebedding, interpolation, \eqref{diff w0}, and \eqref{V0 bound} yields
\begin{align*}
\sup\limits_{t\in[0,T_0]}	\| \tilde{w}^{\delta} \|_{L^{\infty}}
	&
 \lesssim\sup\limits_{t\in[0,T_0]}
	\|(\tilde{\psi}^{\delta}, \tilde{w}^{\delta} )\|_{V_{\sqrt{\mu}}^{\frac{1}{2}^+}}
%	\\
%	&\lesssim
%	\|(\eta_1,u_1)(0)\|_{V^s_{\sqrt{\mu}}}^{\frac{s'}{s}} (o(\delta^{s})\|(\eta_1,u_1)(0)\|_{V^s_{\sqrt{\mu}}}+ \gamma_{\ve})^{1-\frac{s'}{s}}
	\\
	&
	\lesssim 
	\| (\tilde \psi^{\delta},\tilde w^{\delta}) \|_{L^{\infty}_{T_0} V^{s}_{\sqrt{\mu}}}^{\frac{1}{2s}^+}
	\| (\tilde \psi^{\delta},\tilde w^{\delta}) \|_{L^{\infty}_{T_0} V^{0}_{\sqrt{\mu}}}^{1-\frac{1}{2s}^+}
	\\
	&
	\lesssim 
	 \ve K(\delta^{s-\frac{1}{2}^+}+ (K^{-1}\gamma)^{1-\frac{1}{2s}^+})
	%\|(\eta_1,u_1)(0)\|_{V^s_{\sqrt{\mu}}}o(\delta^{s-\frac{1}{2}^+})+ \|(\eta_1,u_1)(0)\|_{V^s_{\sqrt{\mu}}}^{\frac{s'}{s}}\gamma_{\ve}^{1-\frac{s'}{s}}
	,
\end{align*}
where $1-\frac{1}{2s}^+>0$ for $s>\frac{1}{2}^+$. Then arguing as we did for $A_1$ in Step $6$, we obtain that
\begin{align*}
	|\tilde{A}_1| 
	& 
	\lesssim \sup\limits_{t\in[0,T_0]}
	\big{(} \|\tilde{w}^{\delta}\|_{H^s} \|  \eta_2^{\delta}\|_{H^s}  + 	\| \tilde{w}^{\delta} \|_{L^{\infty}}\delta^{-1}\|\eta^{\delta}_2\|_{H^s}
	\big{)}
	 \|\tilde{\psi}^{\delta}\|_{H^s}
	 \\
	 & 
	 \lesssim
	 \ve K
	 \sup\limits_{t\in[0,T_0]}
	 \big{(} \|\tilde{w}^{\delta}\|_{H^s}  + \delta^{s-\frac{3}{2}^+} + \delta^{-1}(K^{-1}\gamma)^{1-\frac{1}{2s}^+}
	 \big{)}
	 \|\tilde{\psi}^{\delta}\|_{H^s}.
\end{align*}
Moreover, for the remaining terms, we can use similar estimates, recalling that for $\tilde{A}_4$ we also need to deal with the non-local operator $\mathcal{K}_{\mu}(D)$ (see step $6$ for details). Indeed,
\begin{align}\label{A tilde}
	\tilde{A}	&\lesssim_{\beta}
	\ve K
	\sup\limits_{t\in[0,T_0]}
	\big{(}
	 	\|(\tilde \psi^{\delta}, \tilde w^{\delta}) \|_{V_{\sqrt{\mu}}^s}
	 +
	  \delta^{s-\frac{3}{2}^+} + \delta^{-1}(K^{-1}\gamma)^{1-\frac{1}{2s}^+}
	\big{)}
	\|(\tilde \psi^{\delta}, \tilde w^{\delta}) \|_{V_{\sqrt{\mu}}^s}.
\end{align}
Consequently, combining estimates \eqref{Energy 2} and \eqref{equiv 2} with \eqref{A tilde} yields
\begin{align*}
	\frac{d}{dt} \tilde E_s(\tilde{\bold{W}}^{\delta}) 
	 \lesssim_{\beta}
	 \ve K
	 \big{(}
	\tilde E_s(\tilde{\bold{W}}^{\delta})  +( \delta^{s-\frac{3}{2}^+} + \delta^{-1}(K^{-1}\gamma)^{1-\frac{1}{2s}^+}) \tilde E_s(\tilde{\bold{W}}^{\delta}) ^{\frac{1}{2}}
	\big{)}.
\end{align*}
% 
% 
% \
Thus, we  have an estimate on $B_3$ by  the energy estimate  \eqref{equiv 2}, Gr{\"o}nwall's inequality and \eqref{reg 4}. Indeed,  %since 
%
%
%
%\begin{equation*}
	%	\|(\tilde \psi^{\delta}, \tilde w^{\delta})(0) \|_{V_{\sqrt{\mu}}^s} = 
%\end{equation*}
%
%
%
 there holds
\begin{align}\notag
	B_3 = \|(\tilde \psi^{\delta}, \tilde w^{\delta}) \|_{V_{\sqrt{\mu}}^s} 
	& \lesssim_{\beta}
	\|(\tilde \psi^{\delta}, \tilde w^{\delta})(0) \|_{V_{\sqrt{\mu}}^s} + \ve K  ( \delta^{s-\frac{3}{2}^+} + \delta^{-1}(K^{-1}\gamma)^{1-\frac{1}{2s}^+})  \notag
	\\ 
	%\notag
	%& \lesssim \ve \big(
	%\underset{ o(1)_{\delta_1}}{\underbrace{| \bold{V}_0 - \chi_{\delta_1}\bold{V}_0 |_{V^s_{\sqrt{\mu}}} }}
	%+
	%\underset{ o(1)_{\delta_2}}{\underbrace{|\bold{Z}_0 - \chi_{\delta_2} \bold{Z}_0|_{V^s_{\sqrt{\mu}}}}}
	%+
	%\underset{\leq R}{\underbrace{|\bold{V}_0 - \bold{Z}_0 |_{V^s_{\sqrt{\mu}}}}} + o(\delta) +  \mathcal{O}(\delta^{-1})R^{\frac{4}{5}^+}
	% \\ 
	& \lesssim_{\beta} \label{B3} 
	\ve o_{\delta}(1)+\ve\gamma
	+
	 \ve K  ( \delta^{s-\frac{3}{2}^+} + \delta^{-1}(K^{-1}\gamma)^{1-\frac{1}{2s}^+}).
\end{align}

Returning to \eqref{Continuity of the flow}, we may conclude the proof of the continuous dependence. We first fix $0<\delta<1$ to be small enough and satisfying
\begin{equation*}
	o_{\delta}(1)+K\delta^{s-\frac{3}{2}^+} < \frac{\lambda}{2c_{\beta}},
\end{equation*}
for some constant $c_{\beta}$ depending on $\beta$. Then let $ \gamma$ verify the restriction:
\begin{equation*}
	\ve \gamma < \frac{1}{2}\|(\eta_1, u_1)(0)\|_{V^s_{\sqrt{\mu}}},
\end{equation*}
such that $\gamma  + K\delta^{-1}(K^{-1}\gamma)^{1-\frac{1}{2s}^+} <\frac{\lambda}{2c_{\beta}}$. Consequently, we have by  \eqref{Continuity of the flow}, \eqref{B1 B2} and \eqref{B3}  that
\begin{align*}
	\sup\limits_{t \in [0,\frac{T_0}{2}]} \| (\eta_1-\eta_2, u_1 - u_2)(t)\|_{V^s_{\sqrt{\mu}}} 
	&
	\leq	\ve c_{\beta} (o_{\delta}(1)+ \gamma
	+
	K  ( \delta^{s-\frac{3}{2}^+} + \delta^{-1}(K^{-1}\gamma)^{1-\frac{1}{2s}^+}) )
	\\
	&  < \ve \lambda.
\end{align*}
As a result, we have demonstrated that the solution of \eqref{full dispersion} depends continuously on the initial data and thus completes the proof of Theorem \ref{Well-posedness long time full dispersion}.

\end{proof}

\section{Final comments}
\noindent
The proof of Theorem \ref{W-P Whitham Boussinesq} and Theorem \ref{W-P 2nd Whitham Boussinesq} are similar and is therefore omitted. %Although,  one should note that both \eqref{Whitham Boussinesq} and \eqref{2nd Whitham Boussinesq} are well-posed on a long time independent from the surface tension parameter (see  Remark \ref{Remark on beta after lemma}). 
%Moreover, we have the full  justification of each system as an asymptotic model to the water wave equations as a consequence of \cite{Emerald2021}. 

%--delete--In summary, we have considered a capillary Whitham-Boussinesq type system and %proved that it is well-posed on large time. 

%Add the proof of ... is presented in details while --- follows the same arguments. The main difficulty is finding suitble energy estimates....

%The system is fully dispersive  where  the proof of Theorem \ref{Well-posedness long time full dispersion} relied on a 'nonlinear technique' to verify the existence of solutions on time $\frac{1}{\ve}$ and the continuity of the flow map with respect to initial data. Consequently, we have the full justification of \eqref{full dispersion} as an asymptotic model to the water wave equations. 

\appendix

\section{}

\subsection{Pointwise estimates for $\sqrt{K_{\mu}(\xi)}$ and $\sqrt{T_{\mu}(\xi)}$}\label{A1}
Before turning to the proof of the pointwise estimates in Lemma \ref{Pointwise est.} and Lemma \ref{Pointwise est. on T}, we make an important observation. Let $\sqrt{\mathcal{T}_{\mu}}(D)$ be the Fourier multiplier associated with the symbol
\begin{equation*}
	\sqrt{T_{\mu}(\xi)} = \sqrt{\frac{tanh(\sqrt{\mu} |\xi|)}{\sqrt{\mu} |\xi|}}.
\end{equation*}
Then the operator  is regularizing for $\mu>0$ and acts similar to the Bessel potential  $J^{-\frac{1}{2}}$ in the $L^2-$norm. While $\sqrt{K_{\mu}(\xi)}$ has a  similar behaviour in low frequency for $\beta<\frac{1}{3}$, but acts like $J^{\frac{1}{2}}$ in high frequencies.\\
\begin{figure}[h!]
	\hspace{-0.9cm}
	\includegraphics[scale=0.4]{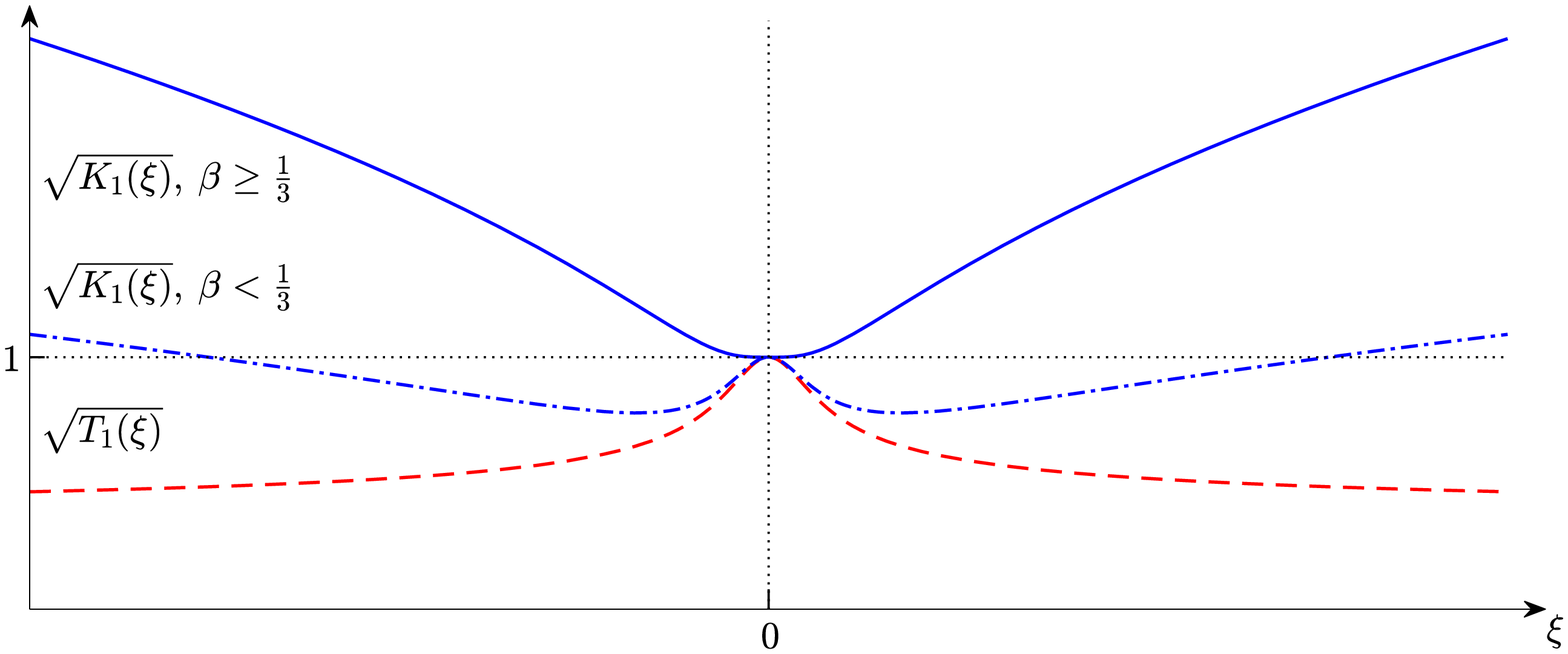}
	\centering
	\caption{\small The multiplier $\sqrt{K_1(\xi)}$ in the cases $\beta \geq \frac{1}{3}$ (line) and $\beta < \frac{1}{3}$ (dash-dots). The red curve is a plot of $\sqrt{T_1(\xi)}$ (dash).}
\end{figure}
\begin{lemma}\label{der T}
	Let $\mu >0$ and take any $n \in \N$.  
	\begin{itemize}
		\item  Then $T_{\mu}(\xi)$ satisfies
		\begin{equation}\label{derivative T}
			\Big{|} \frac{d^n}{d\xi^n}\sqrt{T_{\mu}(\xi)}\Big{|} \lesssim  \mu^{\frac{n}{2}} \langle \sqrt{\mu} \xi \rangle^{-\frac{1}{2}-n}.
		\end{equation}
		\item Similarly, $K_{\mu}(\xi)$ satisfies
		\begin{equation}\label{derivative sqrt K}
			\Big{|}    \frac{d^n}{d\xi^n}\sqrt{K_{\mu}(\xi)}\Big{|} \lesssim_{\beta}  \mu^{\frac{n}{2}} \langle \sqrt{\mu} \xi \rangle^{\frac{1}{2}-n}.
		\end{equation}
	\end{itemize}
\end{lemma}

\begin{proof} The proof is a generalization of Lemma $8$ in \cite{Achenef2020well}. Following their arguments, we observe that since $T_{1} (\sqrt{\mu} r) =  T_{\mu}(r)$ for $r>0$, it is sufficient to show that
\begin{equation*}
	\Big{|}   \frac{d^n}{dr^n} \sqrt{T_{1}(r)} \Big{|} \lesssim \langle  r \rangle^{-\frac{1}{2}-n}.
\end{equation*}
We divide the proof into two steps.

 First, let $0<r<\frac{1}{2}$ and  prove that any derivative of $\sqrt{T_1(r)}$ is bounded. We have that $\sqrt{T_1(r)}$ is bounded in this region. By direct computation, we observe
\begin{align}\label{F' <0.5}
	\frac{d}{dr} \sqrt{T_{1}(r)}
	=
	\frac{sech^2(r)}{(2r)^2\sqrt{T_{1}(r)}}\bigg( 2r +  \frac{e^{-2r} -e^{2r} }{2}\bigg)
	=:  \frac{sech^2(r)}{\sqrt{T_{1}(r)}} G(r),
\end{align}
where $G(r)$ can be written as a series by expanding the exponentials. Indeed, we have that
\begin{align*}
	G(r)  = 
	\frac{1}{(2r)^2}\bigg(
	2r + \frac{e^{-2r} - e^{2r}}{2}
	\bigg)
	=-\sum \limits_{k = 0}^{\infty} \frac{(2r)^{2k+1}}{(2k+3)!}.
\end{align*}
The series is uniformly convergent for $r\geq 0$. Moreover, $G(r)$ and its derivatives are bounded for $0< r < \frac{1}{2}$. By extension, since for all $n\geq 0$ there holds $\frac{d^n}{dr^n}sech^2(r) \lesssim e^{-2r}$ we have that
\begin{equation*}
	\Big{|}   \frac{d^n}{dr^n}   \sqrt{T_{1}(r)} \Big{|} \lesssim 1.
\end{equation*}

Now, we let $r \geq \frac{1}{2}$ and prove the necessary decay estimate. We use the identity
 \begin{equation}\label{id th}
 	tanh(r) = 1- \frac{2}{e^{2r} + 1},
 \end{equation}
and  deduce by the chain rule that
\begin{align*}
	\Big{|}   \frac{d^n}{dr^n} \sqrt{T_{1}(r)} \Big{|} \lesssim
	\sum_{k=0}^n  \Big{(}\frac{1}{r}\big(1 - \frac{2}{e^{2r} + 1}\big)\Big{)}^{\frac{1}{2} -k} r^{-k-n}
	\lesssim \langle r \rangle^{-\frac{1}{2} - n}.
\end{align*}
%
%
%
%by the chain rule.

Lastly, we have that \eqref{derivative sqrt K} follows by the Leibniz rule. Indeed, we observe
\begin{align*}
		\Big{|}   \frac{d^n}{dr^n} \sqrt{K_{1}(r)} \Big{|} 
		&=
		\Big{|}   \frac{d^n}{dr^n} \sqrt{T_{1}(r)(1+ \beta r^2)} \Big{|} 
		\\
		&
		\lesssim
		\sum\limits_{k=0}^n
		 \Big{|}   \frac{d^{n-k}}{dr^{n-k}} \sqrt{T_{1}(r)} \Big{|} 
		 \Big{|}   \frac{d^k}{dr^k}
		 \sqrt{1 + \beta r^2}
		 \Big{|}
		 \\
		 & 
		 \lesssim_{\beta}
		 \langle r \rangle^{\frac{1}{2} -n},
\end{align*}
which concludes the proof of Lemma \ref{der T}.
\end{proof}

\subsection{Proof of Lemmas \ref{Pointwise est.} and \ref{Pointwise est. on T}} \label{A2}
\begin{proof}[Proof of Lemma \ref{Pointwise est.}]
	First, we again make the observation that $K_1(\sqrt{\mu} \xi)  = K_{\mu}(\xi)$. Therefore, we simply let $r>0$ and consider $K_1(r)$. To establish the upper bound given in \eqref{sqrt K}, we note that for $r<1$ we have
	\begin{align*}
		K_1(r) \leq  1 + \beta.
	\end{align*}
	This is because $\frac{tanh(r)}{r} \leq 1$. On the other hand, when $r\geq 1$ then $tanh(r) <1$  and it follows that  
	\begin{align*}
		K_1(r) \leq  1+ \beta r.
	\end{align*}
	Consequently, for all $r>0$ there holds $K_1(r) \leq c_{\beta}^2( 1 + r)$ with  $c_{\beta}^2$ as defined in \eqref{Constant beta}.% \eqref{sqrt K}.
	 \\

	Next, we prove the lower bound given by \eqref{K pos} with $\beta \geq \frac{1}{3}$. We will again split $r$ into two intervals, where we aim to prove
	\begin{equation}\label{want to show}
		K_1(r) =  \frac{tanh(r)}{r} (1 + \beta r^2) \geq (1 - \frac{h_0}{2})+ c  r,
	\end{equation}
	for some positive constant $c>0$ and any $h_{0}\in(0,1)$.  We prove \eqref{want to show} by considering two cases for $r$. When $0\leq r \leq \frac{h_0}{4}$ we use that
	\begin{equation}\label{Taylor tanh}
		tanh(r) = \int_0^r (1-tanh^2(x)) \: dx
		\geq
		r - \frac{r^3}{3},
	\end{equation}
	since $tanh^2(x)\leq x^2$ by the mean value theorem. Therefore, we have that
	\begin{align*}
			K_1(r)  \geq \Big{(}	1 - \frac{r^2}{3} \Big{)} \Big{(}	1 + \frac{r^2}{3} \Big{)} \geq \Big{(} 1 - \frac{h_0}{2} \Big{)} +  \Big{(} \frac{h_0}{4} - \frac{r^4}{9} \Big{)} + r,
	\end{align*}
	which implies \eqref{want to show} since $0\leq r \leq \frac{h_0}{4}$.

	For the remaining part, we use the identity \eqref{id th} and show that \eqref{want to show} holds for $r\geq \frac{h_0}{4}$ if:
	\begin{align*}
		&\hspace{1cm} tanh(r) \bigg(1 + \frac{r^2}{3}\bigg) -r(1 +cr) >0 
		\\
		& \iff
		\bigg(1 + \frac{r^2}{3} - r(1+cr)\bigg) -\frac{2}{e^{2r}+1}\bigg(1 + \frac{r^2}{3}\bigg)   >0
		\\
		& \iff
		\bigg(1 + \frac{r^2}{3} - r(1+cr)\bigg)\big(e^{2r}+1\big) - 2\bigg(1 + \frac{r^2}{3}\bigg) : = G(r)   >0.
	\end{align*}
	But this holds since 
	\begin{align*}
		G'''(r) 
		& 
		=
		\frac{4}{3} e^{2r} \big( 2r^2  -3c(2r^2+6r+3) \big)
		\\
		& 
		\geq 	\frac{4}{3} e^{2r} 
		\Big{(} 
		r^2
		(1-6c)	
		+
		r
		(\frac{h_0}{8} - 18c)
		+
		(\frac{h_0^2}{32} - 9c)
		\Big{)},
	\end{align*}
	and is positive for $0<c<10^{-3}h_0^2$ with $r\geq \frac{h_0}{4}$. Indeed, as a consequence we have the following chain of implications
	\begin{align*}
		%&\hspace{1cm}0<  F'''(r)  = \frac{8}{3}r^2e^{2r}
		%\\
		&% \implies
		0< G''\Big(\frac{h_0}{2}\Big) \leq G''(r) = 
		\frac{2}{3} \big(
		e^{2r}
		(1-2r+2r^2 -3c(2r^2+4r+1))
		-3c -1 
		\big) 
		% \quad \text{increasing}
		\\
		\implies& 0< G'\Big(\frac{h_0}{2}\Big) \leq G'(r) = \frac{1}{3} \big(
		e^{2r}(3-4r+2r^2 -6cr(r+1))
		-
		2r-3-6cr
		\big) 
		% \quad \text{increasing}
		\\
		\implies& 0< G\Big(\frac{h_0}{2}\Big) \leq G(r).
		%  \quad \text{increasing}.
	\end{align*}
	We have therefore verified \eqref{want to show} for all $r\geq 0$ and we conclude that \eqref{K pos} holds true.\\

	Similarly,  for $0<\beta<\frac{1}{3}$, we have that \eqref{K pos beta} is a consequence of the inequality
	\begin{equation*}
		\frac{tanh(r)}{r} (1 + \beta r^2) \geq \beta + c  r.
	\end{equation*}
	One should note that we do not require sharp estimates. In fact, we simply need to obtain the estimate
	\begin{equation*}
		\bigg(1 + \beta r^2- r(\beta +cr)\bigg)\big(e^{2r}+1\big) - 2\bigg(1 + \beta r^2\bigg) =:  H(r) \geq 0,
	\end{equation*}
	for $r\geq 0$. On the other hand, we observe that
	\begin{equation*}
			H'''(r) 
		= 4 e^{2 r} 
		\Big(
		2 + 2\beta r (2 + r)
		-3c+2cr(3+r)
		\Big) >0
	\end{equation*}
for all $r\geq0$ if
\begin{align*}
 (2-3c) + 2r(2\beta - 3c) + 2r^2(\beta-c)>0
\end{align*}
and is ensured for $0<c\leq 10^{-3}\beta $. Consequently, 
	\begin{align*}
		&0< H''(0) \leq 
		H''(r)  
		= 2 e^{2 r} \Big(2+   \beta( 2r^2+2r-1) -  c(2r^2+4r+1)\Big) -2(\beta+c)
		\\
		&0<H'(0) \leq 
		H'(r)  
		= e^{2r}
		\Big(2+
		\beta ( 2r^2-1) - 2c(r^2+2r) 
		\Big)
		-\beta(2r+1) - 2cr
		\\
		&
		0< H(0) \leq H(r),
	\end{align*}
	and we argue as above to conclude. \\

	The proof of estimate \eqref{derivative Symbol} is a direct consequence of Lemma \ref{der T} 	and \eqref{derivative sqrt K} with $n=1$ if we trace the dependence in $\beta$:
	\begin{align*}
		\big{|} \frac{d}{dr}  \sqrt{K_1(r)}\big{|} 
		\lesssim
		\langle r\rangle^{-\frac{1}{2}-1}(1+\beta r^2)^{\frac{1}{2}} + \langle r\rangle^{-\frac{1}{2}} \frac{\beta r}{(1+\beta r^2)^{\frac{1}{2}}} \lesssim \langle r \rangle^{-1} + \sqrt{\beta} \langle r\rangle^{-\frac{1}{2}}.
	\end{align*}

	Estimate \eqref{Comparison}   concerns the  following bound on the difference:
	\begin{align*}
		\Big{|}
		\sqrt{K_{\mu}(\xi)} - \sqrt{\beta} \mu^{\frac{1}{4}}|\xi|^{\frac{1}{2}}
		\Big{|}
		=
		\Big| \Big( \frac{tanh(\sqrt{\mu} |\xi|)}{\sqrt{\mu} |\xi|} (1+ \beta \mu \xi^2 )\Big)^{\frac{1}{2}} 
		-
		 \sqrt{\beta} \mu^{\frac{1}{4}}|\xi|^{\frac{1}{2}}
		\Big|.
	\end{align*}
	For $\beta\sqrt{\mu} |\xi| \leq 1$ there holds trivially by using the triangle inequality that
	\begin{equation*}
		\Big{|}
		\sqrt{K_{\mu}(\xi)} - \sqrt{\beta} \mu^{\frac{1}{4}}|\xi|^{\frac{1}{2}}
		\Big{|}
		\lesssim 1.
	\end{equation*}
	While for $\beta \sqrt{\mu} |\xi| > 1$ we observe by direct calculations that
	\begin{align*}
		\Big{|}
		\sqrt{K_{\mu}(\xi)} - \sqrt{\beta} \mu^{\frac{1}{4}} |\xi|^{\frac{1}{2}}
		\Big{|}
		&
		=
		 \sqrt{\beta} \mu^{\frac{1}{4}} |\xi|^{\frac{1}{2}} \Big|\Big( \frac{tanh(\sqrt{\mu}|\xi|)}{\beta  \mu\xi^2}+tanh(\sqrt{\mu}|\xi|)\Big)^{\frac{1}{2}} - 1\Big|  
		 \\
		 & 
		 \lesssim
		 \sqrt{\beta} \mu^{\frac{1}{4}} |\xi|^{\frac{1}{2}} \Big|\frac{tanh(\sqrt{\mu}|\xi|)}{\beta  \mu\xi^2} + (tanh(\sqrt{\mu}|\xi|)-1)\Big|
		 \\
		 & 
		 \lesssim \frac{1}{(\beta \sqrt{\mu}|\xi|)^{\frac{1}{2}}} \frac{1}{\sqrt{\mu}|\xi|} + \sqrt{\beta}(\sqrt{\mu}|\xi|)^{\frac{1}{2}}e^{-2\sqrt{\mu}|\xi|}  
		 \\
		 &
		 \lesssim
		 \beta + \sqrt{\beta},
	\end{align*}
	 where we used the triangle inequality and that $\beta\sqrt{\mu}|\xi|>1$. 
	
	Lastly, we prove  \eqref{K point. for commutator} by using \eqref{Comparison}:
	\begin{align*}
		\sqrt{K_{\mu}(\xi)} \langle \xi \rangle^{s-1} |\xi| 
		& = 
		\bigg( \sqrt{K_{\mu}(\xi)}
		- \sqrt{\beta}\mu^{\frac{1}{4}} |\xi|^{\frac{1}{2}}
		\bigg)
		 \langle \xi \rangle^{s-1} |\xi| +\sqrt{\beta}\mu^{\frac{1}{4}}
		 \langle \xi \rangle^{s-1} |\xi|^{\frac{3}{2}} 
		 \\
		 &\lesssim ( \beta + \sqrt{\beta}) \langle \xi \rangle^{s} + \sqrt{\beta} \mu^{\frac{1}{4}} \langle \xi \rangle^{s} |\xi|^{\frac{1}{2}}.
	\end{align*}

\end{proof}

\begin{proof}[Proof of Lemma \ref{Pointwise est. on T}]
	To prove \eqref{Inverse T_mu}, since $T_1(\sqrt{\mu} \xi) = T_{\mu}(\xi)$, we only need to establish the following inequality:
	\begin{equation}\label{proof of inverse T_mu}
		1-\frac{h_0}{2}+cr \leq \frac{r}{tanh(r)}\lesssim 1+ r,
	\end{equation}
	for all $r>0$ and some $c>0$. We also note that the upper bound is trivial, so we only prove the lower bound. Let $h_0\in(0,1)$.  By the mean value theorem we find that $tanh(r) \leq r$ and observe %$tanh(r) = r - \tilde{R}(r)$ by Taylors formula and observe
	%
	%
	%
	%\begin{align*}
	%	\frac{r}{tanh(r)} = 1+ \frac{\tilde{R}(r)}{r-\tilde{R}(r)}>1+cr
	%\end{align*}
	%
	%
	%
	%by...
	%\begin{align*}
	%	(1-\frac{h_0}{2}+r)tanh(r)
	%	\leq 
	%	(1-\frac{h_0}{2})r +r^2
	%	\leq 
	%	r.
%	\end{align*}
	%
%	On the other hand, when $r\geq\frac{h_0}{2}$ we have that
	%
	%
	%
	\begin{equation*}
		\frac{r}{tanh(r)} = 	(1-\frac{h_0}{2})\frac{r}{tanh(r)} + 	\frac{h_0}{2}\frac{r}{tanh(r)} 
		\geq% \frac{1-\epsilon}{1-\epsilon}r+ \frac{\epsilon}{1-\epsilon} r.
		1-\frac{h_0}{2} + \frac{h_0}{2} r.
	\end{equation*}

	Next, we consider \eqref{T_mu J_mu equiv}. For $\sqrt{\mu}|\xi| \leq 1$ we have that $T_{\mu}(\xi) \sim 1$ and $\langle \sqrt{\mu} \xi \rangle \sim 1$. On the other hand, when $\sqrt{\mu}|\xi| \geq 1$ then $T_{\mu}(\xi) \sim \frac{1}{\sqrt{\mu}|\xi|}$ and $	\langle \sqrt{\mu} \xi \rangle \sim  \sqrt{\mu}|\xi|$. Multiplying the two functions, we obtain the desired result. 

	We estimate the derivative \eqref{Derivative J^sJ_mu} directly and using that $\mu\in (0,1)$:
	\begin{align*}
		\Big{|}
		\frac{d}{d\xi} \langle\xi\rangle^{s} \langle \sqrt{\mu}\xi \rangle^{\frac{1}{2}}
		\Big{|}
		\lesssim
		\langle\xi\rangle^{s-1} \langle \sqrt{\mu}\xi \rangle^{\frac{1}{2}}
		+
		\langle\xi\rangle^{s}  \langle \sqrt{\mu}\xi \rangle^{\frac{1}{2}} \sqrt{\mu}\langle \sqrt{\mu}\xi \rangle^{-1}
		\lesssim 
		\langle\xi\rangle^{s-1} \langle \sqrt{\mu}\xi \rangle^{\frac{1}{2}},
	\end{align*}
since $\sqrt{\mu}\langle \sqrt{\mu}\xi \rangle^{-1} \leq \langle \xi \rangle^{-1}$.

	Similarly, we have that \eqref{Derivative sqrtT Js J_mu} follows by the same argument after using \eqref{derivative T} and \eqref{Derivative J^sJ_mu}:
		\begin{align*}
		\Big{|}
		\frac{d}{d\xi} \sqrt{T_{\mu}(\xi)} \langle\xi\rangle^{s} \langle \sqrt{\mu}\xi \rangle^{\frac{1}{2}}
		\Big{|}
		&
		\lesssim
		\sqrt{\mu}
		\langle \xi \rangle^{-\frac{3}{2}}
		\langle\xi\rangle^{s} \langle \sqrt{\mu}\xi \rangle^{\frac{1}{2}}
		+
     	\langle \sqrt{\mu}\xi \rangle^{-\frac{1}{2}}	\langle\xi\rangle^{s-1} \langle \sqrt{\mu}\xi \rangle^{\frac{1}{2}}
	    \\
		&
		\lesssim 
		\langle\xi\rangle^{s-1}.
	\end{align*}

For estimate \eqref{Comparison J_mu D_mu}, we observe that
\begin{align*}
	\langle\sqrt{\mu} \xi\rangle^{\frac{1}{2}}- \mu^{\frac{1}{4}}|\xi|^{\frac{1}{2}} 
	& 
	=
	  \mu^{\frac{1}{4}} |\xi |^{\frac{1}{2}} 
	  \Big{(}
	  (\frac{1}{\mu \xi^2} + 1)^{\frac{1}{4}} - 1 
	  \Big{)}
	  \lesssim 1.
\end{align*}

\end{proof}

\subsection{Proof of Lemmas \ref{Commutator L2} and \ref{T_mu and J_mu Commutator L2}}\label{A3}

For the proof of Lemma \ref{Commutator L2}  and Lemma \ref{T_mu and J_mu Commutator L2}, we need a \lq\lq generalized\rq\rq\:version of the Kato-Ponce commutator estimate which holds for symbols defined by: 
\begin{Def}[Symbol class \cite{Lannes_Book2013} Def. $B.7$]\label{Symbol class} We say that a symbol $\sigma(D)$ is a member of the symbol class $\mathcal{S}^s$ with $s\in \mathbb{R}$, if $\xi \mapsto \sigma(\xi) \in \mathbb{C}$ is smooth and satisfies
	\begin{equation*}
		\forall \alpha \in \mathbb{N}, \quad \sup \limits_{\xi \in  \mathbb{R}} \langle \xi \rangle^{\alpha - s} \Big{|} \frac{d^{\alpha}}{d\xi^{\alpha}} \sigma(\xi)\Big{|} <\infty.
	\end{equation*}
	One also associates the following seminorm on $\mathcal{S}^s:$
	\begin{equation}\label{seminorm}
		\mathcal{N}^s(\sigma) =  \sup\limits_{\alpha \in \mathbb{N}, \: \: \alpha \leq 4} \:\:\: \sup\limits_{\xi \in \mathbb{R}} \langle \xi \rangle^{\alpha - s} \Big{|} \frac{d^{\alpha}}{d\xi^{\alpha}} \sigma(\xi)\Big{|} .
	\end{equation}
\end{Def}
\noindent
The following result is found in Appendix B of \cite{Lannes_Book2013}.
\begin{lemma}\label{SymbolLannes} Let $t_0>1/2$, $s\geq 0$ and $\sigma \in \mathcal{S}^s$.
	%\begin{itemize}
	%\item If $0\leq s\leq t_0 +1$ and $f \in H^{t_0+1}$ then, for $g\in H^{s-1}$, one has
	%
	%
	%
	%\begin{equation}\label{symbol T}
	%	\| [\sigma, f]  g \|_{L^2}\lesssim \mathcal{N}^s(\sigma) \|f\|_{H^{t_0+1}} \|g\|_{H^{s-1}}.
	%\end{equation}
	%\item
	If $f \in H^s \cap H^{t_0+1}(\mathbb{R})$, then for all $g \in H^{s-1}(\mathbb{R})$,
	\begin{equation}\label{SymbolCommutator}
		\| [\sigma(D), f]  g \|_{L^2}\lesssim \mathcal{N}^s(\sigma)  \|f\|_{H^{\max\{t_0+1,s\}}} \|g\|_{H^{s-1}} 
	\end{equation}
	%	\end{itemize}
\end{lemma}
\noindent
With this at hand, we may give the proof.

\begin{proof}[Proof of Lemma \ref{Commutator L2}]
	
	To prove \eqref{Product chi K} and \eqref{Commutator chi K}, it suffices to verify for all $n \in \N$ that
	\begin{equation}\label{Goal 1}
		\sup  \limits_{\xi \in  \mathbb{R}} \:  \langle \xi 	\rangle^{n } \: 
		\Big{|} \frac{d^{n}}{d\xi^{n}} 
		(\chi^{(1)}_{\mu} \sqrt{\mathcal{K}_{\mu}}) (\xi)\Big{|} 
		\lesssim_{\beta} 1,
	\end{equation}
	for any  $0<\mu<1$. Indeed, in agreement with Definition \ref{Symbol class}, then $(\chi^{(1)}_{\mu} \sqrt{\mathcal{K}_{\mu}}) \in \mathcal{S}^0$ and \eqref{Commutator chi K} holds true due to Lemma \ref{SymbolLannes}. Moreover,  using Plancherel and \eqref{Goal 1} with $n=0$ we have
	\begin{align*}
		\|(\chi^{(1)}_{\mu} \sqrt{\mathcal{K}_{\mu}})(D) f\|_{L^2} \lesssim_{\beta} \| f\|_{L^2},
	\end{align*}
	proving \eqref{Product chi K}. Now, let us prove \eqref{Goal 1}.  We observe that
	\begin{equation}\label{est on supp}
		 \mu^{\frac{k}{2}}\langle \xi \rangle^k
		 \Big{|}( \frac{d^{k}}{d\xi^{k}}   \chi^{(1)})(\sqrt{\mu}\xi)\Big{|} \lesssim 1, \quad k\geq 0,
	\end{equation}
	since $\sqrt{\mu} |\xi| \lesssim 1$ on the support of $\chi^{(1)}_{\mu}(\xi)$. Moreover, we observe by Lemma \ref{der T} and $\mu\in (0,1)$ that
	\begin{align*}
		\chi^{(1)}_{\mu}(\xi)\: \Big{|} \frac{d^{k}}{d\xi^{k}}  \sqrt{K_{\mu}(\xi)} \Big{|}
		\lesssim_{\beta}
		 \chi^{(1)}_{\mu}(\xi)\langle \sqrt{\mu} \xi \rangle^{\frac{1}{2} }\mu^{\frac{k}{2}}\langle \sqrt{\mu} \xi \rangle^{ -k}
		 \lesssim_{\beta}
		\langle \xi\rangle^{-k}.
	\end{align*}
	Combining these estimates with the Leibniz rule yields
	\begin{align*}
		\langle \xi \rangle^n \Big{|} \frac{d^{n}}{d\xi^{n}}  \big{(}	\chi^{(1)}_{\mu}(\xi)\sqrt{K_{\mu}(\xi)} \big{)} \Big{|}
		&
		\lesssim
		\langle \xi \rangle^n \sum_{k=0}^n \Big{|}
		 \frac{d^{n-k}}{d\xi^{n-k}} 
		 \big{(}\chi^{(1)}(\sqrt{\mu}\xi)\big{)} \: \frac{d^{k}}{d\xi^{k}} \big{(}\sqrt{K_{\mu}(\xi)}\big{)}\Big{|}
		\\
		& 
		\lesssim_{\beta} 
		\langle \xi \rangle^n
		\sum_{k=0}^n 
		\mu^{\frac{n-k}{2}} \big{|}(  \frac{d^{n-k}}{d\xi^{n-k}} 
		\chi^{(1)})(\sqrt{\mu}\xi)\big{|}
		\langle  \xi \rangle^{-k}
		\\
		&
		\lesssim_{\beta} 
		\sum_{k=0}^n 
		\mu^{\frac{n-k}{2}}  \langle \xi \rangle^{n-k}
		\big{|}
		( \frac{d^{n-k}}{d\xi^{n-k}} \chi^{(1)})(\sqrt{\mu}\xi)
		\big{|}
		\lesssim_{\beta} 1.
	\end{align*}
	Hence, $(\chi^{(1)}_{\mu} \sqrt{\mathcal{K}_{\mu}}) \in \mathcal{S}^0$ and $\mathcal{N}^0(\chi^{(1)}_{\mu} \sqrt{\mathcal{K}_{\mu}}) \lesssim_{\beta} 1$ independently from $\mu$, proves \eqref{Goal 1}.

	Next,  we consider estimates \eqref{Product sigma 1/2} and \eqref{Commutator D}. Recalling \eqref{Sigma 1/2} we define
	\begin{align}\label{Sigma 1/2 tilde}
		\tilde{\sigma}_{\mu, \frac{1}{2}}(\xi) 
		=
		 \mu^{-\frac{1}{4}} \sigma_{\mu, \frac{1}{2}}(\xi) = \frac{1}{\mu^{\frac{1}{4}}} \: \frac{1}{\mu^{\frac{1}{4}} |\xi|^{\frac{1}{2}}} \: (1+ \mu \beta \xi^2)^{\frac{1}{2}}. %=: \frac{1}{\mu^{\frac{1}{4}}} a_{\mu}(\xi) b_{\mu}(\xi).
	\end{align}
	Then, it suffices to prove that 
	\begin{equation}\label{Goal 2}
		\sup  \limits_{\xi \in  \mathbb{R}} \:	\langle \xi \rangle^{n-\frac{1}{2}} \Big{|}   \frac{d^{n}}{d\xi^{n}}  (	\chi^{(2)}_{\mu}\tilde\sigma_{\mu,\frac{1}{2}})(\xi)  \Big{|} \lesssim_{\beta} 1,
	\end{equation}
	for all $n \in \N$ and any $\mu \in (0,1)$. Indeed, if we assume \eqref{Goal 2} and take $n=0$ we deduce from Plancherel's identity that
	\begin{equation*}
		\|	(\chi^{(2)}_{\mu}\sigma_{\mu,\frac{1}{2}})(D) f\|_{L^2} \lesssim_{\beta} \mu^{\frac{1}{4}} \|f\|_{H^{\frac{1}{2}}}  \lesssim_{\beta} \|f\|_{L^2} + \mu^{\frac{1}{4}} \|D^{\frac{1}{2}}f\|_{L^2}.
	\end{equation*}
	which proves \eqref{Product sigma 1/2}. Moreover, \eqref{Goal 2} also implies that $(\chi^{(2)}_{\mu} \tilde{\sigma}_{\mu, \frac{1}{2}}) \in \mathcal{S}^{\frac{1}{2}}$ with $\mathcal{N}^{\frac{1}{2}}(\chi^{(2)}_{\mu} \tilde{\sigma}_{\mu, \frac{1}{2}}) \lesssim_{\beta} 1$ so that
	\begin{equation*}
		\| [(\chi^{(2)}_{\mu} {\sigma}_{\mu, \frac{1}{2}})(D) ,f] \partial_x g\|_{L^2} \lesssim_{\beta} \mu^{\frac{1}{4}}\|f\|_{H^s}  \|g\|_{H^{\frac{1}{2}}},
	\end{equation*}
	by Lemma \ref{SymbolLannes}. Now we prove \eqref{Goal 2}. First, we consider the functions,
	\begin{align*}
		a_{\mu}(\xi) = \frac{1}{\mu^{\frac{1}{4}} |\xi|^{\frac{1}{2}}} \quad \text{and} \quad b_{\mu} (\xi) =(1+ \mu \beta \xi^2)^{\frac{1}{2}}.
	\end{align*}
	Then, since $|\xi|> \sqrt{\mu} |\xi| >1$ on the support of $\chi^{(2)}_{\mu}$, we observe that 
	\begin{align}\label{est a}
		\chi^{(2)}_{\mu}(\xi)\Big{|}\frac{d^{k}}{d\xi^{k}} a_{\mu}(\xi) \Big{|} 
		%& \lesssim \chi^{(2)}_{\mu}(\xi)
	%	\frac{1}{(\mu|\xi|)^{\frac{1}{2}}} \frac{1}{|\xi|^k}
	%	\\ 
		\lesssim	\mu^{\frac{1}{4}} \chi^{(2)}_{\mu}(\xi)
		\frac{1}{\sqrt{\mu}|\xi|} \frac{1}{|\xi|^{k-\frac{1}{2}}}
		 \lesssim	\mu^{\frac{1}{4}}
		\langle \xi \rangle^{\frac{1}{2}-k} \langle \sqrt{\mu} \xi \rangle^{-1}.
	\end{align}
	While $b_{\mu}(\xi) \lesssim_{\beta} \langle \sqrt{\mu} \xi\rangle $ and its deriatives satisfy the bound,
	\begin{align}\label{est b}
		\chi^{(2)}_{\mu}(\xi)\Big{|}  \frac{d^{k}}{d\xi^{k}}  b_{\mu}(\xi) \Big{|} \lesssim_{\beta}\chi^{(2)}_{\mu}(\xi) \mu^{\frac{k}{2}} \langle\sqrt{ \mu }\xi \rangle^{1-k}
		\lesssim_{\beta} \langle\sqrt{ \mu }\xi \rangle \langle \xi \rangle^{-k}.
	\end{align}
	Thus, if all derivatives falls on $\tilde{\sigma}_{\mu, \frac{1}{2}}$, the Leibniz rule, \eqref{est a} and \eqref{est b} imply
	\begin{align*}
		\chi^{(2)}_{\mu}(\xi)\Big{|} \frac{d^{k}}{d\xi^{k}}  	\tilde \sigma_{\mu,\frac{1}{2}}(\xi) \Big{|}
		%&
		\lesssim
		\mu^{-\frac{1}{4}}\chi^{(2)}_{\mu}(\xi)
		\sum_{j=0}^k
		\Big{|}\frac{d^{k-j}}{d\xi^{k-j}} 
	a_{\mu}(\xi)\Big{|}
		 \: \Big{|}\frac{d^{j}}{d\xi^{j}} 	 b_{\mu}(\xi) \Big{|}
		%\\
		%& 
		%\lesssim_{\beta} 
		%\sum_{j=0}^k
		%\langle \xi \rangle^{\frac{1}{2}-(k-j)}
		%\langle \xi \rangle^{-j}
		%\\
		%&
		\lesssim_{\beta} \langle \xi \rangle^{\frac{1}{2} - k}.
	\end{align*}
	On the other hand, when derivatives fall the cut-off function, we observe 
	\begin{equation}\label{est on supp 2}
		\mu^{\frac{k}{2}}\langle \xi \rangle^k\Big{|}( \frac{d^{k}}{d\xi^{k}}  \chi^{(2)})(\sqrt{\mu}\xi) \Big{|} \lesssim 1, \quad k\geq 1,
	\end{equation}
	since the support of $ \frac{d^k}{d\xi^k}  \chi^{(2)}_{\mu}$ is contained in the support of $\chi^{(1)}_{\mu}$. As a result, there holds
	\begin{align*}
	\langle \xi \rangle^{n-\frac{1}{2}}  \Big{|} \frac{d^{n}}{d\xi^{n}}  \big{(}	\chi^{(2)}_{\mu}(\xi)\tilde\sigma_{\mu,\frac{1}{2}}(\xi) \big{)}  \Big{|}
	&
	\lesssim
	\langle \xi \rangle^{n - \frac{1}{2}} \sum_{k=0}^n 
	\Big{|}
	\frac{d^{n-k}}{d\xi^{n-k}} \big{(}\chi^{(2)}(\sqrt{\mu}\xi)\big{)}
	\Big{|}
	\:
	\Big{|}
	\frac{d^{k}}{d\xi^{k}} \big{(} \tilde\sigma_{\mu,\frac{1}{2}}(\xi)  \big{)}
	\Big{|}
	\\
	& 
	\lesssim_{\beta} 
	\langle \xi \rangle^{n-\frac{1}{2}}
	\sum_{k=0}^n 
	\mu^{\frac{n-k}{2}} \big{|} (\frac{d^{n-k}}{d\xi^{n-k}} \chi^{(2)})(\sqrt{\mu}\xi) \big{|}
	\: \langle \xi \rangle^{\frac{1}{2}-k}
	\\
	&
	\lesssim_{\beta} 
	\sum_{k=0}^n 
	\mu^{\frac{n-k}{2}}  \langle \xi  \rangle^{n-k} \big{|}(\frac{d^{n-k}}{d\xi^{n-k}} \chi^{(2)})(\sqrt{\mu}\xi) \big{|}
	\\
	&\lesssim_{\beta} 1.
\end{align*}
The estimate is uniform in $\mu \in (0,1)$, and \eqref{Goal 2} is proved, which provides the desired result. 
	
	Lastly, we prove \eqref{Product sigma} and \eqref{Commutator sigma} arguing in the same vein. First, we write:
	\begin{align*}
		\chi^{(2)}_{\mu}(\xi) \sigma_{\mu,0}(\xi)
	%	&
	%	=
		%\chi^{(2)}_{\mu}(\xi) 
	%	\bigg( 
	%	\frac{1}{\sqrt{\mu}|\xi|} + \sqrt{\mu} \beta |\xi| 
	%	-
	%	\frac{tanh(\sqrt{\mu}|\xi|)}{\sqrt{\mu} |\xi|} (1+ \beta \mu \xi^2)			
	%	\bigg)^{\frac{1}{2}}
	%	\\
		& =
		\chi^{(2)}_{\mu}(\xi) \cdot  \: \frac{1}{\mu^{\frac{1}{4}} |\xi|^{\frac{1}{2}}} \: 
		\cdot	\Big(
		1+\mu \beta \xi^2
		\Big)^{\frac{1}{2}}\cdot	\Big( \frac{2}{e^{2\sqrt{\mu}|\xi|} + 1}	\Big)^{\frac{1}{2}}
		\\
		& = : \chi^{(2)}_{\mu}(\xi) \: a_{\mu}(\xi) \:  b_{ \mu}(\xi)\:  c_{\mu}(\xi),
	\end{align*}
	making use of the identity \eqref{id th}. Then we observe for all $N \in \N$ that
	\begin{align}\label{est c}
		\langle \xi \rangle^k \langle \sqrt{\mu} \xi \rangle^N\Big{|} \frac{d^{k}}{d\xi^{k}} c_{\mu}(\xi) \Big{|} 
		& \lesssim
		\mu^{\frac{k}{2}}\langle \xi \rangle^k  \langle \sqrt{\mu} \xi \rangle^N e^{-\sqrt{\mu} |\xi|} \lesssim 1.
	\end{align}
	As a result, we deduce by \eqref{est b} and \eqref{est c} with $N=1$ that
	\begin{align*}
		\chi^{(2)}_{\mu}(\xi)
		\Big{|} \frac{d^{k}}{d\xi^{k}}  (b_{\mu}(\xi)\: c_{\mu}(\xi) ) \Big{|}
		& \lesssim
		\chi^{(2)}_{\mu}(\xi)
		\sum_{j=0}^k\big{|} \frac{d^{k-j}}{d\xi^{k-j}}  \big{(} b_{\mu}(\xi) \big{)}\big{|} \: \big{|}\frac{d^{j}}{d\xi^{j}} \big{(}c_{\mu}(\xi)\big{)}\big{|}
		\\
		& \lesssim_{\beta}
		\sum_{j=0}^k 
		\langle \xi \rangle^{-(k-j)} \langle \sqrt{\mu}\xi \rangle 
		\: \langle \xi \rangle^{-j} \langle \sqrt{\mu}\xi \rangle^{-1}
		\\
		&	
		\lesssim_{\beta}\langle\xi\rangle^{-k}.
	\end{align*}
	Moreover, we use  \eqref{est a} to deduce
	\begin{align*}
		\chi^{(2)}_{\mu}(\xi)	\Big{|} \frac{d^{k}}{d\xi^{k}}  \sigma_{\mu,0}(\xi) \Big{|}
		\lesssim \chi^{(2)}_{\mu}(\xi)
		\sum_{j=0}^k\Big{|} \frac{d^{k-j}}{d\xi^{k-j}} \big{(} a_{\mu}(\xi) \big{)} \Big{|}\: \Big{|}\frac{d^{j}}{d\xi^{j}}  \big{(}b_{\mu,\beta}(\xi) \: c_{\mu}(\xi)\big{)} \Big{|}\lesssim_{\beta} \langle\xi\rangle^{-k},
	\end{align*}
	from which we find
	\begin{align*}
		\langle \xi \rangle^n 	\Big{|}\frac{d^{n}}{d\xi^{n}} \big{(}	\chi^{(2)}_{\mu}(\xi)\sigma_{\mu,0}(\xi) \big{)}	\Big{|}
		\lesssim 
		\langle \xi \rangle^n \sum_{k=0}^n 
		\Big{|}\frac{d^{n-k}}{d\xi^{n-k}} \big{(}\chi^{(2)}_{\mu}(\xi)\big{)}	\Big{|} \: 	\Big{|}\frac{d^{k}}{d\xi^{k}} \big{(} \sigma_{\mu,0}(\xi)  \big{)}	\Big{|}
		\lesssim_{\beta} 
		1,
	\end{align*}
	by  \eqref{est on supp 2}. Arguing as above, we may conclude that the estimates \eqref{Product sigma} and \eqref{Commutator sigma} holds. 
\end{proof}

\begin{proof}[Proof of Lemma \ref{T_mu and J_mu Commutator L2}] In order to prove
	\eqref{L2 Commutator T_mu J_mu}, we simply verify that $\sqrt{\mathcal{T}_{\mu}}J^{\frac{1}{2}}_{\mu} \in \mathcal{S}^0$ and $\mathcal{N}^0(\sqrt{\mathcal{T}_{\mu}}J^{\frac{1}{2}}_{\mu} ) \lesssim 1$ uniformly in $\mu\in (0,1)$. But this is a direct consequence of Lemma \ref{der T} and the Leibniz rule:
	\begin{align*}
		\langle \xi \rangle^n 
		\Big{|}
			\frac{d^n}{d\xi^n} \sqrt{T_{\mu}(\xi)}J^{\frac{1}{2}}_{\mu}
		\Big{|}
		& \lesssim 
		\langle \xi \rangle^n \sum_{k=0}^n
		 \Big{|}\frac{d^{n-k}}{d\xi^{n-k}}\sqrt{T_{\mu}(\xi)}
		 \Big{|}
		 \: 
		 \Big{|}
		 \frac{d^{k}}{d\xi^{k}} \langle \sqrt{\mu} \xi \rangle^{\frac{1}{2}}
		 \Big{|}
		 \\
		 & \lesssim 
		 \langle \xi \rangle^n \sum_{k=0}^n
		 \mu^{\frac{n-k}{2}} \langle \sqrt{\mu} \xi \rangle^{-\frac{1}{2}-(n-k)} \mu^{\frac{k}{2}}\langle \sqrt{\mu} \xi \rangle^{\frac{1}{2} - k }
		 \\
		 & \lesssim 
		 \langle \xi \rangle^n \mu^{\frac{n}{2}} \langle \sqrt{\mu}\xi \rangle^{-n},
	\end{align*}
	and is bounded by a constant independent from $\mu \in(0,1)$. Hence, we may conclude by Lemma \ref{SymbolLannes} that 	\eqref{L2 Commutator T_mu J_mu} holds true.

	A similar approach is used for the proof of \eqref{Bessel Commutator T_mu J}. Indeed, we observe that
	\begin{align*}%\label{above}
		\langle \xi \rangle^{n-s}
		\Big{|}
		\frac{d^n}{d\xi^n} \sqrt{T_{\mu}(\xi)}J^{s}
		\Big{|}
	    \lesssim 
		\langle \xi \rangle^{n-s}\sum_{k=0}^n
		\mu^{\frac{n-k}{2}} \langle \sqrt{\mu} \xi \rangle^{-\frac{1}{2}-(n-k)}  \langle   \xi \rangle^{s - k }
		\lesssim 
		1.
	\end{align*}
	Hence, $\sqrt{\mathcal{T}_{\mu}}J^s \in \mathcal{S}^s$ and $\mathcal{N}^s(\sqrt{\mathcal{T}_{\mu}}J^s ) \lesssim 1$ uniformly in $\mu\in (0,1)$, allowing us to conclude by Lemma \ref{SymbolLannes}.
	
	The proof of \eqref{Bessel Commutator T_mu J=1} is the same, by a direct application of \eqref{derivative T} we deduce that $\sqrt{\mathcal{T}_{\mu}} \in \mathcal{S}^0$ uniformly in $\mu\in (0,1)$. 
	
	Next, we consider \eqref{Commutator dx T}. Define the bilinear form: $a_1(D)(f,g) = \partial_x[\sqrt{\mathcal{T}_{\mu}}(D), f]g$. Then we may use Plancherel to write
	\begin{align*}
		|\hat{a}_1(\xi)(f,g)|
		& \leq
		\int_{\R}
		|\xi|
		\Big{|}
		\sqrt{T_{\mu}(\xi)}  - \sqrt{T_{\mu}(\rho)} 
		\Big{|}
		|\hat{f}(\xi-\rho)| \: | \hat{ g}(\rho)| \: d\rho.
	\end{align*}
	Clearly, if we can prove that 
	\begin{equation}\label{need to show: est on b}
		b_1(\xi,\rho) : =  |\xi|
		\Big{|}
		\sqrt{T_{\mu}(\xi)}  - \sqrt{T_{\mu}(\rho)} 
		\Big{|} \lesssim 1 + |\xi - \rho|,
	\end{equation}
	then we can conclude as we did for the proof of Lemma \ref{Commutator K at level s}. Indeed, assuming the claim \eqref{need to show: est on b}, then there holds
	\begin{align*}
		\|\partial_x[\sqrt{\mathcal{T}_{\mu}}(D), f]g \|_{L^2}  =
		\|	\hat{a}_1(\xi) (f,g)\|_{L^2} \lesssim (\|f\|_{H^{t_0}}+\|\partial_xf\|_{H^{t_0}})\|g\|_{L^2}.
	\end{align*}
	Now, in order to estimate $b_1(\xi,\rho)$ we consider three cases. First, if $|\rho| \leq 1$ then we have by the triangle inequality,
	\begin{equation*}
		b_1(\xi,\rho) \leq  (1+|\xi-\rho|)
		\Big{(}
		\sqrt{T_{\mu}(\xi)}  + \sqrt{T_{\mu}(\rho)} 
		\Big{)} \lesssim 1 + |\xi - \rho|,
	\end{equation*}
	since $\xi \mapsto \sqrt{T_{\mu}(\xi)}$ is bounded by one. Secondly, consider the region where $|\rho|>1$ and $|\xi| \geq |\rho|$. Then since $\xi \mapsto tanh(\sqrt{\mu}|\xi|)$ is increasing  and $\xi \mapsto T_{\mu}(\xi)$ is decreasing, we have that
	\begin{equation*}
		\frac{|\rho|}{|\xi|} 
		\leq
		\bigg{(} \frac{|\rho|}{|\xi|} \bigg{)}^{\frac{1}{2}} 
		\leq 
		% \frac{T_{\mu}(\xi)}{T_{\mu}(\eta)}\leq
		\bigg{(} \frac{T_{\mu}(\xi)}{T_{\mu}(\rho)}\bigg{)}^{\frac{1}{2}}\leq 1.
	\end{equation*}
	%
	%
	%
	%also using that $\xi \mapsto \sqrt{T_{\mu}(\xi)}$ is decreasing. 
	Thus, there holds
	\begin{align*}
		b_1(\xi,\rho) = |\xi|  	\Big{(}
		1-\bigg{(} \frac{T_{\mu}(\xi)}{T_{\mu}(\rho)}\bigg{)}^{\frac{1}{2}}
		\Big{)} \sqrt{T_{\mu}(\rho)} \leq |\xi| - |\rho| \leq|\xi- \rho|.
	\end{align*}
	For $|\rho|>1$ and $|\xi|<|\rho|$ we use a similar argument to find,
	\begin{align*}
		b_1(\xi,\rho)
		=
		|\xi|  	\Big{(}
		1-\bigg{(} \frac{T_{\mu}(\rho)}{T_{\mu}(\xi)}\bigg{)}^{\frac{1}{2}}
		\Big{)} \sqrt{T_{\mu}(\xi)} \leq \frac{|\xi|}{|\rho|}(|\rho|- |\xi|) \leq |\xi - \rho|.
	\end{align*}

Finally, we estimate \eqref{L2 commutator J_mu} using a similar approach. We define the bilinear form $a_2(D)(f,g) = [J_{\mu}^{\frac{1}{2}},f]\partial_x g$ and look in frequency:
\begin{equation*}
	|\hat{a}_2(\xi)(f,g)|
	\leq  
	\int_{\R} 
	\Big{|}
	\langle\sqrt{\mu} \xi \rangle^{\frac{1}{2}} - \langle \sqrt{\mu} \rho \rangle^{\frac{1}{2}}
	\Big{|}
	|\hat{f}(\xi-\rho)| \: | \widehat{\partial_x g}(\rho) |\: d\rho.
\end{equation*}
Then by the same argument as above, we only need to prove that
\begin{equation}\label{Estimate b(xi,eta)}
	b_2(\xi,\eta) = 	\Big{|}
	\langle\sqrt{\mu} \xi \rangle^{\frac{1}{2}} - \langle \sqrt{\mu} \rho \rangle^{\frac{1}{2}}
	\Big{|} \frac{|\rho|}{\langle \sqrt{\mu} \rho \rangle^{\frac{1}{2}}} \lesssim 1+ |\xi-\rho|.
\end{equation}
We consider three cases. If $|\rho| \leq 1$ then since $\mu \in (0,1)$, there holds by the triangle inequality:
\begin{align*}
	b_2(\xi,\rho) \lesssim 1+ \langle \xi-\rho \rangle^{\frac{1}{2}}.
\end{align*}
In the case $|\xi| \geq |\rho|>1$,  observe that
\begin{equation}\label{Obersvation 1}
	\frac{1 + \mu \rho^2}{1+\mu \xi^2} \leq
	 \frac{(1+ \mu \rho^2)^{\frac{1}{4}}}{(1+\mu\xi^2)^{\frac{1}{4}}} \leq 1,
\end{equation}
and we have
\begin{equation}\label{observation 2}
 \frac{\xi^2 - \rho^2}{|\xi-\rho|\:|\xi|} \leq 	\frac{|\xi|+|\rho|}{|\xi|} \lesssim 1.
\end{equation}
As a consequence, recalling $\mu\in(0,1)$, we have that
\begin{align*}
	b_2(\xi,\rho) 
	& \leq
	\Big{(}
	1- \frac{1+\mu\rho^2}{1+\mu \xi^2}
	\Big{)} 
	\frac{\langle \sqrt{\mu} \xi \rangle^{\frac{1}{2}}}{\langle \sqrt{\mu} \rho \rangle^{\frac{1}{2}}}|\rho|
	\\ 
	&
	\leq  \frac{\mu(\xi^2 - \rho^2)}{\mu^{\frac{1}{4}}(1+\mu\xi^2)^{\frac{3}{4}+\frac{1}{4}}}  	\frac{(1+  \mu\xi^2 )^{\frac{1}{4}}}{\langle  \rho \rangle^{\frac{1}{2}}}|\rho|
	\\
	& 
	\leq \frac{\xi^2 - \rho^2}{ |\xi|}   \frac{|\rho|}{|\xi|^{\frac{1}{2}}\langle  \rho \rangle^{\frac{1}{2}}} 
	\\
	& 
	\lesssim |\xi - \rho|.
\end{align*}
Lastly, in the case $|\rho| >1$ and $|\xi| < |\rho|$,  we can simply change the role of $\xi$ and $\rho$  in \eqref{Obersvation 1} and \eqref{observation 2}. As result, we get
\begin{align*}
	b_2(\xi,\rho) & \leq
	 \Big{(}
	 1- \frac{1+\mu\xi^2}{1+\mu \rho^2}
	 \Big{)}|\rho|
	 \leq  
	 \frac{\rho^2 - \xi^2}{|\rho|}  
	 \lesssim
	|\xi-\rho|.
\end{align*}
We may therefore conclude that \eqref{Estimate b(xi,eta)} holds and the estimate \eqref{L2 commutator J_mu} follows.
\end{proof}

\section*{Acknowledgements}
\noindent
This research was supported by a Trond Mohn Foundation grant. I also thank my advisor, Didier Pilod, for many long and helpful mathematical discussions, Henrik Kalisch for providing references and Vincent Duch\^ene for some important comments on the introduction.

%% file: References/ref.tex
%\chapter*{References}